\documentclass[12pt]{amsproc}
\pdfoutput=1

\usepackage[T1]{fontenc}
\usepackage[utf8]{inputenc}
\usepackage[centertags]{amsmath}
\usepackage{amsfonts}
\usepackage{amsthm}
\usepackage{amssymb}
\usepackage{enumitem}
\usepackage{hyperref}
\usepackage[english]{babel}
\usepackage[textsize=tiny]{todonotes}
\usepackage[all]{xy}
\usepackage{tikz}
\usetikzlibrary{patterns}

\definecolor{skin}{HTML}{FFECC9}
\definecolor{pumpkin}{HTML}{FEDFA9}
\definecolor{piggy}{HTML}{FFB99D}
\definecolor{fiolet}{HTML}{CD8F9C}
\definecolor{granat}{HTML}{677081}
\definecolor{ciemnyblekit}{HTML}{91A1B8}

\definecolor{oliwkowy}{HTML}{627037}
\definecolor{ciemnazielen}{HTML}{394D2E}
\definecolor{ciemnyfiolet}{HTML}{424444}
\definecolor{mocnyfiolet}{HTML}{717299}
\definecolor{jasnyfiolet}{HTML}{B0ABCC}
\definecolor{bladyfiolet}{HTML}{C9C7DB}

\theoremstyle{plain}
\newtheorem{thm}[equation]{Theorem}
\newtheorem{cor}[equation]{Corollary}
\newtheorem{lemma}[equation]{Lemma}
\newtheorem{prop}[equation]{Proposition}

\theoremstyle{remark}

\newtheorem{rem}[equation]{Remark}
\newtheorem{ex}[equation]{Example}

\theoremstyle{definition}
\newtheorem{assumpt}[equation]{Assumptions}

\newenvironment{pf}{\begin{proof}}{\end{proof}}

\numberwithin{equation}{section}

\DeclareMathOperator{\Pic}{Pic}
\DeclareMathOperator{\Aut}{Aut}
\DeclareMathOperator{\Spec}{Spec}

\DeclareMathOperator{\Hom}{Hom}
\DeclareMathOperator{\HH}{H}

\DeclareMathOperator{\conv}{conv}

\newcommand{\set}[1]{\left\{#1\right\}}
\newcommand{\fromto}[2]{#1, \dotsc, #2}
\newcommand{\setfromto}[2]{\set{\fromto{#1}{#2}}}

\DeclareMathOperator{\id}{Id}

\newcommand\PP{{\mathbb{P}}}
\newcommand\RR{{\mathbb{R}}}
\newcommand\QQ{{\mathbb{Q}}}
\newcommand\ZZ{{\mathbb{Z}}}
\newcommand\CC{{\mathbb{C}}}

\newcommand\AAA{{\mathbb{A}}}
\newcommand\HHH{{\mathbb{H}}}

\newcommand\FF{{\mathbb{F}}}

\newcommand\cO{{\mathcal O}}

\newcommand\cM{{\mathcal M}}
\newcommand\cC{{\mathcal C}}

\newcommand\cJ{{\mathcal J}}
\newcommand\cR{{\mathcal R}}

\newcommand\cQ{{\mathcal Q}}

\newcommand\m{{\mathfrak m}}
\newcommand\g{{\mathfrak g}}
\newcommand\h{{\mathfrak h}}

\newcommand\ra{{\ \rightarrow\ }}
\newcommand\lra{\longrightarrow}

\newcommand\iso{{\ \cong\ }}

\newcommand{\Wedge}[1]{{\textstyle{\bigwedge\nolimits}^{\! #1}}}

\begin{document}

\title{Algebraic torus actions on contact manifolds}

\author[J.~Buczy{\'n}ski]{Jaros{\l}aw Buczy{\'n}ski}
\address{J.~Buczy\'nski,
Institute of Mathematics of the Polish Academy of Sciences,
  ul.~\'Sniadeckich 8,
  00-656 Warszawa, Poland,
and
Institute of Mathematics,
University of Warsaw,
ul.~Banacha 2,
02-097 Warszawa,
Poland
 }
 \email{jabu@mimuw.edu.pl}

\author[J.~A.~Wi\'sniewski]{Jaros\l{}aw A. Wi\'sniewski}
\address{J.~A.~Wi\'sniewski,
Institute of Mathematics,
University of Warsaw,
ul.~Banacha 2,
02-097 Warszawa,
Poland
}
 \email{jarekw@mimuw.edu.pl}

\author[A.~Weber]{with Appendix by Andrzej Weber}
\address{A.~Weber,
Institute of Mathematics,
University of Warsaw,
ul.~Banacha 2,
02-097 Warszawa,
Poland
}
 \email{aweber@mimuw.edu.pl}

\thanks{The authors are supported by the Polish National Science
  Center (NCN) project Algebraic Geometry: Varieties and Structures,
  2013/08/A/ST1/00804. Buczy{\'n}ski is also supported 
  by a scholarship of Polish Ministry of Science and by the NCN project Complex contact manifolds and geometry of secants, 2017/26/E/ST1/00231.
  Wi\'sniewski and Weber are supported by NCN project Algebraic Torus Action: Geometry and Combinatorics 2016/23/G/ST1/04282.}
\keywords{Fano manifolds, quaternion-Kahler manifolds, complex contact manifolds, algebraic torus action, 
     homogeneous spaces, adjoint action, localization in K-theory}
     
\subjclass[2010]{Primary: 14L30; Secondary: 53C26, 53D10, 14J45, 14M17, 22E46}

\begin{abstract}
  We prove the LeBrun-Salamon Conjecture in low dimensions.  More
  precisely, we show that a contact Fano manifold $X$ of dimension
  $2n+1$ that has reductive automorphism group of rank at least $n-2$
  is necessarily homogeneous.  This implies that any positive
  quaternion-Kahler manifold of real dimension at most $16$ is
  necessarily a symmetric space, one of the Wolf spaces.  A similar
  result about contact Fano manifolds of dimension at most $9$ 
  with reductive automorphism group also holds. 
  The main difficulty in approaching the conjecture is how to
  recognize a homogeneous space in an abstract variety.  We contribute
  to such problem in general, by studying the action of algebraic
  torus on varieties and exploiting Bia{\l}ynicki-Birula decomposition
  and equivariant Riemann-Roch theorems.  From the point of view of
  $T$-varieties (that is, varieties with a torus action), our result
  is about high complexity $T$-manifolds.  The complexity here is at
  most $\frac{1}{2} (\dim X+5)$ with $\dim X$ arbitrarily high, but we
  require this special (contact) structure of $X$.  Previous methods
  for studying $T$-varieties in general usually only apply for complexity at
  most~$2$ or~$3$.
\end{abstract}


\maketitle

\begin{flushright}
   \emph{Dedicated to Andrzej Szczepan Bia{\l}ynicki-Birula.}
\end{flushright}

\setcounter{tocdepth}{2}

\newpage

\tableofcontents

\section{Introduction}

A complex manifold $X$ of dimension $2n+1$ (with $n\ge 1$) is called a \emph{contact manifold} if there exists a rank $2n$ vector subbundle
   $F \subset T X$ of the tangent bundle with a short exact sequence:
\[
  0 \to F \to T X \stackrel{\theta}{\to} L \to 0, 
\]
  such that the derivative $d \theta|_{F} \colon \Wedge{2} F \to L$ 
  of the twisted form $\theta \in H^0(\Omega^1 X \otimes L)$ is nowhere degenerate.
Moreover, we say $X$ is a \emph{contact Fano manifold}, 
  if in addition it is projective and $\Wedge{2n+1} TX \simeq L^{\otimes (n+1)}$ is an ample line bundle.

The geometry of complex contact manifolds attracts a lot of attention for a few notable reasons. 
It naturally generalizes the real case that appears in classical mechanics.
It is motivated by a problem from Riemannian geometry, 
  namely the Berger classification of all manifolds by their holonomy group \cite{berger}.
Classification of nonsymmetric positive quaternion-K\"ahler manifolds 
(so one of the building blocks in the Berger list; 
  in fact, this is the only building block, for which there is no compact example known) 
  is equivalent to the classification of contact Fano manifolds admitting a K\"ahler-Einstein metric \cite{Salamon-Inventiones}.
Further, contact manifolds connect geometry and representation theory via a version of a moment map \cite{Beauville}, \cite{landsbergmanivel02}.
Finally, such varieties produce a fertile test ground for tools of higher dimensional algebraic geometry, such as Minimal Model Program, minimal rational curves, vector bundles.
It is also strictly related to problems in Riemannian geometry, K\"ahler-Einstein metrics, non-compact hyperk\"ahler manifolds, algebraic group actions and homogeneous spaces, 
  dual varieties, and Legendrian varieties. 

\subsection{The LeBrun-Salamon conjecture 
in dimensions \texorpdfstring{$12$}{12} and \texorpdfstring{$16$}{16}}
\label{sec_LB_S_conjecture}
  
A major open question in this area is the classification of projective contact manifolds.
It is known that they all fit in one of three cases \cite{KPSW}, \cite{Demailly}, \cite{Fano-largeindex}, \cite{kobayashi_ochiai}: 
   if $X$ is a projective contact manifolds of dimension $2n+1$ with the line bundle $L$ as above, then either
\renewcommand{\theenumi}{(\roman{enumi})}
\renewcommand{\labelenumi}{\theenumi}
\begin{enumerate}
 \item $X = \PP(T^*Y)$ for a projective manifold $Y$ of dimension $n+1$ with $L \simeq \cO_{\PP(T^*Y)}(1)$, or
 \item $X = \PP^{2n+1}$ with $L=\cO_{\PP^{2n+1}}(2)$, or
 \item $X$ is a contact Fano variety with $\Pic X = \ZZ L$.
\end{enumerate}
\renewcommand{\theenumi}{(\arabic{enumi})}
\renewcommand{\labelenumi}{\theenumi}
Therefore it remains to classify the last case, and the LeBrun-Salamon conjecture \cite{LeBrunSalamon} claims that they are necessarily rational homogeneous spaces,
   more specifically, the adjoint varieties (see for instance \cite[Table~1 on p.~9]{jabu_dr} for more details).
This conjecture has a reinterpretation in terms of twistor spaces of quaternion-K\"ahler manifolds: 
   if $\cM$ is a compact simply connected quaternion-K\"ahler manifold of real dimension $4n$ with positive scalar curvature, 
   then $\cM$ is (conjecturally) a symmetric space, and more specifically one of the Wolf spaces.
Both versions of the conjecture are proven for $n=1$ \cite{hitchin, ye}, and $n=2$ \cite{poon_salamon, druel}.
It was also claimed for $n=3$ by \cite{2herreras}, however, see \cite{2herreras_erratum}.

One consequence of the results of this article is the proof of the LeBrun-Salamon conjecture for $n=3$ and $n=4$.
\begin{thm}\label{thm_qK_dim_12_16}
   Suppose $\cM$ is a compact simply connected qua\-ter\-nion-K\"ahler manifold of real dimension $4n$ with positive scalar curvature.
   If $n =3$ or $n=4$, then $\cM$ is isometric (up to rescaling) to a symmetric space, one of the Wolf spaces:
    either the quaternion projective space $\HHH\PP^{3}$ or $\HHH\PP^4$, or the complex Grassmannian $Gr(\CC^2, \CC^{5})$ or $Gr(\CC^2, \CC^{6})$,
      or the real Grassmannian of oriented subspaces $\widetilde{Gr}(\RR^4, \RR^{7})$ or $\widetilde{Gr}(\RR^4, \RR^{8})$.    
\end{thm}

Also a slightly weaker version of the conjecture for complex contact Fano manifolds holds in (complex) dimensions $7$ and $9$.
\begin{thm}\label{thm_contact_dim_7_9}
   Suppose $X$ is a contact Fano manifold of dimension $d=2n+1$ with $3\le d\le 9$ (equivalently, $1\le n\le 4$), whose automorphism group is reductive.
   Then $X$ is a homogeneous space, explicitly, one of the following manifolds:
   \begin{itemize}
     \item  a (complex) projective space $\PP^{d}$ for $d=3,5,7,9$, or
     \item  a projectivization of a cotangent bundle $\PP(T^*\PP^{n+1})$ for $n=2,3,4$, or
     \item  the $5$-dimensional adjoint variety of $G_2$, or
     \item  the Grassmannian of projective lines on a smooth quadric hypersurface $Gr(\PP^1, \cQ^{n+2})$ for $n=3,4$.
   \end{itemize}
\end{thm}

\subsection{Contact Fano manifolds with reductive group of automorphisms of high rank} 
Let $X$ be a complex projective manifold of
dimension $d$ with an ample line bundle $L$. We assume that $X$ admits
an action of an algebraic torus $H$ of rank $r$ and the map
$H\rightarrow \Aut(X)$ has at most finite kernel.
Our main interest is when $d=2n+1$ and $X$ is a contact Fano manifold as briefly defined above.
 
\begin{table}[hbt]
   \begin{center}
   \scriptsize
   \begin{tabular}{||c|c|c|p{0.159\textwidth}|p{0.085\textwidth}|c|c||}
   \hline
   \hline
      type     & $X$               & $\dim X$ & $\Aut^{\circ}(X)$, $\Aut^{\circ}(X,F)$& rank of $\Aut(X)$ & $\Pic X$& $\cM$\\
   \hline
   \hline
      $A_{n+1}$& $\PP(T^*\PP^{n+1})$ & $2n+1$   & $PGL_{n+2}$ & $n+1$ & $\ZZ^2$& $Gr(\CC^2, \CC^{n+2})$\\
   \hline
      $C_{n+1}$& $\PP^{2n+1}$ & $2n+1$   & $PGL_{2n+2}$, $Sp_{2n+2}/(\pm \id)$&$2n+1$ & $\ZZ\cdot(\tfrac{1}{2} L)$& $\HHH\PP^{n}$\\
   \hline
      $G_2$ & $G_2$-variety& $5$&$G_2$ & $2$   & $\ZZ \cdot L$& $G_2$-Wolf space\\
   \hline
      $B_3$ & $Gr(\PP^1, \cQ^{5})$& $7$&$SO_7$ & $3$   & $\ZZ \cdot L$& $\widetilde{Gr}(\RR^4, \RR^{7})$\\
   \hline
      $D_4$ & $Gr(\PP^1, \cQ^{6})$& $9$ & $SO_8$ & $4$ & $\ZZ \cdot L$& $\widetilde{Gr}(\RR^4, \RR^{8})$\\
   \hline
      $B_4$ & $Gr(\PP^1, \cQ^{7})$& $11$&$SO_9$ & $4$   & $\ZZ \cdot L$& $\widetilde{Gr}(\RR^4, \RR^{9})$\\
   \hline
      $D_5$ & $Gr(\PP^1, \cQ^{8})$& $13$ & $SO_{10}$ & $5$  & $\ZZ \cdot L$& $\widetilde{Gr}(\RR^4, \RR^{10})$\\
   \hline
      $B_5$ & $Gr(\PP^1, \cQ^{9})$& $15$& $SO_{11}$ & $5$   & $\ZZ \cdot L$& $\widetilde{Gr}(\RR^4, \RR^{11})$\\
   \hline
      $D_6$ & $Gr(\PP^1, \cQ^{10})$& $17$ & $SO_{12}$ & $6$ & $\ZZ \cdot L$& $\widetilde{Gr}(\RR^4, \RR^{12})$\\
      \hline
   \hline
   \end{tabular}
   \end{center}
   \caption{The list of contact Fano manifolds $X$ satisfying $\dim X= 2n+1$ and $\Aut(X)$ reductive and of rank at least $n-2$.
            The \emph{type} refers to the Lie algebra type corresponding to the respective adjoint variety $X$.
            $\cQ^k$ denotes the smooth projective $k$-dimensional quadric in $\PP^{k+1}$, 
               and $Gr(\PP^1, \cQ^k)$ is the orthogonal Grassmannian of projective lines contained in the quadric.  
            $\Aut^{\circ}(X)$ is the identity component of the automorphism group of $X$, 
            while $\Aut^{\circ}(X,F)$ is the component of the group of automorphisms preserving the contact distribution $F$.
            The \emph{$G_2$-variety} is the $5$-dimensional homogeneous space admiting a transitive action of a simple group of type $G_2$ 
               that is not isomorphic to $\cQ^5$.
            $\cM$ is the corresponding quaternion-K\"ahler manifold whose twistor space is $X$.}\label{tab_list_of_contact_with_high_rank_of_torus}
   
\end{table}

The main result of the present paper is the following:
\begin{thm}\label{main_thm}
  Let $X$ be a contact Fano manifold of dimension $2n+1$,
     whose group of automorphisms $G$ is reductive and
  contains an algebraic torus $H$ of rank $n-2$. 
  Then $X$ is a homogeneous space.
  The complete list of all such manifolds is given in Table~\ref{tab_list_of_contact_with_high_rank_of_torus}.
\end{thm}

Because of the twistor construction the above theorem is related to
results for quaternion-K\"ahler manifolds, see the surveys
\cite{Salamon-survey}, \cite{Amann_Phd}, or \cite{jabu_moreno_contact_survey},
as well as \cite{Beauville}, \cite{Bielawski},
\cite{Fang1}, \cite{Fang2}, \cite{Kim} and references therein.
A consequence of Theorem~\ref{main_thm} is an analogous statement on the isometries of a quaternion-K\"ahler manifold.

\begin{thm}\label{thm_qK_rank_of_isometries_group}
  Let $\cM$ be a positive quaternion-K\"ahler manifold of dimension
  $4n$. 
  If the isometry group ${\rm Isom}(\cM)$ has rank at least $n-2$,
    then $\cM$ is isometric to one of the Wolf spaces.
\end{thm}
See Table~\ref{tab_results_qK_isometries} for a comparison of Theorem~\ref{thm_qK_rank_of_isometries_group} with earlier results in this direction.
In particular, our result is the strongest known for $\dim \cM \le 36$, but weaker than \cite{Fang2} for dimension at least $48$.

\begin{table}
  \begin{center}
   \scriptsize
   \begin{tabular}{||c|c|c|c|c|c|c|c|c|c|c|c|c||}
   \hline
   \hline
      reference   &  year & \multicolumn{11}{c||}{lower bound on rank for $\dim \cM=$}\\
                  &       & $4n$ & $12$ & $16$ &$20$ &$24$ &$28$ &$32$ &$36$ & $40$ & $44$ & $48$\\ 
   \hline
   \hline
      \cite[Cor.~7]{Bielawski} & 1999 & $n+1$&  $4$ & $5$ & $6$ & $7$ & $8$& $9$&$10$ & $11$ & $12$ & $13$\\
   \hline   
      \cite[Thm.~B]{Fang1} & 2004 &\multicolumn{2}{l|}{$n-2$, if $n\ge 10$}  &  &  &  & & & & $8$ & $9$ & $10$\\
   \hline   
      \cite[Thm.~1.1]{Fang2} & 2008 & $\left\lceil\frac{n}{2}\right\rceil + 3$ & $5$ & $5$ & $6$ & $6$ &$7$&$7$& $8$& $8$ & $9$& $9$\\
   \hline   
      Theorem~\ref{thm_qK_rank_of_isometries_group} & 2017 & $n-2$ &$1$  & $2$ & $3$ & $4$ & $5$& $6$&$7$ & $8$ & $9$ & $10$\\
   \hline
   \hline
   \end{tabular}
   \end{center}
   \caption{Comparing the known results about the rank of the group of isometries of a quaternion-K\"aler manifold.
            The referenced theorems state that if the rank is at least the number or formula in the table, 
               then the quaternion-K\"aler manifold is isometric to a Wolf space.}\label{tab_results_qK_isometries}
\end{table}

\subsection{Projective manifolds with torus action}\label{subsec_intro_proj_mflds_with_torus_action}

In order to prove Theorem~\ref{main_thm} we propose a new approach to algebraic manifolds with an action of a complex torus $(\CC^*)^r$.
This treatment might be of independent interest and it is described in details in Sections~\ref{sec_torus_action} and~\ref{sec_ABB_decomp}, 
  which never mention ``contact manifolds'' explicitly.
Our focus is on the components of fixed-point sets, some related polytopes, and a notion of compass, 
  which locally describes the action near a fixed point.
The narration is built in a way that immitates the classical correspondence between 
  geometry of toric varieties and combinatorics of convex bodies, see for instance \cite{cox_book}.
We briefly review this correspondence in the following paragraph.
  
Let $X$ be a smooth toric projective variety and $L$ an ample line bundle on $X$.
Denote by $H=(\CC^*)^{\dim X}$ the algebraic torus whose action on $X$ has an open orbit.
The space sections of $L$ decomposes into eigenspaces of the action of $H$:
$$
  \HH^0(X,L)=\bigoplus_{u\in M\cap\Delta}\CC_u,
$$
where $M$ is the lattice of characters of $H$ and $\Delta=\Delta(X,L)$
is a lattice polytope in $M_\RR$. Vertices of $\Delta$ can be
identified with fixed points in $X^H$ and for any vertex $v\in \Delta$
the variety $\Spec(\CC[\RR_{\geq 0}(\Delta-v)\cap M])$ is an affine
neighborhood of the respective fixed point. 
The polytope $\Delta$ can
be also obtained via the moment map.
  
In a more general situation, suppose $X$ is a projective manifold $X$ and again $L$ is an ample line bundle on $X$. 
Also assume a torus $H\simeq (\CC^*)^r$ acts on $X$, and $\mu$ is a linearization of the action on $L$.
From this data one can construct two polytopes: $\Gamma$, which comes from weights of global sections of $L$
   (unlike in the toric case, the multiplicities do not need to be equal to $1$), 
   and $\Delta$, which comes from the linearization of the action on $L$ at fixed points.
In nice situations the two polytopes coincide, 
   see for instance Lemmas~\ref{lem_compare_polytopes} and~\ref{lem_contact-Gamma=Delta},
   and Proposition~\ref{prop_Delta=Gamma-small_corank}. 
Another important ingredient is the compass which is analogous to the affine open neighborhood
   interpretation in the case of a toric variety.
It encodes the characters of the action of $H$ 
   on the (co)tangent space at a fixed point.
There is a strict relation between the compass at a fixed point corresponding to $v\in \Delta$, 
   and the polytopes $\Delta$ and $\Gamma$. 
Namely, the elements of the compass must be contained in the cone generated by $\Delta-v$, 
   and (again, in a sufficiently nice situation) the cone and a related semigroup of lattice points associated to sections of multiples of $L$ must be generated by the elements of the compass.
   
The tool box we use includes Bia{\l}ynicki-Birula decomposition, Theorem~\ref{thm_ABB-decomposition},
   and the localization for torus action, Theorem~\ref{thm_RRekwgeneral}.
As a result we develop new criteria to recognize a homogeneous space in some relatively abstract variety,
   by analyzing the fixed locus $X^H$ and the combinatorics of $\Delta$, $\Gamma$ and the compass 
   ---
   see for instance Propositions~\ref{prop_fixpt+compass-determine-homogeneous} and~\ref{prop_smallDelta2}, Lemmas~\ref{lem_E-and-F} and~\ref{lem_2-fixpts-comp}, and Corollary~\ref{cor_torus-on-quadrics}.
The statements in this article are taylormade to the applications for contact Fano manifolds, 
   but the potential of the methods is more general and yet to be discovered.

\subsection{Content of the paper and an outline of the proof} 

In Section~\ref{sec_torus_action} of the article we consider a general
situation when a torus $H$ acts on a projective manifold $X$, as described in Subsection~\ref{subsec_intro_proj_mflds_with_torus_action}.
We introduce combinatorial objects that encode a lot of information about the action, 
  and explain how these objects change, when we restrict the action to a smaller subtorus, 
  or to an invariant submanifold, or if we replace the polarization by its tensor power, etc.
  
In Section~\ref{sec_ABB_decomp} we describe the BB-decomposition (Theorem~\ref{thm_ABB-decomposition}) 
   and derive its applications relevant to the content of this article.  
The case of primary interest is when $\Pic X=\ZZ$
and rank of $H$ is not too small so that we can use action of its
subtori (downgrading) and the action of quotient tori on the fixed-point sets of these smaller tori (restriction). 
We relate the properties of fixed-point components to those of $X$, 
see Lemmas~\ref{lem_ABBdecomposition-extremal_cells} and~\ref{lem_ABB-surjectivity}.
The most accessible components of $X^H$ are the \emph{extremal components},
which are associated to vertices of a polytope $\Delta$.

In Section~\ref{sec_contact_general} we use these tools to deal with contact
manifolds. 
In the case considered in Theorem~\ref{main_thm}, when the rank of the torus is large with respect to the dimension of the contact manifold,
two convex polytopes $\Gamma$ and $\Delta$ associated to the torus action coincide by Lemma~\ref{lem_contact-Gamma=Delta}.
Here $\Gamma$ is a convex hull of weights of sections of the line bundle $L$.
Results of \cite{Beauville} imply that $\Gamma$ is the convex hull of the roots of the group $G$ of contactomorphism of $X$.
By contrast, the other polytope $\Delta$ is full dimensional in the character lattice of $H$.
Therefore, $G$ must be semi-simple, as we show in Lemma~\ref{lem_contact->Gsemisimple},
  and eventually simple by Proposition~\ref{prop_contact->Gsimple}. 
Thus $\Delta(X,L,H)$ is the convex hull of
roots in the space of weights of a simple group $G$,
known as the root polytope of
the group $G$. 

In Section~\ref{sec_contact_simple},
  we list a bunch of assumptions that are implied by the hypotheses of Theorem~\ref{main_thm} 
  and the results of earlier sections.
Then we show Theorem~\ref{thm_contact_simple} that under the above assumptions, there is only a short list of possible groups of automorphisms, 
  or the dimension is too large to remain consistent with Theorem~\ref{main_thm}.
In the proof we run a case-by-case analysis of root polytopes of simple groups, exploiting heavily the theory introduced in 
  Sections~\ref{sec_torus_action} and \ref{sec_ABB_decomp} in very explicit situations.
An important criterion is Proposition~\ref{prop_fixpt+compass-determine-homogeneous}, 
  which allows to recognize a homogeneous manifold from the properties of an action of the maximal torus.

In Section~\ref{sec_proofs} we conclude the proofs of all theorems mentioned in the introduction and we review the relevant literature.
In particular, we slightly strengthen a theorem of Salamon to show that the dimension of the automorphisms group of a contact Fano manifold
  in dimension $7$ or $9$ is bounded from below by $5$ or $8$, respectively (Theorem~\ref{thm_dim_of_Aut_X}).
We use this to eliminate the last remaining case left in the proof of Theorem~\ref{main_thm}.
We also explain how the classification of contact Fano manifolds in dimension $7$ and $9$ follows (Theorem~\ref{thm_contact_dim_7_9}),
  and how to derive from the literature the appropriate results about quaternion-K\"ahler manifolds 
  (Theorems~\ref{thm_qK_dim_12_16} and~\ref{thm_qK_rank_of_isometries_group}).

In Appendix~\ref{sect_appendix} we give an exposition of the Localization Theorem for torus action (Theorem~\ref{thm_RRekwgeneral}).
It is derived from classical works by  Atiyah--Bott, Grothendieck, Atiyah--Singer, and Berline--Vergne on localization in cohomology.
In addition, we apply it to prove Corollary~\ref{cor_RRekw}, 
  a special case, when the fixed-point locus consist of isolated points and curves only.
This set-up provides a clear and manageable tool for calculating characters of torus representation on the space of sections of equivariant ample line bundles.
Then the combinatorial data discussed in Section~\ref{sec_torus_action} are enough to determine the space of sections of a line bundle as an $H$-representation.
The corollary is applied in the proofs in Subsection~\ref{subsec_comparing_G_varieties}.

\subsection{Notation}

The following notation is used throughout the article.
\begin{itemize}[leftmargin=*]
\item $X$ is a connected projective manifold over complex numbers of dimension
  $d$. 
  In Sections~\ref{sec_contact_general} and~\ref{sec_contact_simple} we will additionally assume that $X$ admits
  contact structure and $d=2n+1$.
\item $L$ is an ample line bundle over $X$. 
  (We frequently assume that $\Pic X=\ZZ\cdot L$.)
\item If $X$ admits a contact structure (in Sections~\ref{sec_contact_general} and~\ref{sec_contact_simple}),
  then $\theta\in\HH^0(X,\Omega_X\otimes L)$ is a contact form on $X$ (see Section~\ref{sec_contact_general} for more details).
\item $H$ denotes an algebraic torus acting on $X$, and 
  $M$ is the lattice of characters (or weights) of $H$.
  We assume both $H$ and $M$ are of rank $r$.
  (We frequently suppose the action of $H$ is \emph{almost faithful}, that is $H \rightarrow \Aut(X)$ has finite kernel,
     or equivalently, there are only finitely many group elements that act trivially on $X$.)
  By a minor abuse we will shamelessly identify weights and characters throughout Sections~\ref{sec_torus_action}--\ref{sec_proofs}.
    In Appendix~\ref{sect_appendix} we will also use characters of representations of $H$ and we are careful to distinguish between weights and characters.
\item $G$ is a connected reductive group with a maximal torus $H$, 
     and $G$ acts almost faithfully on $X$.
\item By $\mu$ or $\mu_L$ we denote a linearization of the action of
  $G$ or $H$ on the line bundle $L$.
  $\Gamma(L)=\Gamma(X,L,H,\mu)$
   is the \emph{polytope of sections} in $M_\RR$, that is, the convex hull of the set of
   weights (eigenvalues) of the action of $H$ on $\HH^0(X,L)$, 
   see Subsection~\ref{subsec_polytopes_of_sections_and_fixed_points} for details.
\item ${\mathcal R}(X,L)$ is the ring of sections of $L^{\otimes
    m}$, $m\geq 0$, graded by $M\times \ZZ_{\geq 0}$.
\item By $X^H$ we denote the set of fixed points of the
  action of $H$.
\item $\Delta(L)=\Delta(X,L,H,\mu)$ is the {\em polytope of
    fixed points} in $M_\RR$,
    which is the convex hull of the characters $\mu(Y)$
    with which $H$ acts on $L$ over a fixed-point component $Y\subset X^H$,
     see Subsection~\ref{subsec_polytopes_of_sections_and_fixed_points} for details.
    A component of $X^H$ associated to a vertex of $\Delta(L)$
     is called {\em extremal}. 
\item The {\em compass} $\cC(Y,X,H)$ of the action of $H$ on a
  component $Y\subset X^H$ is the set of (nontrivial) characters
  $\nu_1,\dots,\nu_{\dim X-\dim Y}$ (possibly, with repetitions)
  associated to the linearization of the action of $H$ on the conormal
  bundle of $Y$ in $X$, see Subsection~\ref{subsec_compass} for details.
\item For a semi-simple group $G$ by $\Delta(G)$ we denote the
  \emph{root polytope} of $G$.
  That is, $\Delta(G)$ is the polytope in the lattice of characters of the maximal torus $H$,
  which is the convex hull of roots of $G$.
  If $G$ is simple and simply connected of type $A_r, B_r,\dots, G_2$ then we also write $\Delta(A_r),$ etc.
\end{itemize}

\subsection*{Acknowlegdements}

We would like to thank Klaus Altmann, Roger Bielawski, and Aleksandra Bor{\'o}wka for many interesting conversations.
We are also grateful to the participants of workshops in Oberwolfach (Novemeber 3--9, 2013)
  and in Levico Terme (November 2--6, 2015) for inspiring talks and discussions.
We thank Michel Brion, Joseph Landsberg, Eleonora Romano, 
  Uwe Semmelmann, and anonymous referees for their comments on the initial versions of the article.

\section{Torus action and combinatorics}\label{sec_torus_action}

\subsection{Polytopes of sections and of fixed points}\label{subsec_polytopes_of_sections_and_fixed_points}
We consider an ample line bundle $L$ over $X$ with an action of an
algebraic torus $H$. We will assume that $H\rightarrow \Aut(X)$ has at
most finite kernel.

There exists a
linearization $\mu=\mu_L$ of the action of $H$ on $L$, \cite[Prop.~2.4, and the following Remark]{Knop-etal}. Any two
linearizations differ by a character in $M$. 
For more details see \cite{Brion}. 
To each $p\in X^H$ we associate $\mu(p)\in M$, the weight of the action of $H$ on the fiber $L_p$.

Let $X^H=Y_1\sqcup\cdots\sqcup Y_s$ be a decomposition of the fixed-point locus into connected components. 
Then each $Y_i$ is smooth \cite{Luna} and
if $y_1$ and $y_2$ belong to the same connected component $Y$ then
$\mu(y_1)=\mu(y_2)$ and we will denote this by $\mu(Y)$.  This
determines lattice points associated to characters $\mu(Y_1),\dots
\mu(Y_s)\in M$. By $\widetilde{\Delta}(X,L,H,\mu)$ we denote the set
of these characters and we define a {\em polytope of fixed points}
$\Delta(X,L,H,\mu)$ in $M_\RR=M\otimes\RR$ as the convex hull of these
points. We will frequently simplify the notation and when referring to
$\widetilde{\Delta}$ or $\Delta$ we will skip the elements of the
quadruple $(X,L,H,\mu)$ which are obvious from the context.

In the literature (for instance \cite{brion_faces_of_moment_polytope}, \cite{cannas_da_silva_intro_sympl_geom}) 
   polytope $\Delta$ is also called the \emph{moment polytope} and arises via the \emph{moment map}.
However, in the setting of contact manifolds, the meaning of moment map is ambiguous, as a different map is also called the \emph{moment map},
   see \cite{Beauville}.
Thus to avoid confusion we refrain from using the word \emph{moment} in either context.
We also want to stress the difference between the two polytopes $\Delta$ and $\Gamma$ which are built on the fixed points and 
  on the sections of possibly not spanned line bundle.

A component $Y\subset X^H$ is called {\em extremal} if $\mu(Y)$
is a vertex of $\Delta(L)$.

A linearization $\mu$ determines a decomposition into
eigenspaces of the $H$ action:
$$
\HH^0(X,L)=\bigoplus_{u\in\widetilde{\Gamma}(X,L,H,\mu)}\HH^0(X,L)_u,
$$ 
where $\HH^0(X,L)_u$ denotes the eigenspace on which $H$ acts with the
weight $u\in M$ and $\widetilde{\Gamma}(X,L,H,\mu_L)$ is the set of
the weights (eigenvalues) of the action of $H$ on $\HH^0(X,L)$. We
define a lattice {\em polytope of sections} $\Gamma(X,L,H,\mu)$ in
$M_\RR$ as the convex hull of $\widetilde{\Gamma}(X,L,H,\mu)$.  Again,
we will frequently simplify the notation and skip the elements of the
quadruple $(X,L,H,\mu)$ which are obvious from the context.

As a digression, we remark that if $L$ is very ample 
  then $\Gamma(L)$ is the image of the moment map, see \cite[Ch.~2]{Kirwan}.

Observe that a linearization of $L$ determines
a linearization of $L^{\otimes m}$ for $m \in \ZZ$.
Therefore the action of $H$
determines the grading of the ring of sections
\begin{equation}
{\mathcal R}(X,L)=\bigoplus_{m\geq 0}\HH^0(X,L^{\otimes m})=
\bigoplus_{m\geq 0}\bigoplus_{u\in M}\HH^0(X,L^{\otimes m})_u.
\end{equation}

We introduce the \emph{weight cone} of ${\mathcal R}(X,L)$, denoted by $\widehat{\Gamma}(L)\subset M_\RR\times\RR$ and defined as
\begin{equation}\label{equ_cone_of_grading}
  \widehat{\Gamma}(L)=\RR_{\geq 0}\cdot
  \left(\frac{\Gamma(L^{\otimes
        m})}{m}\times \{1\}\right)=
  \RR_{\geq 0}\cdot
  \left(\Gamma(L^{\otimes
        m})\times \{m\}\right)
\end{equation}
where $m$ is sufficiently large. 
This is a polyhedral cone, since ${\mathcal R}(X,L)$ is finitely generated.
If $\gamma\subset \widehat{\Gamma}(L)$ is a ray (that is, a one-dimensional face), 
   then we can define a homogeneous ideal in ${\mathcal R}(X,L)$:
\begin{equation}\label{equ_extremal_ideal}
   {\mathcal J}_\gamma = \bigoplus_{m\geq 0} \bigoplus_{u\not\in\gamma} \HH^0(X,L^{\otimes m})_u.
\end{equation}
The ideal ${\mathcal J}_\gamma$ defines a closed $H$-invariant subset of $X$.
Note that this subset is non-empty, since the Hilbert polynomial of ${\mathcal R}(X,L)/ {\mathcal J}_\gamma$
   is equal to $\dim H^0(X, L^{\otimes m})_u$ for large values of $m$, which is non-zero.
We will briefly use this set in the proof of Lemma~\ref{lem_compare_polytopes}, 
   and then discuss it in more details in Corollary~\ref{cor_ideal_of_extremal_component}.

\begin{lemma}\label{lem_compare_polytopes}
In the situation above the following holds:
\begin{enumerate}[leftmargin=*]
\item \label{item_compare_polytopes_rescaling} 
      $m\Delta(L)=\Delta(mL)$ and $m\Gamma(L)\subseteq\Gamma(mL)$ for any integer $m> 0$,
\item \label{item_compare_polytopes_Gamma_in_Delta} 
      $\widetilde{\Gamma}(L)\subset\Delta(L)$,
\item \label{item_compare_polytopes_base_point_free_Delta_eq_Gamma} 
      if $L$ is base point free then ${\Delta}(L)={\Gamma}(L)$.
\end{enumerate}
\end{lemma}
\begin{proof}
The first part follows directly from the definitions. 

In the proof we will consider the following situation.
Fix any integer $m>0$, and suppose $y\in X^H$ is a fixed point that is not contained in the base point locus of $L^{\otimes m}$.
Then we have the short exact sequence of $H$-modules:
\begin{equation}\label{equ_restriction_to_point}
0\lra \HH^0(X,L^{\otimes m}\otimes\m_y)\lra \HH^0(X,L^{\otimes m})\lra L^{\otimes m}_y\lra0,
\end{equation} 
where $\m_y$ is the maximal ideal of $y$.
Note that the sequence admits an $H$-equivariant splitting.

Now we argue for the second item.
Assume by contradiction that $\widetilde{\Gamma}(L)$ is not contained in $\Delta(L)$.
Then there  exists a $1$-dimensional face $\gamma$ of the weight cone $\widehat{\Gamma}(L)$ 
   such that the intersection $\gamma\cap M_\RR\times\{1\}$ is not contained
in $\Delta(L)\times\{1\}$. 
Consider the closed non-empty $H$-invariant subset of $X$ defined by the ideal ${\mathcal J}_\gamma$ as in Equation~\eqref{equ_extremal_ideal}.
We claim, that the subset does not contain any $H$-fixed point.
Indeed, suppose $y$ is such a fixed point.
By the equivariant splitting in the exact sequence~\eqref{equ_restriction_to_point}
   there exists a homogeneous section of~$L^{\otimes m}$ (for $m$ sufficiently large) that does not vanish at $y$.
Its degree determines a character in $\Delta(L^{\otimes m})=m\Delta(L)$. 
Since $\gamma\cap\Delta(L)=\emptyset$, the section must be in ${\mathcal J}_\gamma$.
In other words, the section vanishes at $y$, a contradiction, which proves the claim.
This in turn concludes the proof of \ref{item_compare_polytopes_Gamma_in_Delta}, 
  since $X$ is projective and there is no closed $H$-invariant subset of $X$ with no fixed points.

For the third part, by \ref{item_compare_polytopes_Gamma_in_Delta} we only have to prove ${\Delta}(L) \subset {\Gamma}(L)$.
Using the equivariant splitting in the exact sequence~\eqref{equ_restriction_to_point} for $m=1$ we find a homogeneous section that 
  has the degree equal to $\mu(y)$, proving the inclusion.
\end{proof}

\begin{cor}\label{cor_Delta_has_maximal_dim}
   In the notation as above, suppose $H$ acts almost faithfully on $X$. 
   Then $\Delta(L)\subset M_{\RR}$ is a lattice polytope of full dimension
      (that is, $\dim \Delta(L) = \dim M_{\RR} = \dim H$).
\end{cor}

\begin{proof}
   If $mL$ is very ample and defines an embedding of $X$ into a projective space $\PP^N$,
      then $\Gamma(X, mL) = \Gamma(\PP^N, \cO(1))$ is of full dimension 
      since the action on $X$ and also on $\PP^N$ is almost faithful.
   Thus $\Delta(mL)$ is of full dimension by Lemma~\ref{lem_compare_polytopes}\ref{item_compare_polytopes_base_point_free_Delta_eq_Gamma}, 
      and also $\Delta(L)$ is of full dimension by Lemma~\ref{lem_compare_polytopes}\ref{item_compare_polytopes_rescaling}.
\end{proof}

\begin{lemma}\label{lem_Delta_of_submanifold}
    Suppose $X'$ is a projective manifold, $L'$ is a line bundle, a torus $H$ acts on $X'$ with a linearization $\mu_{L'}$ of the action. 
    Further assume $X\subset X'$ is a closed $H$-invariant submanifold, and $L = L'|_{X}$, 
      and the linearization $\mu_{L}$ is the restriction of $\mu_{L'}$ to $X$.
    Then:
    \begin{enumerate}
     \item For all $y \in X^H$ we have $\mu_{L}(y)= \mu_{L'}(y)$.
     \item $\Delta(X, L, H, \mu_{L}) \subset \Delta(X', L', H, \mu_{L'})$.
     \item \label{item_Delta_of_submanifold_equal_to_Delta_of_projective_space}
            In the special case, where $X'=\PP(\HH^0(X,L)^*)$, $L' = \cO(1)$, 
              and the embedding $X\subset X'$ is given by the linear system,   
              $\Delta(X', L', H, \mu_{L'}) = \Delta(X, L, H, \mu_{L})$. 
    \end{enumerate}
\end{lemma}
The proof is straightforward from definitions
   and, for Item \ref{item_Delta_of_submanifold_equal_to_Delta_of_projective_space}, 
   also using Lemma~\ref{lem_compare_polytopes}\ref{item_compare_polytopes_base_point_free_Delta_eq_Gamma}.
We ommit the details.

As shown in Lemma~\ref{lem_compare_polytopes}, 
  for sufficiently large $m$ the polytopes satisfy $m\Delta(L)=\Delta(L^{\otimes  m})=\Gamma(L^{\otimes m})$.
Thus the vertices of $\Delta(L)$ are equal to 
  the generators of rays of the weight cone $\widehat{\Gamma}(L)$ defined in Equation~\eqref{equ_cone_of_grading}.
Let $v\in \Delta(L)$ be a vertex and let $\gamma_v$ be the corresponding ray of $\widehat{\Gamma}(L)$.
Consider the homogeneous ideal $\cJ_{\gamma_v} \subset \cR(X,L)$ as in Equation \eqref{equ_extremal_ideal}.

\begin{cor}\label{cor_ideal_of_extremal_component}
  The ideal $\cJ_{\gamma_v}$ is radical and defines a closed subset of $X$ consisting
     of all those fixed points $y\in X^H$ that $\mu(y)=v$.
\end{cor}
\begin{proof}
  Suppose $s\in \cR(X,L)_{u}$ is a homogeneous element of the section ring such that $s^m\in \cJ_{\gamma_v}$ for some integer $m>0$.
  Then $m u \notin \gamma_v$, hence $u\notin \gamma_v$ and $s\in \cJ_{\gamma_v}$. That is, $\cJ_{\gamma_v}$ is radical as claimed.

  We consider $m\gg 0$ such that $L^{\otimes m}$ is very ample and
    it equivariantly embeds $X$ into a projective space $\PP^N = \PP(\HH^0(X,L^{\otimes m})^*)$ 
    with the action of $H$
    determined by the decomposition 
  $$\HH^0(X,L^{\otimes m})=\bigoplus_{u\in\Gamma(L^{\otimes m})}\HH^0(X,L^{\otimes m})_u.$$

  Suppose $y\in X \subset \PP^N$ is such that all sections from $\cJ_{\gamma_v}$ vanish on $y$.
  In particular, all sections from $\cJ_{\gamma_v}\cap\HH^0(X,L^{\otimes m})$ vanish on $y$, 
     thus $y\in \PP((\HH^0(X,L^{\otimes m})^*)_{mv})$.
  All points in this subspace are fixed by the torus action and by Lemma~\ref{lem_Delta_of_submanifold} the corresponding character is
     \[
       m v = \mu_{\cO_{\PP^N}(1)} (y) = \mu_{L^{\otimes m}} (y) = m \mu_L (y).
     \]
  
  Conversely, suppose $y\in X$ is a fixed point of the torus action.
  Then $y\in \PP^N$ is in a fixed-point locus, and the character $\mu_L(y)$ forces
     it to be in the linear space $\PP((\HH^0(X,L^{\otimes m})^*)_{mv}$.
  Thus $(\cJ_{\gamma_v})_{mv}$ vanishes on $y$. 
  The same happens for all sufficiently large $m$, and since the ideal is radical, whole $\cJ_{\gamma_v}$ vanishes on $y$ as claimed. 
\end{proof}

\subsection{Reduction of torus action: downgrading and restriction}\label{subsec_downgrading_and_reduction}
Let $H_1\subset H$ be a subtorus with the quotient torus
$H_2=H/H_1$. Clearly the linearization $\mu$ of the action of $H$ on
$L$ defines the linearization $\mu_1$ of the action of $H_1$. We have
an exact sequence of the respective lattices of characters
\begin{equation}
\xymatrix{0\ar[r]& M_2\ar[r]^\iota& M\ar[r]^\pi &M_1\ar[r]& 0}
\end{equation} 

The following lemma describes the situation of {\em downgrading of the
  action} of $H$ to $H_1$ and the properties of the {\em restriction
  of the action} of $H$ to components of $X^{H_1}$ which yields
the action of the quotient torus $H_2$.
\begin{lemma}\label{lem_reduction_of_action}
Let $H$'s, $M$'s and $\mu$'s be as above.
\begin{enumerate}[leftmargin=*]
\item The projection $\pi\colon M\ra M_1$ restricts to surjective maps of polytopes 
\begin{align*}
  \Gamma(X,L,H,\mu)&\ra\Gamma(X,L,H_1,\mu_1) \text{ and} \\
  \Delta(X,L,H,\mu)&\ra\Delta(X,L,H_1,\mu_1)
\end{align*}
\item Let $Y_1\subset X^{H_1}$ be a connected component of the fixed-point locus of $H_1$.  Then $Y_1$ is $H$ invariant and 
  the action of $H$ on $Y_1$ descends to the action of the quotient $H_2$.
  The linearization $\mu$ of $L|_{Y_1}$ can twisted by $\widetilde{\mu_1(Y_1)}\in M$
  which is any preimage of $\mu_1(Y_1)\in M_1$, and this twisted linearisation $\mu - \widetilde{\mu_1(Y_1)}$ 
  induces a linearization  $\mu_2$ of the action of $H_2$ on $L|_{Y_1}$.
  (We stress that $\mu_2$ depends on the choice of the preimage $\widetilde{\mu_1(Y_1)}\in M$).
\item\label{item_reduction_of_action_restricting_to_fix_points_of_subgroup}
  Moreover, $Y_1^{H_2}=X^H\cap Y_1$ and
  $$\widetilde{\Delta}(Y_1,L_{|Y_1},H_2,\mu_2)\subseteq 
  \pi^{-1}(\mu_1(Y_1))\cap \widetilde{\Delta}(X,L,H,\mu)-\widetilde{\mu_1(Y_1)}.$$
\end{enumerate}
\end{lemma} 
\begin{proof}

The claims follow directly from definitions.
Perhaps the only not completely trivial check is that $Y_1$ is $H$-invariant, which follows from commutativity and connectedness of $H$.
Indeed, if $h\in H$ and $y_1\in Y_1$, then for all $h_1\in H_1$ we have $h_1 \cdot (h \cdot y_1) =  h \cdot (h_1\cdot  y_1 ) =  h \cdot y_1$.
That is $h \cdot y_1 \in X^{H_1}$, and by connectedness of $H$, it must belong to the same connected component as $y_1$, which  is $Y_1$. 

The choice of $\widetilde{\mu_1(Y_1)}$ ensures that the linearization 
$\mu-\widetilde{\mu_1(Y_1)}$ of the action of $H$ is trivial after the restriction to $H_1$.
\end{proof}

\subsection{The compass}\label{subsec_compass}
For every $y\in X^H$ we have a natural linearization of the action of
$H$ on the cotangent space $T_yX^*=\m_y/\m_y^2$.
This determines the set of weights of this action $\nu_1(y),\dotsc, \nu_d(y)\in M$ (possibly with repetitions) which is the
same for every $y$ in a connected component $Y\subset X^H$, hence
denoted by $\nu_1(Y),\dots, \nu_d(Y)$.  The number of zero weights
among $\nu_i(Y)$ is equal to $\dim Y$. The set of non-zero
$\nu_i(Y)$'s, with possibly multiple entries, will be called {\em the
  compass} of the component $Y$ in $X$ with respect to the action of
$H$ and denoted ${\mathcal C}(Y,X,H)$. We will usually write
$\cC(Y,X,H)=(\nu_1^{a_1},\dots,\nu_k^{a_k})$ where $\nu_i$'s are
pairwise different elements of $M$ and the positive integers $a_i$ denote the number of occurences of $\nu_i$ in the compass.

In the literature a similar notion of \emph{moment graph} is discussed 
  \cite{braden_macpherson_moment_graphs}, \cite{fiebig_moment_graphs_in_repr_theory_and_geom}.
In some references it is also called the \emph{GKM-graph} \cite{guillemin_holm_zara_GKM_description_equivariant_cohomology}.
Under the assumption that there are only isolated fixed points and finitely many one dimensional orbits, 
  the behaviour of this graph encodes in particular the mutual relations between elements of compasses of various fixed points.

We first discuss a special case when the torus $H$ has dimension $1$.

\begin{ex}\label{ex_compass_for_Cstar_acting_on_PV}
Let $H=\CC^*$ be a torus acting on a vector space $V$ with weights $a_1,\dots, a_s$, 
  where $a_1 <\cdots < a_s$, and the associated eigenspaces $V_i$, $\dim V_i=d_i$. 
The action of $H$ on $X=\PP(V)$ 
  has $s$ fixed-point components $Y_i=\PP(V_i)$ of dimension $\dim Y_i =d_i-1$. 
There is a natural linearization $\mu$ of the action of $H$ on $L=\cO(1)$ (so that $H^0(L)=V^*$) such that $\mu(\PP(V_i))=-a_i$. 
Then  $\cC(Y_i,X,H)=((a_i-a_1)^{d_1},\dots, (a_i-a_s)^{d_s})$
  where the sequence ommits the zero terms $(a_i-a_i)^{d_i}$. 
The extremal
  fixed-point components of the action of $H$ are $Y_1$ and $Y_s$.
They coincide with those
  for which the sign of all the elements of the compass is the same.
They are called the source  and the sink of this action, 
  as a general orbit ``flows'' from the highest weight $\mu(Y_1)$ to the lowest weight $\mu(Y_s)$.
\end{ex}

\begin{lemma}\label{lem_compass_and_Delta_for_Cstar}
  Let $H=\CC^*$ be the torus acting almost faithfully on a projective manifold $X$ with an ample line bundle $L$.
  Suppose $v$ is a vertex of $\Delta$  and $Y \subset X^{\CC^*}$ is a corresponding extremal component. 
  Then:
  \begin{enumerate}
   \item if $v$ is such that $\Delta \subset v + \RR_{\ge 0}$ then ${\mathcal C}(Y,X,H) \subset \ZZ_{>0}$ 
      (in this situation, we say the extremal component $Y$ is a \emph{sink}),
   \item if $v$ is such that $\Delta \subset v + \RR_{\le 0}$ then ${\mathcal C}(Y,X,H) \subset \ZZ_{<0}$
      (in this situation, we say the extremal component $Y$ is a \emph{source}).
  \end{enumerate}
\end{lemma}
\begin{proof}
  By passing to a multiple of $L$ we only rescale $\Delta$ (see Lemma~\ref{lem_compare_polytopes}\ref{item_compare_polytopes_rescaling}),
    thus we may assume that $L$ is very ample.
  Then $X$ is embedded equivariantly in the projective space on which the torus $H$ acts as in Example~\ref{ex_compass_for_Cstar_acting_on_PV}
    with $V=\HH^0(X,L)^*$. 
  By Lemma~\ref{lem_Delta_of_submanifold}\ref{item_Delta_of_submanifold_equal_to_Delta_of_projective_space}
    we have the equality $\Delta(\PP(V), \cO(1), \CC^*) = \Delta(X, L, \CC^*)$.
  Moreover, we have ${\mathcal C}(Y,X,H)\subset {\mathcal C}(Y',\PP(V),H)$ by definition of the compass 
    (where $Y'$ is the component of $\PP(V)^H$ containing $Y$).
  Thus the claim follows from Example~\ref{ex_compass_for_Cstar_acting_on_PV}.
\end{proof}

Lemma~\ref{lem_compass_and_Delta_for_Cstar} also has a partial converse. Namely, if $Y \subset X^{\CC^*}$ is a component of the fixed-point locus 
   and the compass consists of only positive or only negative elements, then $Y$ is an extremal component.
To show this one uses Bia{\l}ynicki-Birula decomposition (Theorem~\ref{thm_ABB-decomposition}), but we will not use this statement here, so we skip the proof.

We have the following extension of Lemma~\ref{lem_reduction_of_action}.

\begin{lemma}\label{lem_reduction_of_compass}
  In the situation of Lemma~\ref{lem_reduction_of_action} take connected
  components $Y_1\subset X^{H_1}$ and $Y\subset Y_1\cap X^H$. Then
  \begin{align*}
    {\mathcal C}(Y_1,X,H_1)  &= \pi({\mathcal C}(Y,X,H)) \text{ and}\\
     {\mathcal C}(Y,Y_1,H_2) &=     {\mathcal C}(Y,X,H)\cap\ker\pi.
  \end{align*}
\end{lemma}
\begin{proof} 
  Again, the proof is a direct consequence of the definitions.
\end{proof}

\begin{cor}\label{cor_compass_and_fixpt}
  Let $Y\subset X^H$ be a fixed-point component of the action of an
  algebraic torus $H$ and $\nu\in\cC(Y,X,H)$ an element in the compass
  of $H$ at $Y$. Then there exists a component $Y'\subset X^H$ and
  $\lambda\in\QQ_{>0}$ such that $\mu(Y')=\mu(Y)+\lambda\cdot \nu$. In particular 
  \[
    (\mu(Y)+\QQ_{>0}\cdot \nu)\cap\Delta(X,L,H,\mu)\cap M\ne\emptyset.
  \]
\end{cor}
\begin{proof}
  We restrict the action of $H$ to the action of a torus $H_1$
  obtained by projection of $M$ along $\nu$. Then there exists a
  component $Y_1\subset X^{H_1}$ which contains $Y$. The action of the
  1-dimensional torus associated to $\nu$ on $Y_1$ has $Y$ as a fixed-point component by Lemma~\ref{lem_reduction_of_action}\ref{item_compare_polytopes_Gamma_in_Delta}, 
     and $\nu$ is in its compass by Lemma~\ref{lem_reduction_of_compass}.
  Thus the claim of the corollary follows from the rank one case discussed in Lemma~\ref{lem_compass_and_Delta_for_Cstar}.
\end{proof}

The usefulness of the notion of compass (particularly for the case of rank of $H$ at least $2$) lies in the ease of controlling the dimensions and fixed points.
We illustrate it in details on specific applications in Section~\ref{sec_contact_simple}.
Here we mention a general observation which allows us to even better control the directions and multiplicities of the compass.

\begin{lemma}\label{lem_number_of_vectors_on_compass_in_a_fixed_direction}
   Suppose $Y\subset X^H$ is a fixed-point component of dimension $a$, and that the compass of $Y$ in $X$ contains $b>0$ elements (counting with multiplicities) 
      on some linear subspace $W \subset M_\RR$.
   Then for every $\nu \in \cC(Y,X, H) \cap W$ there exists a fixed-point component $Y' \subset X^H$ such that:
   \begin{itemize}
     \item  $\mu(Y') -\mu(Y) \in \nu \cdot \QQ_{> 0}$,
     \item  $a'+ b' = a+ b$, where $a' =\dim Y'$ and $b'$ is the number of elements of $\cC(Y',X, H) \cap W$ (with multiplicities). 
   \end{itemize}
   In particular, $\dim Y < a' + b'$, where we can take $a'$ and $b'$ as in the second item, maximazing over all suitable $Y'$.
\end{lemma}
\begin{proof}
   Consider the subtorus $H_1 \subset H$ which corresponds to the projection $M  \to  M/(W\cap M)$.
   Let $Y_1\subset X^{H_1}$ be the fixed-point component containing $Y$ and let $H_2 = H/H_1$.
   By Lemma~\ref{lem_reduction_of_compass} the compass $\cC(Y, Y_1, H_2)$ consists of $b$ elements,
     in particular, $\dim Y_1 = a+b > \dim Y$.
   By Corollary~\ref{cor_compass_and_fixpt} there is a component $Y' \ne Y$ of $Y_1^{H_2}$
     with $\mu_2(Y') -\mu_2(Y) \in \nu \cdot \QQ_{> 0}$.
   We have $\dim Y' + \#\cC(Y',Y_1, H_2)  = \dim Y_1$ and $Y'$ is a component of $X^H$.
   By Lemma~\ref{lem_reduction_of_action} 
      we must have $\mu_2(Y) -\mu_2(Y') = \mu(Y) -\mu(Y')$ so the first item is satisfied.
   Therefore, by Lemma~\ref{lem_reduction_of_compass}, we have $ \cC(Y',Y_1, H_2)= \cC(Y',X, H) \cap W$ 
     and the itemized properties hold.
\end{proof}

Let $y\in Y\subset X^H$ be a fixed point in a connected component $Y$ with $\dim Y=d_Y$. 
Then, up to an \'etale cover,  the action of $H$ can be diagonalized with weights
associated to the compass of $H$ at $Y$, see \cite[Lemma~on p.~96]{Luna}.
That is, there exists a coordinate system $x_1,\dots x_d$ around $y=\{x_i=0, i=1,\dots, d\}$
  such that for $t \in H$ it holds
\begin{equation}\label{equ_local-action}
t\cdot(x_1,\dots,x_d)=
(t^{\nu_1} x_1,\dots,t^{\nu_{d-d_Y}} x_{d-d_Y}, x_{d-d_Y+1},
\dots,x_d),
\end{equation}
where $\cC(Y,X,H)=(\nu_1,\dots\nu_{d-d_Y})$ and $t^{\nu}$ denotes the application of the character $\nu\colon M \to \CC^*$ to $t \in H$.
(In coordinates, if $t = (\fromto{t_1}{t_r}) \in H\simeq (\CC^*)^r$, and $\nu = (\fromto{\nu^1}{\nu^r}) \in M \simeq \ZZ^r$, 
  then $t^{\nu} = t_1^{\nu^1} \dotsm t_r^{\nu^r}$.)  
In other words, the
functions $x_i$ are graded by elements of $M$, with $\deg x_i=\nu_i$, for $i\leq
d-d_Y$ and $\deg x_i= 0$ for $i > d - d_Y$.

\begin{lemma}\label{lem_local_description_inv_divisor}
  Let $D_\sigma\subset X$ be an $H$-invariant divisor associated to 
    an $H$-equivariant section $\sigma\in\HH^0(X,L)_u$, where $u\in M$.
  Suppose that $y$ and $Y$ are as above.
  Then in the above local coordinates near $y$
  the divisor $D_\sigma$ is described as the zero set of a homogeneous
  power series $f\in\CC[[x_1,\dots, x_d]]$ of degree $u-\mu(Y)$ with
  respect to the grading introduced above.
\end{lemma}

\begin{pf}
  Let $z$ be a local coordinate in the fiber of $L$ around $y$. Then,
  locally around $y$ the bundle can be trivialized and the section
  $\sigma$ can be presented as $\sigma(x)=f(x)\cdot z$ where $f$ is a
  power series in $H$-homogeneous coordinates introduced above. By our
  assumption $\sigma$ is an eigenvector of the action of $H$
  associated to the weight $u$ while $H$ acts on $z$ with the weight
  $\mu(Y)$. Thus $H$ acts on $f$ with weight $u-\mu(Y)$ and the claim
  follows.
\end{pf}

We remark that the power series $f$ appearing in the lemma can be chosen to be convergent in some neighbourhood of $0$. We will not use this stronger statement.

\begin{cor}\label{cor_weights-jets}
  For every $u\in\widetilde{\Gamma}(X,H,L,\mu)$ and a component $Y\subset X^H$
  there exist non-negative integers $a_i$ such that
  $$u-\mu(Y)=\sum a_i\nu_i$$
  where $\nu_i$ are in the compass $\cC(Y,X,H)$.
\end{cor}
\begin{proof}
The equality follows from calculating the degree of $f$ from the
previous lemma, where $0\ne \sigma \in H^0(X,L)_{u}$, and $f$ is the local equation of the zero set of $\sigma$.
\end{proof}

\begin{cor}\label{cor_lower-bound}
  Let $L$ be an ample line bundle on a manifold $X$ with an action of
  a torus $H$. Then the dimension of $X$ is bigger or equal to the
  maximal number of edges from any vertex of $\Delta(X,H,L)$.
\end{cor}

\subsection{Examples}\label{subsec_examples_of_torus_action}
Projective toric varieties with the action of the "big" torus or its downgrading
to a smaller subtorus are the most natural examples to illustrate the
objects that we have disussed so far. 
Then the relations between geometry of torus actions, including fixed points and orbits, and the properties of polytopes of sections, 
  are described in a lot of details in classical works and modern textbooks, such as \cite{cox_book}. 

Quadrics are the next natural class of examples.

\begin{ex}
  The torus $H=(\CC^*)^r$ with coordinates $(t_1,\dots,t_r)$ acts on
  $\CC^{2r+1}$:
  $$\begin{array}{r}(t_1,\dots,t_r)\cdot (z_0,z_1,z_2,\dots,z_{2r-1},z_{2r})=\\
    (z_0,t_1z_1,t_1^{-1}z_2,\dots,t_rz_{2r-1},t_r^{-1}z_{2r})\end{array}.$$
  The action of $H$ descends to an almost faithful action on
  the quadric $\cQ^{2r-1}\subset\PP^{2r}$ given by the equation
  $$z_0^2+z_1z_2+\cdots+z_{2r-1}z_{2r}=0.$$
  The action has $2r$ isolated fixed points.
  If $M$ is the lattice of characters of $H$ with the
  basis $\fromto{e_1}{e_r}$, then
    $$\Delta(\cQ^{2r-1},\cO(1),H)={\rm conv}(\pm e_1,\dots,\pm e_r)$$
  and we see that all fixed points are extremal. 
  We will see in Corollary~\ref{cor_torus-on-quadrics} that this shape of $\Delta$ with some additional properties essentially identifies a quadric hypersurface.
  The compass of $H$
  at the fixed point associated to the character $e_i$ consists of
  $-e_i$ and $\pm e_j-e_i$, for $j\ne i$. Note that the compass
  generates the semigroup $\RR_{\geq 0}(\Delta-e_i)\cap M$.

  The case of even dimensional quadric is similar. The torus $H$ acts
  on $\CC^{2r}$
    $$\begin{array}{r}(t_1,\dots,t_r)\cdot (z_1,z_2,\dots,z_{2r-1},z_{2r})=\\
      (t_1z_1,t_1^{-1}z_2,\dots,t_rz_{2r-1},t_r^{-1}z_{2r}).\end{array}$$
  The action of $H$ descends to an action of the quotient
  torus $H'=H/\langle(-1,\dots,-1)\rangle$ on the quadric
  $\cQ^{2r-2}\subset\PP^{2r-1}$ given by equation
    $$z_1z_2+\cdots+z_{2r-1}z_{2r}=0$$
    The action has $2r$ isolated fixed points.  As before, $M=\ZZ^r$
    generated by $e_i$'s and $M'\subset M$ is an index 2 sublattice of
    vectors $\sum_i a_ie_i$ such that $\sum_i a_i$ is even.
    Now 
    $$\Delta(\cQ^{2r-2},\cO(1),H)={\rm conv}(\pm e_1,\dots,\pm e_r)$$ 
    and the compass of $H$ at the fixed point associated to
    the character $e_i$ consists of $\pm e_j-e_i$, for $j\ne i$.  Note
    that the compass generates $\RR_{\geq 0}(\Delta-e_i)\cap M'$.
\end{ex}

In the present paper we will attempt to recognize homogeneous spaces $X$ with respect to some simple group $G$
only looking at the properties of the action of the maximal torus $H \subset G$ (or even a smaller torus) on $X$.
For the rest of this subsection we assume (as an example) that $X$ is homogeneous with respect to $G$ in order to illustrate some characteristic properties of this situtation.
Then the action of $G$ (perhaps after taking a finite cover) admits a unique linearization and
$\HH^0(X,L)$ is a $G$-module (representation) with associated set of
$H$-weights $\widetilde{\Gamma}(X,H,L)$. In fact, the case of quadrics
discussed above arises from the standard representations of $SO(2r)$
and $SO(2r+1)$.

The case of main interest for us is when $X$ is the closed orbit of
$G$ in $\PP(\g)$, $\g$ denoting the Lie algebra of $G$, and the group
acts via the adjoint action. Then we have the root decomposition
$\g=\h\oplus\bigoplus_{u\in\cR}\CC_u$, where $\h$ denotes the tangent
space of $H$ on which the adjoint action is trivial and $\cR$ is the
set of roots. If $L$ denotes the restriction of $\cO(1)$ from
$\PP(\g)$ to $X$ then $\Gamma(X,L,H)=\Delta(X,L,H)$ is the root
polytope $\Delta(G)$ which, by definition, is the convex hull of roots in the
$M_\RR$, with $M$ denoting the space of weights of $G$.

The following example illustrates the concept of downgrading and
restricting the action of the maximal torus in the group $SO(7)$ (or more precisely, $Spin(7)$, the double cover of $SO(7)$).

\begin{ex}\label{ex-downgradingB3}
The figure on the left hand side presents the root system $B_3$
inscribed in the unit cube. Long roots are denoted by $\bullet$ and
short roots by $\circ$. The figure on the right presents the
polytope $\Delta(B_3)$.
The only fixed points of the action of the $3$-dimensional maximal torus $H \subset Spin(7)$ on $X = Gr(\PP^1, \cQ^5)$ are extremal.
There are 12 fixed points associated to long roots which are vertices of $\Delta$.
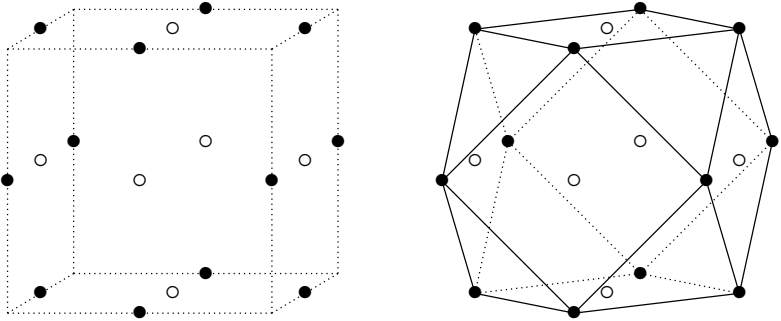
$$\begin{array}{ccc}
\begin{xy}<25pt,0pt>:
(2,2)*={}="a", (2,-2)*={}="b", (-2,-2)*={}="c", (-2,2)*={}="d",
(3,2.6)*={}="a1", (3,-1.4)*={}="b1", (-1,-1.4)*={}="c1", (-1,2.6)*={}="d1",
"a";"b" **@{.}, "a";"d" **@{.}, "a";"a1" **@{.}, 
"c";"b" **@{.}, "c";"d" **@{.}, "c";"c1" **@{.}, 
"b1";"b" **@{.}, "b1";"c1" **@{.}, "b1";"a1" **@{.}, 
"d1";"d" **@{.}, "d1";"c1" **@{.}, "d1";"a1" **@{.},
(0,0)*={\circ}="s1", (0.5,2.3)*={\circ}="s2", (1,0.6)*={\circ}="s3",
(0.5,-1.7)*={\circ}="s4", (2.5,0.3)*={\circ}="s5",
(-1.5,0.3)*={\circ}="s6",
(0,2)*={\bullet}="l1", (0,-2)*={\bullet}="l2", (2,0)*={\bullet}="l3",
(-2,0)*={\bullet}="l4", (1,2.6)*={\bullet}="l5",
(1,-1.4)*={\bullet}="l6", (3,0.6)*={\bullet}="l7",
(-1,0.6)*={\bullet}="l8", (-1.5,2.3)*={\bullet}="l9",
(-1.5,-1.7)*={\bullet}="l10", (2.5,2.3)*={\bullet}="l11",
(2.5,-1.7)*={\bullet}="l12",
\end{xy}
&\phantom{\longrightarrow}&
\begin{xy}<25pt,0pt>:
(2,2)*={}="a", (2,-2)*={}="b", (-2,-2)*={}="c", (-2,2)*={}="d",
(3,2.6)*={}="a1", (3,-1.4)*={}="b1", (-1,-1.4)*={}="c1", (-1,2.6)*={}="d1",
(0,0)*={\circ}="s1", (0.5,2.3)*={\circ}="s2", (1,0.6)*={\circ}="s3",
(0.5,-1.7)*={\circ}="s4", (2.5,0.3)*={\circ}="s5",
(-1.5,0.3)*={\circ}="s6",
(0,2)*={\bullet}="l1", (0,-2)*={\bullet}="l2", (2,0)*={\bullet}="l3",
(-2,0)*={\bullet}="l4", (1,2.6)*={\bullet}="l5",
(1,-1.4)*={\bullet}="l6", (3,0.6)*={\bullet}="l7",
(-1,0.6)*={\bullet}="l8", (-1.5,2.3)*={\bullet}="l9",
(-1.5,-1.7)*={\bullet}="l10", (2.5,2.3)*={\bullet}="l11",
(2.5,-1.7)*={\bullet}="l12",
"l1";"l3" **@{-}, "l3";"l2" **@{-}, "l2";"l4" **@{-}, "l4";"l1" **@{-},
"l5";"l7" **@{.}, "l7";"l6" **@{.}, "l6";"l8" **@{.}, "l8";"l5" **@{.},
"l9";"l1" **@{-}, "l9";"l4" **@{-}, "l9";"l5" **@{-}, "l9";"l8" **@{.},
"l11";"l1" **@{-}, "l11";"l5" **@{-}, "l11";"l3" **@{-}, "l11";"l7" **@{-},
"l12";"l3" **@{-}, "l12";"l7" **@{-}, "l12";"l2" **@{-}, "l12";"l6" **@{.},
"l10";"l2" **@{-}, "l10";"l6" **@{.}, "l10";"l4" **@{-}, "l10";"l8" **@{.}
\end{xy}
\end{array}
$$ 
The next figure shows a downgrading of the action to a 2-dimensional
torus: we take the planar projection along the diagonal in the
cube. The number of the fixed points does not change but now only six
of them are extremal.
$$
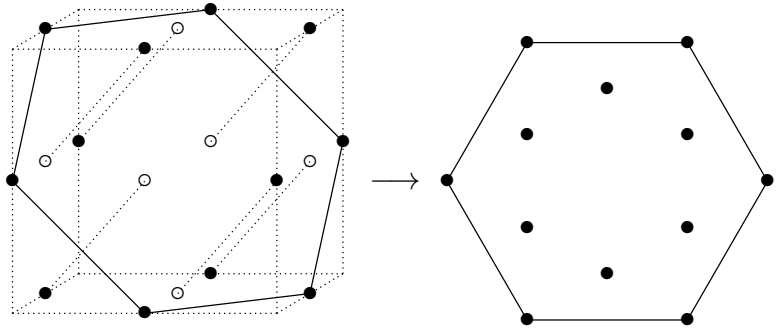
\begin{array}{ccc}
\begin{xy}<25pt,0pt>:
(2,2)*={}="a", (2,-2)*={}="b", (-2,-2)*={}="c", (-2,2)*={}="d",
(3,2.6)*={}="a1", (3,-1.4)*={}="b1", (-1,-1.4)*={}="c1", (-1,2.6)*={}="d1",
"a";"b" **@{.}, "a";"d" **@{.}, "a";"a1" **@{.}, 
"c";"b" **@{.}, "c";"d" **@{.}, "c";"c1" **@{.}, 
"b1";"b" **@{.}, "b1";"c1" **@{.}, "b1";"a1" **@{.}, 
"d1";"d" **@{.}, "d1";"c1" **@{.}, "d1";"a1" **@{.},
(0,0)*={\circ}="s1", (0.5,2.3)*={\circ}="s2", (1,0.6)*={\circ}="s3",
(0.5,-1.7)*={\circ}="s4", (2.5,0.3)*={\circ}="s5",
(-1.5,0.3)*={\circ}="s6",
(0,2)*={\bullet}="l1", (0,-2)*={\bullet}="l2", (2,0)*={\bullet}="l3",
(-2,0)*={\bullet}="l4", (1,2.6)*={\bullet}="l5",
(1,-1.4)*={\bullet}="l6", (3,0.6)*={\bullet}="l7",
(-1,0.6)*={\bullet}="l8", (-1.5,2.3)*={\bullet}="l9",
(-1.5,-1.7)*={\bullet}="l10", (2.5,2.3)*={\bullet}="l11",
(2.5,-1.7)*={\bullet}="l12",
"l4";"l9" **@{-}, "l9";"l5" **@{-}, "l5";"l7" **@{-}, "l7"; "l12" **@{-},
"l12";"l2" **@{-}, "l2";"l4" **@{-},
"s1";"l10" **@{.}, "s2";"l8" **@{.}, "s3";"l11" **@{.}, "s4";"l3" **@{.},
"s5"; "l6" **@{.}, "s6";"l1" **@{.},
\end{xy}
&\longrightarrow&
\begin{xy}<35pt,0pt>:
(0,1)*={\bullet}="b0",
(0.866,0.5)*={\bullet}="b1", 
(0.866,-0.5)*={\bullet}="b2",
(0,-1)*={\bullet}="b3", 
(-0.866,-0.5)*={\bullet}="b4",
(-0.866,0.5)*={\bullet}="b5",
(0.866,1.5)*={\bullet}="a0",
(1.732,0)*={\bullet}="a1", 
(0.866,-1.5)*={\bullet}="a2", 
(-0.866,-1.5)*={\bullet}="a3",
(-1.732,0)*={\bullet}="a4", 
(-0.866,1.5)*={\bullet}="a5",
"a0";"a1" **@{-}, "a1";"a2" **@{-}, "a2";"a3" **@{-}, "a3";"a4" **@{-}, "a4";"a5" **@{-},
"a5";"a0" **@{-},
\end{xy}
\end{array}
$$ 
If we further downgrade to 1-dimensional torus action then we get
fixed-point set consisting of 5 components: 2 isolated points, $\PP^1$
and two copies of $\PP^1\times\PP^1$. The restriction of $L$ to
$\PP^1$ is $\cO(2)$ and to each of $\PP^1\times\PP^1$ is $\cO(2,1)$.
$$
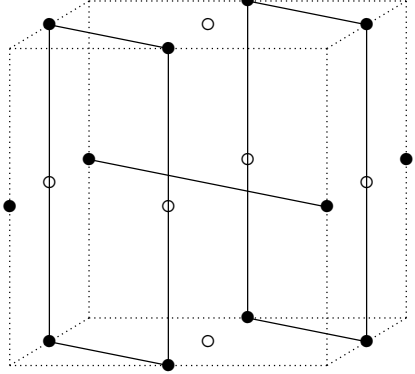
\begin{xy}<30pt,0pt>:
(2,2)*={}="a", (2,-2)*={}="b", (-2,-2)*={}="c", (-2,2)*={}="d",
(3,2.6)*={}="a1", (3,-1.4)*={}="b1", (-1,-1.4)*={}="c1", (-1,2.6)*={}="d1",
"a";"b" **@{.}, "a";"d" **@{.}, "a";"a1" **@{.}, 
"c";"b" **@{.}, "c";"d" **@{.}, "c";"c1" **@{.}, 
"b1";"b" **@{.}, "b1";"c1" **@{.}, "b1";"a1" **@{.}, 
"d1";"d" **@{.}, "d1";"c1" **@{.}, "d1";"a1" **@{.},
(0,0)*={\circ}="s1", (0.5,2.3)*={\circ}="s2", (1,0.6)*={\circ}="s3",
(0.5,-1.7)*={\circ}="s4", (2.5,0.3)*={\circ}="s5",
(-1.5,0.3)*={\circ}="s6",
(0,2)*={\bullet}="l1", (0,-2)*={\bullet}="l2", (2,0)*={\bullet}="l3",
(-2,0)*={\bullet}="l4", (1,2.6)*={\bullet}="l5",
(1,-1.4)*={\bullet}="l6", (3,0.6)*={\bullet}="l7",
(-1,0.6)*={\bullet}="l8", (-1.5,2.3)*={\bullet}="l9",
(-1.5,-1.7)*={\bullet}="l10", (2.5,2.3)*={\bullet}="l11",
(2.5,-1.7)*={\bullet}="l12",
"l9";"l1" **@{-}, "l1";"l2" **@{-}, "l2";"l10" **@{-}, "l10";"l9" **@{-},
"l5";"l11" **@{-}, "l11";"l12" **@{-}, "l12";"l6" **@{-}, "l6";"l5" **@{-},
"l3";"l8" **@{-}
\end{xy}
$$
\end{ex}

\subsection{Root polytopes of exceptional Lie groups of types \texorpdfstring{$E$}{E} and~\texorpdfstring{$F$}{F}}\label{sectErF4}

Consider the root systems of type $E_6$, $E_7$, $E_8$ or $F_4$.
Let $M$ be the weight lattice of the simple Lie group/algebra of such type
   and $\Delta( \cdot)$ be the convex hull of the roots from the system.
We have the following restrictions on the dimensions of projective manifolds,
   which admit an action of a torus $H$ corresponding to $M$,
   such that the action resembles the adjoint action of $H$ on the Lie algebra.

\begin{lemma}\label{lem_E-and-F}
Let $X$ be a manifold with an ample line bundle $L$. Suppose that $X$
admits an almost faithful action of a torus $H$ such that all extremal fixed-point components are isolated points.
\begin{enumerate}[leftmargin=*]
\item if $\Delta(X,L,H)=\Delta(E_6)$ then $\dim X\geq 20$,
\item if $\Delta(X,L,H)=\Delta(E_7)$ then $\dim X\geq 32$,
\item if $\Delta(X,L,H)=\Delta(E_8)$ then $\dim X\geq 56$,
\item if $\Delta(X,L,H)=\Delta(F_4)$ and the set
  $\widetilde{\Gamma}(X,L,H)$ of weights of the action of $H$ on
  $\HH^0(X,L)$ contains all roots of $F_4$ then $\dim X\geq 14$.
\end{enumerate}
\end{lemma}

\begin{proof}
We use the information from \cite[Tables 5--8]{Bourbaki}. 
By Corollary~\ref{cor_lower-bound} the bound on the dimension of $X$ for the cases
$E_i$ comes from the number of edges from a vertex of the respective
polytope $\Delta(E_i)$ which was calculated by {\tt magma} \cite{magma}.

In the case of $F_4$ the number of edges at a vertex is only $8$, so it is not sufficient to prove the claim as stated.
However, by Corollary~\ref{cor_weights-jets} we need more elements in the compass.
We use notation from \cite[Table 8]{Bourbaki} and consider a long root
$v=e_1+e_2$ which is a vertex of $\Delta(F_4)$. 
The edges of the cone $\RR_{\geq 0}\cdot(\Delta(F_4)-v)$ are spanned by $\pm e_j-e_i$, with
$i=1,2$ and $j=3,4$. Additionally, taking the short roots and subtracting
$v$ we get $-e_1$, $-e_2$ and $(\pm e_3\pm e_4 -e_1-e_2)/2$ which
makes the total of 14. We note that all these points lie on a
hyperplane $\{u: (e^*_1+e^*_2)(u)=-1\}$ hence they are the minimal set
which satisfies the condition for the compass from Corollary~\ref{cor_weights-jets}.
\end{proof}

\subsection{Comparing \texorpdfstring{$G$}{G}-varieties}\label{subsec_comparing_G_varieties}
One of the difficulties in proving the cases of the LeBrun-Salamon Conjecture (see Subsection~\ref{sec_LB_S_conjecture}) 
   is how to recognize a relatively abstract variety as a homogeneous space.
The combination of torus action tools we reviewed so far and the localization for equivariant K-theory
   discussed in Appendix~\ref{sect_appendix} 
   result in the following relatively easy to test criterion.
This arises as an application of an easy case of Corollary~\ref{cor_RRekw} phrased explicitly as Proposition~\ref{prop_if_combinatorial_data_agrees_then_representations_agree}.
A more refined use of this statement might lead to more powerful criteria, and should be a topic of further investigation.
For our applications in Section~\ref{sec_contact_simple} the following propositions are sufficient. 
Roughly, for a simple Lie group action on a projective manifold we look at the action of the maximal torus.
If it has only finitely many fixed points, and the compasses can be compared to compasses of a similar action on a homogeneous space, 
then the manifold is isomorphic to the homogeneous space.

The proposition below explains, that if the action of $H$ on $X$ and $X'$ and line bundles $L$ and $L'$ has the same combinatorial data, 
   then the spaces of sections are isomorphic as $H$-representations.
\begin{prop}\label{prop_if_combinatorial_data_agrees_then_representations_agree}
   Suppose $X$ and $X'$ are two projective manifolds with an action of a torus $H$ on both of them, 
     such that $X^H = \setfromto{y_1}{y_k}$ and $(X')^H= \setfromto{y'_1}{y'_k}$ consist of isolated points, and are of the same cardinality.
  Assume $L$ and $L'$ are two ample line bundles on $X$ and $X'$, respectively, both with linearizations $\mu$ and $\mu'$, and both with vanishing higher cohomologies (for instance, if $X$ and $X'$ are Fano).
  If for all $i \in \setfromto{1}{k}$ both conditions hold:
  \begin{itemize}
   \item   the compasses agree $\cC(y_i, X,H) = \cC(y'_i, X',H)$, and  
   \item   the characters agree $\mu(y_i) = \mu'(y'_i)$,  
  \end{itemize}
  then  $\HH^0(X,L) \simeq \HH^0(X',L')$ as representations of $H$.
\end{prop}

\begin{proof}
    The fixed points and related combinatorial data agree, so  Corollary~\ref{cor_RRekw} in Appendix~\ref{sect_appendix} 
       implies that also the equivariant Euler characteristic must agree $\chi^H(X,L) = \chi^H(X',L')$.
    Since higher cohomologies of both $L$ and $L'$ vanish, 
    the isomorphism classes of the representations $\HH^0(X,L)$ and $\HH^0(X',L')$
       are uniquely determined by the equivariant Euler characteristic.
\end{proof}

In the next proposition under much stronger assumptions we show that not only spaces of sections are isomorphic but actually also $X$ is isomorphic to $X'$.

\begin{prop}\label{prop_fixpt+compass-determine-homogeneous}
  Suppose that a semisimple group $G$ with a maximal torus $H$ acts on a
  projective  manifold $X$.
  Assume in addition that the restricted action of $H$ has only finitely
     many fixed points, $X^H=\{y_1,\dots,y_k\}$.
  Denote by $\cC_i=\{\nu_1(y_i),\dots,\nu_d(y_i)\}$ the compass of the
  action of $H$ at $y_i$. 
  Let $L$ be an ample line bundle on $X$ and $\mu$ a linearization of the action of $G$ on $L$. 
  If there is a
  $G$-homogeneous manifold $(X',L')$ with 
     $(X')^H=\{y'_1,\dots,y'_k\}$ and linearization $\mu'$ on $L'$
     such that $\cC(y_i,X',H')=\cC_i$ and $\mu'(y'_i)=\mu(y_i)$, 
     then there exists an isomorphism $(X',L')\iso (X,L)$. 
  \end{prop}

\begin{proof}
  By passing to a multiple of $L$ (and taking the same multiple of $L'$) we may assume that $L$ is very ample
  with no higher cohomology.  
  By Proposition~\ref{prop_if_combinatorial_data_agrees_then_representations_agree},
  $\HH^0(X,L)$ and $\HH^0(X',L')$ are isomorphic as $H$-modules. 
  Thus, by the representation theory, they are isomorphic as $G$-modules \cite[Thm~23.24(b)]{fultonharris}. 
  Being homogeneous, $X'\subset\PP(\HH^0(X',L')^*)$ is the
  unique closed orbit of the action of the semisimple group $G$ on
  $\PP(\HH^0(X',L')^*)$. Therefore, via the isomorphism
  $\HH^0(X',L')\iso\HH^0(X,L)$, the homogeneous space $X'$ is contained in the $G$-invariant
  $X\subset\PP(\HH^0(X,L)^*)$.
  Finally, since the dimensions are encoded in the number of elements of any compass,
     we must have $\dim X = \dim X' =d$ and thus $X=X'$.
\end{proof}

\begin{prop}\label{prop_character_is_a_Laurent_polynomial}
   Suppose a torus $H$ of rank $r$ acts on $(X,L)$ in such a way that $X^H = \setfromto{y_1}{y_k}$ is a finite set.
   Then the rational function 
   \[
      F=\sum_{i=1}^k\frac {t^{\mu(y_i)}}{\prod_{\nu\in \cC(y_i, X, H)} (1-t^{\nu})}
   \]
   in variables $\fromto{t_1}{t_r}$ is a  Laurent polynomial in these variables.
\end{prop}

\begin{proof}
Since $X$ is projective Corollary~\ref{cor_RRekw} of Appendix~\ref{sect_appendix} applies. 
The character of the finite dimensional $H$ representations $\HH^i(L)$ is necessarily a Laurent polynomial.
The function $F$ is equal to the equivariant Euler characteristic of $L$, that is, a sum with signs of the characters above,
  hence it also is a Laurent polynomial.
\end{proof}


\section{Applications of Bia{\l}ynicki-Birula decomposition}\label{sec_ABB_decomp}

In what follows we will reduce the action of a higher
dimensional torus $H$ to a suitably chosen 1-parameter subgroup of $H$. 
We note that by \cite[Lemma~2.3]{Bialynicki}, see also
Lemma~\ref{lem_reduction_of_compass}, a sufficiently general choice of the
1-parameter subgroup does not change the set of fixed points. 
However, we will frequently be interested in choosing a special 1-parameter subgroup whose fixed-point locus is larger than that of $H$,
  see also Lemma~\ref{lem_reduction_of_action}.

We use BB-decomposition as discussed in \cite{Carrell}, for the original
exposition see \cite{Bialynicki}. 

\begin{thm}\label{thm_ABB-decomposition}
  Suppose $\Lambda=\CC^*$ with a coordinate $t$ acts almost faithfully on $X$, and $L$  is an ample line bundle  with a linearization $\mu$ of the action of $\Lambda$.
  Take the decomposition of the fixed-point set
    into the connected components $X^\Lambda=Y_1\sqcup\cdots\sqcup Y_s$. 
  For every $Y_i$ by $\nu^\pm(Y_i)$ denote 
     the number of positive and negative
     characters in the compass $\cC(Y_i,X,\Lambda)$.  Let us
  define
      $$X_i^+=\{x\in X: \lim_{t\rightarrow 0}
      t\cdot x\in Y_i\}{\rm \ \ and \ \ }X_i^-=\{x\in X:
      \lim_{t\rightarrow \infty} t\cdot x\in Y_i\}.$$ 
  Then the following holds:
  \begin{itemize}
  \item $X^{\pm}_i$ are locally closed subsets and $X=X^+_1\sqcup\cdots\sqcup X^+_s=X^-_1\sqcup\cdots\sqcup X^-_s$,
  \item 
    There are unique $X_i^\pm$-cells associated to
    the largest/smallest value of $\mu(Y_i)$. Such cell is necessarily dense and the corresponding $Y_i$ is called
    the \emph{source} or \emph{sink}, respectively.
  \item The natural map $X^\pm_i\rightarrow Y_i$ is algebraic and is a $\CC^{\nu^\pm(Y_i)}$ fibration,
    which implies the decomposition in homology
$$H_m(X,\ZZ)=\bigoplus_iH_{m-2\nu^+(Y_i)}(Y_i,\ZZ)=\bigoplus_iH_{m-2\nu^-(Y_i)}(Y_i,\ZZ).$$
\end{itemize}
\end{thm}

We note that the source and sink are extremal (in the sense of Subsection~\ref{subsec_polytopes_of_sections_and_fixed_points}) fixed-point components
of the action of $\Lambda$.  Conversely, given an action of a higher
dimensional torus $H$ on $X$ and a vertex $v$ of $\Delta(X,L,H)$ we
can choose $\lambda\in \Hom(M,\ZZ)$ such that the affine hyperplane
$\lambda^\perp+v$ in $M_\RR$ meets $\Delta(X,L,H)$ at a vertex
$v$. Then the extremal fixed-point component of the action of $H$
associated to $v$ becomes a source or sink of the respective
$1$-dimensional subtorus $\Lambda\hookrightarrow H$. Thus we have the
following observation.
\begin{lemma}\label{lem_extremal=sink}
In the situation above the following holds:
\begin{enumerate}[leftmargin=*]
\item $Y_i\subset X^H$ is extremal if and only if there exists a
  1-parameter subgroup $\Lambda$ such that $Y_i$ is a source (or sink)
  of $\Lambda$,
\item For every vertex $v$ of  $\Delta(L)$ there is a unique component $Y$ of 
   $X^H$ with $\mu(Y,L, H) = v$.
   In particular, there is a bijection between extremal fixed-point components and the vertices of $\Delta(L)$. 
\end{enumerate}
\end{lemma}

The uniqueness of the extremal components implies a strengthening of the inclusion 
  in Lemma~\ref{lem_reduction_of_action}\ref{item_reduction_of_action_restricting_to_fix_points_of_subgroup}.
To explain this, for any face $\delta \subset \Delta(L)$ let $H_1\subset H$ be the subtorus corresponding to a projection $\pi \colon M \to M_1$
  which contracts $\delta$ to a point $v \in M_1$ and  $ v\notin \pi(\Delta(L) \setminus \delta)$.
We define $Y_{\delta}$ to be the (unique!) component  of $X^{H_1}$ corresponding to $v$.
The downgraded torus $H_2 = H/H_1$ acts on $Y_{\delta}$ as in Subsection~\ref{subsec_downgrading_and_reduction}.

\begin{lemma}\label{lem_equality_of_Delta_for_faces}
   For any face $\delta \subset \Delta(X, L, H)$ and with the notation as above we have 
   \[
      \widetilde{\Delta}(Y_{\delta}, L|_{Y_{\delta}}, H_2) = \widetilde{\Delta}(X, L, H)\cap \delta - w,
   \]
   where $w = \widetilde{\mu_1(Y_{\delta})}\in M$ is any lattice point shifting $\delta$ into $M_2\otimes \RR$,
      as in Lemma~\ref{lem_reduction_of_action}.
   In particular, 
   \[
      \Delta(Y_{\delta}, L|_{Y_{\delta}}, H_2) = \Delta(X, L, H)\cap \delta - w.
   \]
\end{lemma}
\begin{proof}
   The inclusion ``$\subset$'' is shown in Lemma~\ref{lem_reduction_of_action}\ref{item_reduction_of_action_restricting_to_fix_points_of_subgroup}.
   So pick $y \in X^{H}$ such that $\mu(y) \in \delta$.
   We have to show that $y\in Y_{\delta}$.
   Indeed, clearly, $y \in X^{H_1}$ and $\mu_{1}(y) = v$.
   Therefore by the uniqueness in Lemma~\ref{lem_extremal=sink}, we must have $y \in Y_{\delta}$.
\end{proof}

\subsection{BB-decomposition for \texorpdfstring{$\Pic X =\ZZ$}{Pic X = Z}}

In the case $\Pic X=\ZZ$ we have the following.
\begin{lemma}\label{lem_ABBdecomposition-extremal_cells}
  Let a $1$-parameter group $\Lambda$ act almost faithfully on $X$, as
  above. Assume in addition that $\Pic X =\ZZ\cdot L$ and $Y_0\subset
  X^\Lambda$ is the source of the action of $\Lambda$.
  Then $X$ is Fano and
\begin{enumerate}[leftmargin=*]
\item \label{item_extremal_cells_case_dim_positive}
    either $\dim Y_0>0$ and
 \begin{itemize}[leftmargin=*]
 \item $Y_0$ is Fano with $\Pic Y_0=\ZZ\cdot L$, and
 \item the complement of the BB-cell~$X_0^+$ is of codimension at least~$2$ in
   $X$, 
\end{itemize}
\item \label{item_extremal_cells_case_dim_0}
     or $Y_0$ is a point and
  \begin{itemize}[leftmargin=*]
  \item $X_0^{+}$ is an affine space with the linear action of
    $\Lambda$ associated to weights in
    $\cC(Y_0,X,\Lambda)=(\nu_1,\dots,\nu_d)$ with all $\nu_i$ negative.
  \item $D=X\setminus X_0^{+}$ is an irreducible divisor which
    is in the linear system $|L|$,
  \item there exists a unique fixed-point component $Y_1\subset
    X$ such that $\mu(Y_1)$ is minimal in
    $\widetilde{\Delta}(X,L,\Lambda) \setminus \mu(Y_0)$,
  \item the respective BB-cell $X_1^{+}$ associated to $Y_1$
    is dense in $D$.
  \end{itemize}
\end{enumerate}
\end{lemma}

\begin{proof}
  The manifold $X$ is generically covered by non-trivial orbits of the
  action of $\Lambda$ whose closures are rational curves, thus it is
  uniruled. Since $\Pic X=\ZZ$ this implies $X$ is Fano.
  If $\dim Y_0>0$, then $X\setminus X^{+}_0$ is of codimension $\geq 2$ in $X$:
  Indeed, its intersection with $Y_0$ is zero, and every effective divisor is ample, thus $X\setminus X^{+}_0$ cannot contain any divisor of $X$.
  
  The rest of the lemma follows by BB-decomposition, see Theorem~\ref{thm_ABB-decomposition}, or \cite[Thm~4.2 and
  Thm~4.4]{Carrell}. 
  To prove that $Y_0$ is Fano if its dimension is
  $\geq 1$ we use rational curves again.
  Namely, we can choose a
  rational curve in $X$ which does not meet $X\setminus X^{+}_0$. 
  Hence by the map $X_0^+\rightarrow Y_0$ we have a rational curve in $Y_0$.
\end{proof}

If $Y\subset X^H$ is an extremal component of the fixed-point locus 
  then by Corollary~\ref{cor_ideal_of_extremal_component} we have

\begin{equation}\label{equ_kernel_of_restriction_for_extremal_components}
\bigoplus_{u\ne \mu(Y)}\HH^0(X,L)_u = \ker\left(\HH^0(X,L)\lra\HH^0(Y,L|_Y)\right).
\end{equation}

\begin{lemma}\label{lem_ABB-surjectivity}
  Let $X$ be a projective manifold with an ample line bundle $L$ and
  an almost faithful action of an algebraic torus $H$.  Moreover 
  assume $\Pic X=\ZZ$.  If $Y\subset X^{H}$ is the extremal fixed-point
  component associated to a vertex $\mu(Y)\in \Delta(L)$ then the
  restriction $\HH^0(X,L)\lra \HH^0(Y,L|_{Y})$ is surjective and
  $\HH^0(X,L)_{\mu(Y)}=\HH^0(Y,L|_Y)$. Therefore
$$
  \cR(Y,L|_Y)\iso \cR(X,L)/\cJ_{\gamma_{\mu(Y)}},
$$
     where $\cJ_{\gamma_{\mu(Y)}}$ is the ideal defined in \eqref{equ_extremal_ideal}, 
     and $\gamma_{\mu(Y)}$ is the ray of the weight cone $\widehat{\Gamma}(L)$
     corresponding to the vertex of $\mu(Y)$ as in Subsection~\ref{subsec_polytopes_of_sections_and_fixed_points}.
\end{lemma}
\begin{proof}
  We choose a suitable 1-parameter subgroup of $H$ which does not change the extremal component $Y$. 
  Up to the inverse of the action, by Theorem~\ref{thm_ABB-decomposition},
     $Y$ is associated to the maximal cell $X_0^+$  which is dense in $X$ 
     and admits a fibration (retract) $p_0: X_0^+\ra Y\hookrightarrow X_0^+$. 
  
  If $Y$ is a point then our statement follows by Lemma~\ref{lem_ABBdecomposition-extremal_cells}\ref{item_extremal_cells_case_dim_0}.
  Indeed, there is a divisor in the linear system of $L$, which does not contain $Y$. Thus the restriction map in \eqref{equ_kernel_of_restriction_for_extremal_components} is surjective, and all the other claims follow.
  
  Also, if $\dim Y> 0$ then any 
  section in $\HH^0(Y,L|_{Y})$ lifts via $p_0: X_0^+\ra Y$ to a section $\HH^0(X_0^+,L)$,
    and the complement of $X_0^+$ is of codimension $\geq 2$ in $X$ by Lemma~\ref{lem_ABBdecomposition-extremal_cells}\ref{item_extremal_cells_case_dim_positive}.
  Therefore the section extends uniquely to $X$ and the map is $\HH^0(X,L)\lra\HH^0(Y,L|_Y)$ is surjective.
  Equation~\eqref{equ_kernel_of_restriction_for_extremal_components} shows that $\cR(Y,L|_Y)\iso \cR(X,L)/\cJ_{\gamma_{\mu(Y)}}$.
\end{proof}

\subsection{Extending divisors and equality of polytopes}

\begin{lemma}\label{lem_Fano_3_folds_H0L}
   Suppose $Y$ is a Fano manifold of dimension at most $3$ 
      and $L_Y$ is an ample line bundle on $Y$ such that 
      $\Pic Y = \ZZ \cdot L$.
   Then $h^0(L_Y) \ge 2$ or $Y$ is a point. 
\end{lemma}

\begin{proof}
   If $\dim Y \le 2$, then the statement is clear.
   If $\dim Y = 3$, then \cite[Prop.~(1.3)(ii)]{Iskovskih_Fano_3fold_1} 
     gives a formula for the Hilbert polynomial of $-K_Y$:
   \[
     h^0(\cO_Y(-m K_Y))= \frac{m(m+1)(2m+1)}{12} (-K_Y)^3 + 2m +1.
   \]
   Since $L= -\frac{1}{i}K_Y$ for a positive integer $i$,
     and by Serre vanishing the formula above also works for fractional $m$,
     we must have $h^0(L) >1$.
\end{proof}

We have the following immediate corollaries.
\begin{cor}\label{cor_Delta=Gamma-small_components}
  If $\Pic X=\ZZ\cdot L$ and every extremal component of $X^H$ is of
  dimension $\leq 3$ then $\Gamma(L)=\Delta(L)$.
\end{cor}
\begin{pf}
  If $\dim Y\le 3$, then $\HH^0(Y,L)\ne 0$ by Lemma~\ref{lem_Fano_3_folds_H0L}.
  Therefore, if $Y\subset X^H$ is an extremal fixed-point component then by Lemma~\ref{lem_ABB-surjectivity} there exists 
  a section in $\HH^0(X,L)$ which is an eigenvalue of the action of $H$.
  Moreover, it does not vanish identically on $Y$,
    and its weight is $\mu(Y)$.
  Thus $\mu(Y) \in \Gamma(L)$. 
  This shows $\Delta(L) \subset \Gamma(L)$, while the opposite inclusion is shown in general in
     Lemma~\ref{lem_compare_polytopes}\ref{item_compare_polytopes_Gamma_in_Delta}.
\end{pf}

\begin{prop}\label{prop_Delta=Gamma-small_corank}
  Suppose that a torus $H$ of rank $r$ acts almost faithfully on the projective
  manifold $X$ of dimension $d$ with $\Pic X=\ZZ\cdot L$. If $d\leq
  r+4$ then $\Gamma(L)=\Delta(L)$.
\end{prop}
\begin{proof}
  By Corollary~\ref{cor_Delta=Gamma-small_components} we will be done if we prove
  that every extremal component of $X^H$ is of dimension $\leq 3$.

  For a vertex $v\in\Delta(L)$ and we choose a flag of linear spaces
  $V_0=\{0\}\subsetneq\cdots\subsetneq V_{r-1}\subset M_\RR$ such that
  $(v+V_i)\cap\Delta(L)$ is an $i$-dimensional face of $\Delta(L)$. 
  The flag determines a sequence of subtori
  $$H=H^0\supsetneq H^1\supsetneq\cdots\supsetneq H^{r-1}$$
  where $H^i$ is of dimension $r-i$ associated to the quotient
  $M/(M\cap V_i)$.
  For each such subtorus $H_i$ one has an irreducible
  variety $Y_i$, where 
  $$Y=Y^0\subsetneq Y^1\subsetneq\cdots\subsetneq Y^{r-1},$$
  and $Y^i$ is an extremal component of $X^{H^i}$.
  Note that $Y^i\ne Y^{i+1}$ because $Y^{i+1}$ contains not only $Y^i$ but also
  some other extremal components of $X^H$ associated to vertices of
  $\Delta(L)$ which are in the face $V^{i+1}\cap \Delta(L)$ but are
  not in the face $V^i\cap\Delta(L)$. Therefore $\dim Y^{r-1}\geq \dim
  Y^0 + r-1$.

  By our assumptions $H^{r-1}$ acts almost faithfully on $X$ and if
  $Y^{r-1}\subset X^{H^{r-1}}$ is a divisor then it is in $|L|$ and
  the action of $H^{r-1}$ determines a non-trivial section of
  $TX\otimes L^{-1}$ and thus $X\iso \PP^d$, by a result of Wahl,
  \cite{Wahl}.

  Since the result is true for $X=\PP^d$ by Lemma~\ref{lem_compare_polytopes}\ref{item_compare_polytopes_base_point_free_Delta_eq_Gamma} 
     we may assume $\dim Y^{r-1}\leq d-2$ and therefore
  $$\dim Y^0\leq \dim Y^{r-1} - r + 1\leq d-r-1$$
  and we are in situation of Proposition~\ref{prop_Delta=Gamma-small_corank}.
\end{proof}

We also note that the property $\Gamma=\Delta$ propagates to extremal components corresponding to subfaces.
\begin{lemma}\label{lem_Gamma_eq_Delta_propagates_to_faces}
  Suppose $\Pic X = \ZZ$, $L$ is an ample line bundle on $X$, and a torus $H$ acts on $(X,L)$
    in such a way that $\Gamma(X,H, L) = \Delta(X,H,L)$.
  Let $\delta\subset \Delta(X,H,L) $ be a proper face, $H' \subset H$ 
    a subtorus corresponding to a projection $M \to M'$ that contracts $\delta$ to a point $v\in M'$ and no other face is contracted to $v$.
  Let $Y_{\delta}$ be the extremal component of $X^{H'}$ corresponding to $v$.
  Then also for the action of $H/H'$ on $Y_{\delta}$ the section and fixed-point polytopes are equal:
   \[
      \Delta(Y_{\delta}, H/H', L|_{Y_{\delta}}) = \Gamma(Y_{\delta}, H/H', L|_{Y_{\delta}}) = \delta - w, 
   \]
   where $w$ is any lattice point in the affine span of $\delta$ 
      (the choice of the shift by $w$ corresponds to the choice of the linearization as in Lemma~\ref{lem_reduction_of_action}).
\end{lemma}

\begin{proof}
  Suppose $W\subset M\otimes \RR$ denotes the affine span of $\delta$ and set $Y:=Y_{\delta}$ for brevity.
  We have the following inclusions:
  \begin{align*}
    \Delta(Y, L|_{Y}, H/H') & 
        \stackrel{\text{Lem.~\ref{lem_reduction_of_action}\ref{item_reduction_of_action_restricting_to_fix_points_of_subgroup}}}{\subset}
        \delta-w\\
     &= 
         \Gamma(X, L, H) \cap W  - w  \\
     & = \conv (\widetilde{\Gamma}(X, L, H) \cap W  - w) \\
     &\stackrel{\text{Lem.~\ref{lem_ABB-surjectivity}}}{=} \conv (\widetilde{\Gamma}(Y, L|_{Y}, H/H'))\\
     &= \Gamma(Y, L|_{Y}, H/H')\\ 
     &\stackrel{\text{Lem.~\ref{lem_compare_polytopes}\ref{item_compare_polytopes_Gamma_in_Delta}}}{\subset}
     \Delta(Y, L|_{Y}, H/H').
  \end{align*}
  Therefore all the inclusions are equalities, as claimed in the lemma.
\end{proof}

\subsection{Fano manifolds, projective spaces and quadrics}\label{subsec_Fano_Pd_Qd}
Suppose $X$ is a projective manifold with an almost faithful action of a nontrivial torus $H$.
Then the line bundle $\det TX$ has a natural linearization coming from that of $TX$. 
Recall, that the \emph{index} of a Fano manifold is the maximal positive integer,
   such that $\det TX \iso L^{\otimes\iota}$ for an ample line bundle $L$.

\begin{lemma}\label{Fano-case}
Let $\mu$ be the natural linearization of $L^{\otimes\iota}= \det
TX$. Then for every fixed-point component $Y_i\subset X^H$ we have
$$\mu_{L^{\otimes\iota}}(Y_i)=-\sum_{\nu_j\in\cC(Y_i)} \nu_j$$
\end{lemma}
\begin{proof}
For every $y\in Y_i$ the character of the action of $H$ on $\det T_yX$
is the sum of weights of the action on $T_yX$. Hence the claim.
\end{proof}

\begin{prop}\label{prop_smallDelta2}
  Suppose that a projective manifold $X$ of dimension $d$ with $\Pic
  X \simeq \ZZ$ admits a nontrivial action of a 1-dimensional torus $H$. 
  Assume
  that $\Delta(L)=[0,2]\subset M_\RR=\RR$ and the extremal fixed-point
  components are of dimension zero. Then one of the following holds:
  \begin{enumerate}[leftmargin=*]
  \item \label{item_small_Delta2_P1}
         $d=1$ and $(X,L)$ is either $(\PP^1,\cO(1))$ or
    $(\PP^1,\cO(2))$,
  \item \label{item_small_Delta2_Pd} 
        $d \geq 2$ and $(X,L)=(\PP^d,\cO(1))$ and in some coordinates
    $[z_0,\dots,z_d]$ the action of $H$ has weights $(0,1,\dots,1,2)$, 
  \item \label{item_small_Delta2_Qd}
        $d \geq 3$ and $(X,L)=(\cQ^d,\cO(1))$ and for some equivariant
    embedding $\cQ^d \hookrightarrow \PP^{d+1}$ in which $\cQ^d$ has
    equation $z_0z_{d+1}+z_1z_d+\cdots=0$ the action of $H$ on
    $\PP^{d+1}$ has weights $(0,1,\dots,1,2)$.
 \end{enumerate}
\end{prop}

\begin{proof}
$X$ is Fano by Lemma~\ref{lem_ABBdecomposition-extremal_cells}.
Let $\mu$ denote the linearization of $L$ such that
$\Delta(L,\mu)=[0,2]$ and by $y_0, y_2\in X^H$ denote the extremal
fixed points associated to, respective endpoints of $\Delta(L,\mu)$. If
$\mu'$ is the natural linearization of $\det TX = L^{\otimes\iota}$
then
$$
\begin{array}{ccc} 
  \mu'(y_0)=-\sum_{\nu_j\in\cC(y_0)}\nu_j\leq
  -d&&\mu'(y_2)=-\sum_{\nu_j\in\cC(y_2)}\nu_j\geq d
\end{array}
$$ 
by  
Lemma~\ref{Fano-case}.
In particular, $\mu'(y_2) - \mu'(y_0) \ge 2d$.
On the other hand
$\Delta(L^{\otimes\iota},\mu_{L^{\otimes\iota}})=[0,2\iota]$.
Since these two linearizations differ by a constant in $M$, it
follows that $\iota\geq d$.
By the Kobayashi and Ochiai characterization \cite{kobayashi_ochiai} 
  the pair $(X,L)$ is either $(\PP^d,\cO(1))$ or $(\cQ^d,\cO(1))$.
The rest of the claim follows 
by a straightforward verification.
\end{proof}

\begin{cor}\label{cor_smallDelta2-compass}
  In the situation of Proposition~\ref{prop_smallDelta2} the compass at the
  extremal fixed point associated to 0 is $(1^{d-1},2)$ in the case \ref{item_small_Delta2_Pd}
  and $(1^d)$ in the case  \ref{item_small_Delta2_Qd}.
\end{cor}
The following special case will be relevant to our investigations in Section~\ref{sec_contact_simple}.
\begin{cor}\label{cor_torus-on-quadrics}
  Let $H$ be a torus of rank $r$ with a basis $x_1,\dots, x_r$ of the
  lattice of characters $M\iso \ZZ^r$. Suppose that
  $H$ acts almost faithfully on a projective manifold $X$ of dimension at least $2$, such
  that $\Pic X=\ZZ\cdot L$, and all extremal fixed-point components of
  $X^H$ are isolated points. If $\Delta(X,H,L)= \conv(\fromto{x_1, -x_1}{x_r, -x_r })$
     and the compass of the action
     of $H$ at any extremal fixed point corresponding to the vertex $\pm x_i$ of $\Delta(X,H,L)$ 
     does not contain $\mp 2 x_i $, then $X\iso\cQ^d$ where $d=2r+\dim(\HH^0(X,L)_0)-2$.
\end{cor}
\begin{proof}
  We note that for every $i$ the projection $M\rightarrow \ZZ\cdot x_i$
  yields the situation as in Proposition~\ref{prop_smallDelta2}, hence the claim
  follows from Corollary~\ref{cor_smallDelta2-compass}
\end{proof}

The last result in this sub-section seems to be known to experts 
but we were not able to find a proper reference for it. 
\begin{lemma}\label{lem_2-fixpts-comp}
  Let $X$ be a projective manifold with an action of 1-dimensional
  torus $H$. If the fixed-point set has two components $X^H=Y_0\sqcup
  Y_1$ and $Y_0$ is a point then $X\iso \PP^d$ and $Y_1\iso
  \PP^{d-1}$.
\end{lemma}
\begin{proof}
  By the Bia{\l}ynicki-Birula decomposition of cohomology in Theorem~\ref{thm_ABB-decomposition},
     we see that $\dim Y_1=\dim X-1$ and the
  second Betti number of $X$ is 1. Hence $Y_1$ is an ample divisor and
  the vector field tangent to the action of $H$ vanishes along
  it. Thus the result follows from \cite{Wahl}.
\end{proof}

\begin{ex}\label{ex_Hirzebruch_surface}
   Assume $X$ is a smooth connected projective surface with an action of $1$-dimensional torus $\Lambda$ 
     and an ample linearized line bundle $L$.
   Suppose that $\Delta(X,L, \Lambda)$ is an interval $[0,3] \subset M_{\RR} $ and $X^{\Lambda}$ consists of isolated fixed points only.
   Moreover, suppose $\Gamma(X, L, \Lambda) = \Delta(X,L, \Lambda)$ and for all  $y \in X^{\Lambda}$
   \[
     \cC(y, X, \Lambda) = \begin{cases}
                           (1^2) &\text{if } \mu(y) =0,\\
                           (-1,1) &\text{if } \mu(y) \in {1, 2}, \text{ and}\\
                           ((-1)^2) &\text{if } \mu(y) =3,\\
                           \end{cases}
   \]
   Then $X^{\Lambda}$ consist of $4$ points, one for each integral point of $\Delta(X,L, \Lambda)$.
\end{ex}
We provide two proofs of the example, illustrating the strength of various techniques presented in
  Sections~\ref{sec_torus_action}, \ref{sec_ABB_decomp} and Appendix~\ref{sect_appendix}.
A reader who is not fond of the geometric analysis presented in the first proof will certainly
  appreciate a short argument based on localization theorem which however does not explain the geometry 
  as thoroughly as the first argument does. 

\begin{proof}[Proof using BB-decomposition]
   By Lemma~\ref{lem_extremal=sink} there is a unique extremal point $y_i$ for each vertex $i \in \set{0,3}$.
   Note that these are not the only fixed points by Theorem~\ref{thm_ABB-decomposition},
      as then $X$ would be a one point compactification of $X_0^+ \simeq \CC^2$, which is impossible.
   Without loss of generality, suppose there is $y_2 \in X^{\Lambda}$ with $\mu(y_2)=2$.
   
   Since $\Gamma(X, L, \Lambda) = \Delta(X,L, \Lambda)$ there are sections $\sigma_0$ and $\sigma_3$ of $L$ which have weights $0$ and $3$ respectively.
   Let $D_0$ and $D_3$ be the corresponding divisors.
   Consider the local defining equations $f_{0,y}$ and $f_{3,y}$ of $D_0$ and $D_3$ at a fixed point $y$, as in Lemma~\ref{lem_local_description_inv_divisor}.
   The weights of $f_{0,y_0}$ and $f_{3,y_3} $ are $0$, while all the weights of the remaining $f_{i,y}$  are nonzero and equal to $i -\mu(y)$ (for $i=0$ or $i=3$).
   In particular, the local expression of $\sigma_0$ near a fixed point $y\ne y_0$ is equal to $f_{0,y}$, which is homogeneous and nonconstant, thus vanishing at $y$.
   Therefore $D_0$ contains all fixed points except $y_0$, 
      and analogously $D_3$ contains all fixed points except $y_3$.      
   Moreover, $D_0$ near $y_3$ is defined by $f_{0, y_3}$, which is a homogeneous polynomial of degree $-3$ 
     in two variables both of degree $-1$.
   Thus $D_0$ is a union of $3$ orbits of $\Lambda$ (perhaps with multiplicities).
   
   Now consider any point $y'_2 \in X^{\Lambda}$ such that $\mu(y'_2)=2$ (by our choice, there is at least one such point).
   By the local coordinates as in Lemma~\ref{lem_local_description_inv_divisor}, 
   the neighbourhood of $y'_2$ is isomorphic to (up to an \'etale cover) $\Spec \CC[x_1, x_2]$ 
     where the action of $\Lambda$ has weights $1$ and $-1$.
     In particular, there are only two one dimensional orbits in $X$,
     whose closures $C_{y'_2, -1}$, and $C_{y'_2, 1}$ contain $y'_2$,
      and one of them, say $C_{y'_2, 1}$, contains also $y_3$ in its closure.
   The divisor $D_0$ near $y'_2$ is defined by $f_{0,y'_2} \in \CC[x_1, x_2]$ of degree $-2$, 
      thus $D_0$ is equal near $y'_2$ to $a C_{y'_2, -1} + (a+2) C_{y'_2, 1}$ for some nonnegative integer $a$.
   In particular, looking again near $y_3$  (where the multiplicity is $3$ as shown above) we see that $a\le 1$,
      and that $D_0$ near $y_3$ has a component of multiplicity at least $2$.
   Therefore, there can be at most one point $y'_2$ with $\mu(y'_2)=2$, namely $y_2$, 
      as any other such fixed point would lead to another component of multiplicity $2$ near $y_3$, 
      contradicting its multiplicity $3$ at $y_3$.

   If $a=0$, then $D_0$ has another component (closure of an orbit) with multiplicity $1$ near $y_3$.
   The other point in the closure of the orbit is a fixed point, which is not equal to $y_0$ (because $y_0 \notin D_0$), nor $y_2$ 
    (because, $C_{y_2, 1}$ is the only orbit, whose closure contains both $y_2$ and $y_3$), nor $y_3$ (by the local description of the action as in Equation~\eqref{equ_local-action}).
   Thus there must exist a point $y_1 \in X^{\Lambda}$ with $\mu(y_1)=1$.
   Similarly, if $a=1$, then $C_{y_2, -1}$ is contained in $D$ with multiplicity $1$ and the other end of  $C_{y_2, -1}$ must be a point  $y_1$ as above.
   In both cases, we can swap the roles of $0$ and $3$, and argue in the same way, to show that $y_1$ is a unique point with $\mu(y_1)=1$, which concludes the proof.
\end{proof}

\begin{proof}[Proof using Localisation Theorem]
    There are unique fixed points corresponding to $0$ and $3$ by Lemma~\ref{lem_extremal=sink}.
    Suppose that there are $a$ fixed points corresponding to $1$ and $b$ fixed points corresponding to $2$.
    Then the rational function from Proposition~\ref{prop_character_is_a_Laurent_polynomial} gives:
    \begin{align*}
        F &= \frac{1}{(1-t)^2} + \frac{a t}{(1-t^{-1})(1-t)} + \frac{b t^2}{(1-t^{-1})(1-t)} + \frac{t^3}{(1-t^{-1})^2}\\
          &= \frac{1 - at^2  - bt^3 + t^5 }{(t-1)^2}
    \end{align*}
    Thus by the same proposition, $F$ must be a Laurent polynomial, in particular, the numerator must have a double root at $1$.
    It is straightforward to check that this happens if and only if $a=b=1$.
\end{proof}

We remark that in the situation of Example~\ref{ex_Hirzebruch_surface}, the surface $X$ must be isomorphic either to $\PP^1 \times \PP^1$ or the Hirzebruch surface $\FF_{2}$.
Indeed, by the cohomology description in Theorem~\ref{thm_ABB-decomposition} the variety $X$ is a rational surface with $\Pic X\simeq \ZZ^2$, 
  thus it is a Hirzebruch surface $\FF_a$ for some $a$. 
Using the toric geometry one shows that $a=0$ or $a=2$.
We leave out the details as we are not going to use this observation.

\section{Torus actions on contact manifolds}\label{sec_contact_general}
Throughout this section we assume that $X$ is a contact Fano ma\-ni\-fold of
dimension $d=2n+1$,  with a contact form $\theta\in
\HH^0(X,\Omega_X(L))$ which detemines the exact sequence of vector bundles on
$X$:
\begin{equation*}
  0\lra F\lra TX \lra L\lra 0.
\end{equation*}
The form $d\theta|_{F}\in \HH^0(X,\Wedge{2} F\otimes L)$ determines a nondegenerate
bilinear skew symmetric pairing $F\times F \lra L$.
See \cite[Sect.~E.3 and~Chapt.~C]{jabu_dr} and references therein for introduction and more details about contact manifolds, 
   the \emph{contact distribution $F$}, the \emph{contact form $\theta$}, the \emph{contact line bundle}, and the skew-pairing $d\theta|_{F}$.
   
By a slight abuse of standard notation we will assume that any smooth projective curve is a contact manifold as well. In this case $n=1$, $F=0 \subset TX =L$, and $X$ is Fano if and only if $X \simeq \PP^1$.
This is need for a uniform statement of Corollary~\ref{cor_contact-fixedpt-contact} and consequently also in multiple proofs in Section~\ref{sec_contact_simple}, where the case of $n=0$ is critical.

We will use the set up introduced in \cite{Beauville} or \cite[Sect.~E.3.2]{jabu_dr}.
We consider a connected nontrivial reductive group $G$ of automorphims of $X$.
In particular, we will consider a maximal torus
$H\subset G$. The tangent action of $G$ and of $H$ on $TX$ induces a
canonical linearization $\mu_0$ of $L$.

We remark that unless $X \simeq \PP^{2n+1}$, the action of $G$ always preserves the contact distribution $F$, see \cite[Cor.~4.5]{kebekus_lines1}, since the contact structure on $X$ is unique.
Our main interest is in the case $\Pic X= \ZZ \cdot L$, which in particular excludes the case of $\PP^{2n+1}$.
In fact, from Subsection~\ref{subsect_adjoint_representation} onwards we will suppose that $\Pic X= \ZZ \cdot L$,
   but in \ref{subsect_fixed_pts_contact} we need this more general setup for applications in Section~\ref{sec_contact_simple}.
There we will also consider contact submanifolds of $X$, which arise as fixed points of a torus action 
   (see Corollary~\ref{cor_contact-fixedpt-contact}) 
   and they do not necessarily satisfy $\Pic X= \ZZ \cdot L$.

\subsection{Fixed points components of 
contact automorphisms}\label{subsect_fixed_pts_contact}

\begin{lemma}\label{lem_t-action-contact}
  Assume that $H$ is a torus acting on a contact manifold $X$
  and preserving the contact structure. For a fixed point $y\in X^H$
  we consider $\nu_i(y)\in M$, the weights of the action of $H$ on
  $T_y^*X$. Then after a renumbering, for $i=1,\dots,n$, the weights
  satisfy the following equalities in $M$: $$\nu_i+\nu_{i+n}=\nu_0$$
  and $\mu_0(y)=-\nu_0$ is the weight of the action of $H$ on $L_y$.
\end{lemma}
\begin{proof}
The claim follows from the duality $F\simeq F^* \otimes L$ defined by the nondegenerate form $d\theta|_{F}$.
\end{proof}
\begin{cor}
  In the situation of the previous lemma suppose in addition that $Y\subset X^H$
  is an extremal fixed-point component. 
  If the action of $H$ on $X$ is nontrivial, then $\mu_0(Y)\ne 0$.
\end{cor} 
\begin{proof}
  We can reduce the action to a 1-dimensional torus for which $Y$ is the
  source (or sink) and thus all elements in the compass of $Y$ are
  negative (or positive). Thus $\mu_0(Y)=0$ contradicts the previous
  lemma.
\end{proof}
Now, in view of Lemma~\ref{lem_t-action-contact}, as immediate
consequences we get the following description of fixed-point components
$Y\subset X^H$ depending on whether $\mu_0(Y)= 0$ or not.
\begin{cor}\label{cor_contact-fixedpt-isotropic}
  In the situation of Lemma~\ref{lem_t-action-contact} assume that
  $Y\subset X^H$ is a component such that $\mu_0(Y)\ne 0$. Then $Y$ is
  isotropic with respect to $\theta$, that is
    $$TY\hookrightarrow F_{|Y}\hookrightarrow TX_{|Y}$$ 
 and the form $d\theta$ is zero on $TY$. In particular, $\dim Y\leq
 n$ and, in fact, $\dim Y+1$ is equal to the multiplicity of the weight
 $-\mu_0(Y)$ in the compass $\cC(Y,X,H)$.
\end{cor}
\begin{cor}\label{cor_contact-fixedpt-contact}
  In the situation of Lemma~\ref{lem_t-action-contact} assume that
  $Y\subset X^H$ is a component such that $\mu_0(Y)=0$.
  Then $Y$ is a contact manifold, with a contact form defined by the 
  restriction of $\theta$ 
  (in particular, $L|_{Y}$ is the contact line bundle on $Y$, that is the quotient of $TY$ by the contact distribution of $Y$ is $L|_Y$)
  and $d\theta|_{F}$ induces bilinear pairing on the normal bundle
  $N_{Y/X}\times N_{Y/X}\ra L|_Y$.
\end{cor}

\subsection{Adjoint representation}\label{subsect_adjoint_representation}
In addition to the assumptions listed at the beginning of this section, from now on we suppose 
$\Pic X= \ZZ \cdot L$.
Let $\g\subset \HH^0(X,TX)$ denote the Lie algebra of vector fields
tangent to the action of $G$. 
\begin{lemma}\label{lem_contact-canonical-linearization}
  The canonical $G$-linearization $\mu_0$ of $L$ makes $\HH^0(X,L)$ into a representation of $G$ isomorphic to $\HH^0(X, TX)$.
  If in addition $G=\Aut(X)$, then the set and
  multiplicities of weights in $\widetilde{\Gamma}(X,L,H,\mu_0)$
  coincide with the set and multiplicities of weights of the adjoint
  representation of $G$.
\end{lemma}
\begin{proof}
  The first statement is discussed in 
  \cite[Sect.~1.2]{Beauville} or in \cite[Thm~E.13, Cor.~E.14]{jabu_dr}. 
  The second claim follows, as then $\HH^0(X, TX) = \g$.
\end{proof}

\begin{lemma}\label{lem_contact->Gsemisimple}
  Let $X$ be a projective contact manifold with $\Pic X = \ZZ \cdot L$ as above.
  Suppose that $G=\Aut(X)$ is a reductive group with a maximal torus $H$.
  If $\Gamma(X,L,H,\mu_0)=\Delta(X,L,H,\mu_0)$ then $G$ is semisimple.
\end{lemma}
\begin{proof}
  By Lemma~\ref{lem_contact-canonical-linearization} the polytope $\Gamma(L)$ is the convex
  hull of the weights of the adjoint action of $H$ on $\g$. 
  Hence $\Gamma(L)$ is of maximal dimension if and only if $G$ is
  semisimple. 
  On the other hand, since the action of $H$ is almost faithful, 
    the polytope $\Delta(L)$ is of maximal dimension by Corollary~\ref{cor_Delta_has_maximal_dim}.
  Therefore $G$ is semisimple.
\end{proof}

\subsection{Semisimple groups of automorphisms of high rank}
As before $X$ is a contact Fano manifold of dimension $2n+1$ with 
$\Pic X=\ZZ \cdot L$ and $-K_X=(n+1)L$. By $G$ we denote the automorphism group of $X$ and we assume it is reductive of rank $r$.
\begin{lemma}\label{lem_contact-Gamma=Delta}
  Le $X$ be a contact Fano manifold of dimension $2n+1$ and $\Pic
  X=\ZZ\cdot L$. Suppose that the group $G$ of automorphisms
  of $X$ is reductive with a maximal torus $H$ of rank $r$. If
  $r+2\geq n$ then $\Delta(X,L,H,\mu_0) = \Gamma(X,L,H,\mu_0)$ and all
  extremal fixed-point components of $X^H$ are isolated points.
\end{lemma}
\begin{proof}
  The arguments are similar to the proof of Proposition~\ref{prop_Delta=Gamma-small_corank}.
  By Corollary~\ref{cor_Delta=Gamma-small_components} the first statement will follow
  if we prove that every extremal component $Y^0$ of $X^H$ is of
  dimension $\leq 3$. Indeed, as in the proof of
  Proposition~\ref{prop_Delta=Gamma-small_corank} we construct a flag of submanifolds
  $Y^0\subsetneq Y^1\subsetneq\cdots\subsetneq Y^{r-1}$ and by
  Corollary~\ref{cor_contact-fixedpt-isotropic} we have $\dim Y^{r-1}\leq n$ hence
  $\dim Y^0\leq n - (r-1)\leq 3$. 

  Now we know that $\dim Y^0\leq 3$ and $Y^0$ is Fano 
  with $\Pic Y^0=\ZZ\cdot L_{|Y^0}$ by Lemma~\ref{lem_ABBdecomposition-extremal_cells}.
  Moreover, by Lemma~\ref{lem_ABB-surjectivity} we know
  that $\dim\HH^0(Y^0,L_{|Y^0})$ is the same as the multiplicity of
  the respective root in the weights of the adjoint representation of
  $G$, hence it is one. Finally,  
  $\dim\HH^0(Y^0,L_{|Y^0})=1$ implies that $Y^0$ is a point by Lemma~\ref{lem_Fano_3_folds_H0L}.
\end{proof}

\begin{prop}\label{prop_contact->Gsimple} 
  Let $X$ be a contact Fano manifold of dimension $2n+1$ and $\Pic
  X=\ZZ\cdot L$. Suppose that the group $G$ of automorphisms
  of $X$ is reductive of rank $r\geq n-2$. Then $G$ is simple.
\end{prop}

\begin{proof}
  Let $H$ be the maximal torus of $G$.
  By Lemma~\ref{lem_contact-Gamma=Delta} the polytopes $\Gamma(X,L,H,\mu_0)$ and $\Delta(X,L,H,\mu_0)$
    are equal. 
  Lemma~\ref{lem_contact->Gsemisimple} thus implies that $G$ is semisimple.
  Up to a finite cover we can write a decomposition of $G$ into
  simple factors $G=G_1 \times \cdots \times G_s$.
  Correspondingly, the maximal torus and character lattice decompose as 
     $H=H_1\times\cdots\times H_s$ and $M=M_1 \oplus \cdots\oplus M_s$.
  Here $H_i$ is a maximal torus in $G_i$ and $M_i$ is the lattice of weights of $G_i$, hence also the lattice of characters of $H_i$. 
  If $R_i$ is a root system of $G_i$ in $M_i \hookrightarrow M$
  then $\Delta(L)$ is the convex hull of $\bigcup R_i \subset M$.

  Suppose that $s>1$. For $i=1,\ 2$ take $u_i\in R_i\subset M_i
  \hookrightarrow M$ which is the vertex of the root polytope
  $\Delta(G_i)\subset M_i$. Then $\RR_{\geq 0}(u_2-u_1)$ is a ray of
  the cone $\RR_{\geq 0}(\Delta(X,L)-u_1)$. The element $(u_2-u_1)$ is
  primitive in the semigroup $\RR_{\geq 0}(u_2-u_1)\cap M$ unless both 
  $G_1$ and $G_2$ are groups of type $A_1$ or $C$ and then $\RR_{\geq
    0}(u_2-u_1)\cap M$ is generated by $(u_2-u_1)/2$. 

  In the former case, by Corollary~\ref{cor_weights-jets} the element $(u_2-u_1)$ is
  in the compass of the extremal fixed-point component $Y_1$
  associated to $u_1$. Then, however by Lemma~\ref{lem_t-action-contact},
  $-u_2=-u_1-(u_2-u_1)$ should be contained in $\cC(Y_1,X,H)$ as well,
  hence by Corollary~\ref{cor_compass_and_fixpt} we get $u_1-\lambda u_2\in\Delta(X,L)$ for a positive integer $\lambda$, a
  contradiction.

  If both $G_1$ and $G_2$ are of type $A_1$ or $C$ then we consider the
  projection of $M_R$ along the direction of the edge spanned by $u_1$
  and $u_2$. That is we consider the reduction of the action of the
  torus $H$ to $H'$ associated to the
  $\ZZ\cdot(u_2-u_1)/2\hookrightarrow M\ra M'$. 
  Let $Y'$ be the extremal component of $X^{H'}$ corresponding to the vertex, 
    which is the image of the edge spanned by $u_1$ and $u_2$.
  The manifold $Y'$ admits an almost faithful $\CC^* = H/H'$ action with $\Delta(Y', L|_{Y'}, \CC^*)$
     equal to an interval of length $2$ by 
     Lemma~\ref{lem_reduction_of_action}\ref{item_reduction_of_action_restricting_to_fix_points_of_subgroup}.
  Since $\dim Y'>0$, by Lemma~\ref{lem_ABBdecomposition-extremal_cells}\ref{item_extremal_cells_case_dim_positive} 
     we also have $\Pic Y' = \ZZ\cdot L|_{Y'}$.
  Moreover, the extremal components of $(Y')^{\CC^*}$ are the extremal components of $X^{H}$ corresponding to vertices $u_1$ and $u_2$,
  thus they are points.
  Thus $Y'$ with the action of $\CC^*$ satisfies the assumptions of Proposition~\ref{prop_smallDelta2},
     and the pair $(Y', L|_{Y'})$ is isomorphic either to $(\PP^d, \cO_{\PP^d}(1))$ for $d\ge 1$
     or to $(\cQ^d, \cO_{\cQ^d}(1))$ for $d\ge 3$.
  By Lemma~\ref{lem_ABB-surjectivity} we get $\dim \HH^0(Y',L|_{Y'})=2$, hence
  $(Y',L|_{Y'}) \simeq (\PP^1,\cO(1))$.
  But then $(u_2-u_1)/2$ is not in the compass of the action of $\CC^*$ on $Y'$ at the fixed point associated to $u_1$.
  By Lemma~\ref{lem_reduction_of_compass} the vector $(u_2-u_1)/2$ is not in the compass of the action of $H$ on $X$ at the same point either.
  Therefore, we get a contradiction as in the previous case.
\end{proof}

\section{Contact manifolds with a simple automorphisms group}\label{sec_contact_simple}
In this section, we always suppose the following assumptions hold.

\begin{assumpt}[contact manifolds with an action of torus]\label{assumptions-weaker} \hfill
\begin{enumerate}[leftmargin=*]
   \item \label{item_assumption_contact} $X$ is a contact Fano manifold of dimension $2n+1$, 
           with a contact form $\theta\in\HH^0(X,\Omega_X\otimes L)$.
           In particular, $L$ is an ample line bundle on $X$.
  \item  \label{item_assumption_PicX_ZL}
           $\Pic X = \ZZ \cdot L$.
  \item \label{item_assumption_torus_action}
           There exists a simply-connected simple group $G$ of rank $r$ with a maximal torus $H$, such that $H$ acts on $X$.
           $M$ is the lattice of weights of $G$, thus also the lattice of characters of $H$.
  \item \label{item_assumption_Delta_weak}
    $\Delta(X,L,H)=\Delta(G)$, where $\Delta(G)$ is
    the convex hull of roots.
  \item \label{item_assumption_extremal_points_are_isolated_pts}
    all extremal fixed-point components in $X^H$ are isolated  points
    thus, in particular, $\Delta(X,L,H)=\Gamma(X,L,H)$ by Corollary~\ref{cor_Delta=Gamma-small_components}
    and for every vertex $v$ of $\Delta(X,L,H)$ the dimension of $\Gamma(X,L,H)_v$ is $1$ by Lemma~\ref{lem_ABB-surjectivity}.
  \end{enumerate}
\end{assumpt}

Moreover, our main interest is in the stronger version of these assumptions.

\begin{assumpt}[contact manifolds with simple automorphism group]\label{assumptions-simple}
In addition to Assumptions~\ref{assumptions-weaker} we suppose:
\begin{enumerate}[leftmargin=*]
   \addtocounter{enumi}{5}
   \item \label{item_assumption_torus_simple_automorphisms}
         The identity component of the group of automorphisms of $X$ is a simple group.
         The group $G$ from \ref{item_assumption_torus_action} is the universal cover of that simple group.
   \item \label{item_assumption_H0L} 
         $\HH^0(X,L)$, as a representation of $G$, can be identified with~$\g$,
         the Lie algebra of $G$, with the adjoint action of $G$. 
         Hence $\widetilde{\Gamma}(X,L,H)$ contains roots of $G$, each with
         multiplicity $1$ and $0\in \widetilde{\Gamma}(X,L,H)$ with multiplicity $r$.
  \end{enumerate}
\end{assumpt}

These assumptions are motivated by the hypothesis of Theorem~\ref{main_thm} 
   together with results of Section~\ref{sec_contact_general}.
The details are discussed in Subsection~\ref{subsec_classification_contact}.
As we will see, the role of $G$ is rather decorative, 
   but, instead, the action of $H$ rules the roost. 
The importance of $G$ is mainly in Item~\ref{item_assumption_H0L}, 
   which forces the weights of $H$ to be prescribed by the simple root system.
We use the weaker set of assumptions to be able to say something explicit about some of the existing homogeneous cases, 
   but with an action of a subtorus of the maximal torus, which highly resembles a general situation for a smaller group (see Subsection~\ref{sectA2}).

\begin{thm}\label{thm_contact_simple}
   Suppose Assumptions~\ref{assumptions-simple} are satisfied. 
   Then one of the following holds:
   \begin{itemize}
    \item $X\simeq Gr(\PP^1,\cQ^{n+2})$ and $G$ is of type $B_{\frac{n+3}{2}}$ or $D_{\frac{n+4}{2}}$ (depending on the parity of $n$), or
    \item $G$ is of type $G_2$ and either $X$ is the (homogeneous) $5$-di\-men\-sional adjoint variety of $G_2$ or $\dim X \ge 11$, or
    \item $\dim X \ge 11$ and $G$ is of type $A_2$, or
    \item $\dim X \ge 21$ and $G$ is of type $E_6$, or
    \item $\dim X \ge 33$ and $G$ is of type $E_7$, or
    \item $\dim X \ge 57$ and $G$ is of type $E_8$, or
    \item $\dim X \ge 15$ and $G$ is of type $F_4$, or
    \item $G$ is of type $A_1$.
   \end{itemize}
\end{thm}

A simple Lie group is of one of the types $A_r$ (for $r\ge 1$), $C_r$ (for $r \ge 2$), $B_r$ (for $r\ge 3$), $D_r$ (for $r \ge 4$), 
   $E_6$, $E_7$, $E_8$, $F_4$, or $G_2$.
The goal of this section is to analyze case-by-case all of them and figure out which of them can appear under the assumptions above.
The cases $A_r$ are excluded in Subsection~\ref{sectAr} (for $r\ge 3$), the cases $C_r$ in Subsection~\ref{sectCr}, 
  while the cases $B_r$ and $D_r$ are analyzed in Subsection~\ref{sectBrDr}.
The exceptional cases $E_r$ and $F_4$ follow from  Lemma~\ref{lem_E-and-F} (and the fact, that for a contact manifold $X$ its dimension is odd).
Finally, the types $A_2$ and $G_2$, which require much more attention than the other ones,
  are treated in Subsection~\ref{sectA2}.
Moreover, in Corollary~\ref{cor_A1} using a different method we strengthen the last case of type $A_1$, showing that then $\dim X\ge 11$,
  analogously to the $A_2$ and $G_2$ cases.

\subsection{General techniques}

The highlights of the general strategy for the cases are as follows. 
Firstly, we consider subtori of $H$ associated to hyperplanes supporting faces and use Lemma~\ref{lem_ABB-surjectivity} to understand the extremal fixed-point components
of these actions. In particular, we use Proposition~\ref{prop_smallDelta2} and its
corollaries. 
Secondly, we study the compass of $H$ at any extremal fixed point of $X^H$  
and apply  Lemma~\ref{lem_t-action-contact} and Corollaries~\ref{cor_compass_and_fixpt} and~\ref{cor_weights-jets}.

We gather a few observations that we exploit in all cases.
For every face $\delta$ of $\Delta(G)$ we denote by $Y_{\delta}$ the corresponding extremal component of $X^{H'}$, 
where $H' \subset H$ is the subtorus corresponding to the quotient of $M$ by the linear space parallel to $\delta$.

\begin{lemma}\label{lem_faces_and_h0L}
   Suppose Assumptions~\ref{assumptions-weaker} hold.
   For every proper face $\delta$ of $\Delta$, that is not a vertex, we have $\Pic {Y_{\delta}} = \ZZ \cdot L|_{Y_{\delta}}$.
   For the action of the quotient torus $H/H'$ on $Y_{\delta}$
       the polytopes of sections and fixed points, 
       $ \Gamma(Y_{\delta}, H/H', L|_{Y_{\delta}})$ and $\Delta(Y_{\delta}, H/H', L|_{Y_{\delta}})$, 
       coincide and both are equal to $\delta$ up to a shift by a lattice vector in the affine span of $\delta$. 
   In particular, $h^0(L|_{Y_{\delta}})$ is at least the number of vertices of $\delta$. 
   If $\delta$ has no lattice points other than the vertices, then equality holds.
   If Assumptions~\ref{assumptions-simple} hold, then $h^0(L|_{Y_{\delta}})$ is equal to the number of roots of $G$ contained in $\delta$. 
\end{lemma}
\begin{proof}
    The claim on $\Pic Y_{\delta} = \ZZ \cdot L|_{Y_{\delta}}$ follows from Assumption~\ref{assumptions-weaker}\ref{item_assumption_PicX_ZL} 
        and Lemma~\ref{lem_ABBdecomposition-extremal_cells}\ref{item_extremal_cells_case_dim_positive}.
    The equality of polytopes follows from Lemma~\ref{lem_Gamma_eq_Delta_propagates_to_faces}.
    Finally, the statements on $h^0(L|_{Y_{\delta}})$ are implied by Lem\-ma~\ref{lem_ABB-surjectivity} 
        and Assumption~\ref{assumptions-simple}\ref{item_assumption_H0L}.
\end{proof}

We note that for most simple groups we have one of the following properties of an edge $\delta$ of $\Delta(G)$.
\renewcommand{\theenumi}{\textnormal{($\star_{\arabic{enumi}}$)}}
\begin{enumerate}
 \addtocounter{enumi}{-1}
  \item \label{item_star_zero}
    There is no lattice point of $\delta$ other than the end points.
  \item \label{item_star_one}
    There is exactly one lattice point of $\delta$ other than the end points and this middle lattice point is a root of $G$.
\end{enumerate}
\renewcommand{\theenumi}{(\arabic{enumi})}

It can be verified that (except in the case of type $A_1$) either \ref{item_star_zero} or \ref{item_star_one} holds for all edges of the root polytope.
We check this for the relevant cases in the subsequent subsections.
Now we list some consequences.

\begin{lemma}\label{lem_edges_of_root_polytopes}
   Under Assumptions~\ref{assumptions-weaker}, consider an edge $\delta$ of the polytope $\Delta$ and let $Y_{\delta}$ 
     be the corresponding extremal component as above.
   \begin{enumerate}
    \item \label{item_short_edge}
          If $\delta$ satisfies \ref{item_star_zero},
             then $(Y_{\delta}, L|_{Y_{\delta}}) \simeq (\PP^1, \cO(1))$.
    \item \label{item_long_edge}
          If Assumptions~\ref{assumptions-simple} hold and $\delta$ satisfies \ref{item_star_one},
             then $(Y_{\delta}, L|_{Y_{\delta}}) \simeq (\PP^2, \cO(1))$.
   \end{enumerate}
\end{lemma}
\begin{proof}
   The $1$-dimensional quotient torus $\Lambda$ corresponding to the line in $M_{\RR}$ parallel to $\delta$ acts on $Y_{\delta}$ 
     and the polytope $\Delta(Y_{\delta}, L|_{Y_{\delta}}, \Lambda)$ 
     is equal to $\delta$ by Lemma~\ref{lem_reduction_of_action}\ref{item_reduction_of_action_restricting_to_fix_points_of_subgroup} 
     and Lemma~\ref{lem_faces_and_h0L}.
   Also by Lemma~\ref{lem_faces_and_h0L} we must have $\Pic Y_{\delta} = \ZZ \cdot L$ and $h^0(L|_{Y_{\delta}}) =2$ or $3$.
   Finally, the claim follows from Proposition~\ref{prop_smallDelta2}. 
\end{proof}

One of the ideas to exlude many of the cases is to study the compass of an extremal fixed-point component in $X$.
This way we can restrict the dimension of $X$ and many other of its properties.

\begin{cor}\label{cor_compass_and_faces}
   Under Assumptions~\ref{assumptions-weaker}, pick a vertex $v\in \Delta$ and suppose $y \in X^H$ is the corresponding fixed point.
   Let $\cC= \cC(y, X, H)$ be the compass of $y$ in $X$. Then:
   \begin{enumerate}
    \item \label{item_edges_are_in_compass}
          If $\delta \subset \Delta$ is an edge containing $v$ and satisfying \ref{item_star_zero} (or \ref{item_star_one}, if in addition Assumptions~\ref{assumptions-simple} are satisfied),
            then for every lattice point $m \in \delta\setminus\set{v}$ we have $m-v\in \cC$ with multiplicity exactly $1$.
          Moreover, for any lattice point $m'$ in the line containing $\delta$, 
            the vector $m'-v$ is in the compass if and only if $m'\in \delta \setminus \set{v}$.
    \item \label{item_minus_v_in_compass}
          The compass $\cC$ contains $-v$ with multiplicity exactly $1$.
    \item \label{item_compass_in_cones}
          Consider a convex cone $\sigma \subset M_{\RR}$ generated by the shift $\Delta - v$.
          Then the elements of the compass $\cC$ are contained in the set of lattice points of the intersection $\sigma \cap (-v -\sigma)$.
   \end{enumerate}
\end{cor}

\begin{proof}
    To see \ref{item_edges_are_in_compass} consider $Y_{\delta} \subset X$.
    By Lemma~\ref{lem_reduction_of_compass} it is enough to prove $m-v\in \cC(y, Y_{\delta})$ with multiplicity $1$. 
    If $\delta$ satisfies \ref{item_star_zero}, then the only possible $m$ is the end point of $\delta$, 
      and $Y_{\delta} \simeq \PP^1$ with $L|_{Y_{\delta}}\simeq \cO(1)$ by Lemma~\ref{lem_edges_of_root_polytopes}.
    In particular, there is only one element of the compass and it must be equal to $v'-v$, 
       where $v'\in \delta$ is the other vertex.
    If $\delta$ satisfies \ref{item_star_one}, then there are two possible $m$, the end point of $\delta$ and the middle point, and both are roots of $G$. 
    By Lemma~\ref{lem_edges_of_root_polytopes} we have $Y_{\delta} \simeq \PP^2$ with $L|_{Y_{\delta}}\simeq \cO(1)$
      and by Corollary~\ref{cor_smallDelta2-compass} there are exactly two elements of the compass equal to the two possible values of $m-v$.
      
    Item~\ref{item_minus_v_in_compass} follows from  the short exact sequence of vector spaces $0 \to L^*_y \to T^*_yX \to F^*_y \to 0$, because $-v$ is the weight of the action on $L^*_y$,
       and from the duality in Lemma~\ref{lem_t-action-contact}. 
    
    To prove Item~\ref{item_compass_in_cones}, we note that the compass is contained in $\sigma$ by Corollary~\ref{cor_compass_and_fixpt}.
    It is also contained in  $-v -\sigma$ by Lemma~\ref{lem_t-action-contact}.
\end{proof}

\begin{lemma}\label{lem_zero_not_in_Delta_tilde}
   Under Assumptions~\ref{assumptions-weaker} suppose in addition that $0 \in \widetilde{\Delta}$, that is there exists a fixed-point component $Y \subset X^H$ with $\mu_0(Y)=0$.
   Then  $\widetilde{\Delta}$ also contains another point, which is nonzero, and not a vertex of $\Delta$.
\end{lemma}

\begin{proof}
   By Corollary~\ref{cor_contact-fixedpt-contact} the manifold $Y$ is contact, in particular, $\dim Y \ge 1$.
   Let $\nu \in \cC(Y, X, H)$ be any element and define $W$ to be the linear span of $\nu$.
   By Lemma~\ref{lem_number_of_vectors_on_compass_in_a_fixed_direction}, there is another component $Y' \subset X^H$,
      such that $\mu(Y') \in W$ and $ \dim Y' + b' \ge 2$, where $b'$ is the number of elements in $\cC(Y', X, H) \cap W$.
   But if $Y'$ was an extremal component, then $\dim Y' =0$ by Assumption~\ref{assumptions-weaker}\ref{item_assumption_extremal_points_are_isolated_pts}
      and $b' =1$ by Corollary~\ref{cor_compass_and_faces}\ref{item_minus_v_in_compass}, a contradiction. 
   Thus $Y'$ is not extremal, showing the claim of the lemma.
\end{proof}

\subsection{Case of \texorpdfstring{$A_r$}{Ar} with \texorpdfstring{$r\geq 3$}{r at least 3}}\label{sectAr}
We use the notation consistent with \cite[Table 1]{Bourbaki}.
That is, the roots of $A_r$ are located on a hyperplane $\sum_{i=0}^r e_i^*=0$ 
in a Euclidean space~$\RR^{r+1}$ with an orthonormal basis~$e_0,\dots, e_r$, 
where by $e_i^*$ we denote the dual functionals. 
The roots are $e_i-e_j$, $i\ne j$, and they are the vertices of $\Delta(A_r)$. 
The lattice of weights is generated by the roots and $e_0-(e_0+\cdots+e_r)/(r+1)$.
All facets of $\Delta(A_r)$ can be described as follows: 
  take a proper nonempty subset of indices $J\subset\{0,\dots, r\}$ and define the facet $\delta(J)$ as 
$$
  \delta(J)=\Delta(A_r)\cap 
    \left\{u\in M_{\RR}: \left(\sum_{i\in J}e_i^*\right)(u)=1 \right\}.
$$ 
Note that $e_i-e_j\in \delta(J)$ if and only if $i\in J$ and $j\not\in
J$. 
Our main interest is in the face $\delta(\set{0,1})$ and the hyperplane $\RR^{r-1}$ containing this face, given by the equation $e_0^*+ e_1^* =1$ 
   (together with the omnipresent equation $\sum_{i=0}^r e_i^*=0$).

\begin{lemma}\label{lem_Ar_lattice_points_on_delta}
   Suppose $r \ge 2$. Then all weight lattice points in $\delta(\set{0,1})$ are the vertices, except in the case $r=3$,
      when there is in addition one interior lattice point $\frac{1}{2}(e_0 + e_1 - e_2 - e_3)$.
   In particular, all edges of $\Delta$ satisfy \ref{item_star_zero}.
\end{lemma} 

\begin{proof}
   Any point in the weight lattice is of the form $v = \sum a_i e_i + b ( e_0-(e_0+\cdots+e_r)/(r+1))$ for some integers $a_i$ satisfying $\sum a_i =0$, and $b \in \setfromto{0}{r-1}$.
   If $b=0$ and $v$ is in $\delta(\set{0, 1})$, then it is straightforward to verify that $v$ is one of the vertices $e_0-e_j$ or $e_1-e_j$ ($j \notin \set{0,1}$).
   Thus assume $b \in \setfromto{1}{r-1}$.
   The value $(e_0^*+ e_1^*)(v)$ is a sum of an integer and $\frac{b (r-1)}{r+1}$.
   Thus if this value is an integer, then 
   $ 2 b \equiv 0 \mod (r+1)$, in particular, $r$ must be odd, say equal to $2k+1$.
   Then by our assumptions $b=k+1$ and $k\ge 1$. 

   So assume $v \in \delta(\set{0,1})$ so that $1= a_0 + a_1 + \frac{2k (k+1)}{2k+2} = a_0 + a_1 + k$. 
   Moreover, $v$ safisfies the inequalities of $\Delta(A_r)$, that is for all proper subsets $J \subset \setfromto{0}{r}$
      we have $\sum_{i \in J} e_i^*(v) \le 1$.
   The relevant inequalities are those for $J$ equal to one of the following: $\set{0}$, $\set{1}$, $\set{0,1 ,i}$, or $\setfromto{0}{r}\setminus\set{i}$.
   The first two inequalities give:
   \[
      a_0 + k + \tfrac{1}{2} \le 1 \text{ and } a_1 - \tfrac{1}{2} \le 1.
   \]
   Combining them with the equation $a_0 + a_1  = 1-k$ and the assumption that $a_i$ are integers we get $a_0 = -k$ and $a_1 = 1$.
   Then
   \[
     v = \tfrac{1}{2} e_0 + \tfrac{1}{2} e_1  + \sum_{i=2}^{2k+1} (a_i -\tfrac{1}{2}) e_i.
   \]
   Applying the remaining two inequalities we get 
   \begin{align*}
      1 + a_i -\tfrac{1}{2} & \le 1  & (\Rightarrow a_i &\le 0)\\
      \underbrace{1 + \sum_{j=2}^{2k+1} (a_j - \tfrac{1}{2})}_{= 0 \text{ since }\sum_{i=0}^{2k+1} a_i =0} - (a_i - \tfrac{1}{2}) &\le 1   & (\Rightarrow a_i &\ge 0).
   \end{align*}
   Therefore each $a_i$ for $i\ge 2$ is zero, and 
   \[
     v = \tfrac{1}{2} (e_0 +  e_1  - \sum_{i=2}^{2k+1} e_i),
   \]
   which satisfies the omnipresent equation $\sum_{i=0}^{2k+1} e_i^*=0$ if and only if $k=1$, that is $r=3$, and the proof is concluded.
\end{proof}

\begin{prop}\label{prop_A_r-rbig}
  Suppose that $(X,L)$ and $(G,H)$ satisfy Assumptions~\ref{assumptions-weaker} 
  with $G$ of type $A_r$.  
  Then $r\le 3$.
  If, in addition, they satisfy Assumptions~\ref{assumptions-simple}, then $r\le 2$.
\end{prop}

\begin{proof}
  Suppose by contradiction that $r\ge 3$.
  We take a facet $\delta(\{0,1\})$ associated to $J$ consisting of two
  indices, $0$ and $1$.
  The vertices of this facet are $e_0-e_i$ and $e_1-e_i$, with $i=2,\dots, r$.
  It follows that $\delta(\{0,1\})$ (without taking the lattice structure into account) 
     has the same combinatorial structure as a product
     of an interval and a simplex of dimension $r-2$. 
  In particular, $\delta$ has $2(r-1)$ vertices. 
  Fix a vertex of $\delta$, say  $e_0-e_2$. 
  The edges of $\delta$ that contain this vertex are those connecting it to vertices $e_1 - e_2$ and $e_0 - e_i$ (for $i \in \setfromto{3}{r}$). 
  In particular, there are $r-1$ of them.

  We will run the torus action program on the manifold  
  \[
    Y = Y_{\delta(\set{0,1})} \subset X
  \]
  associated to the facet
  $\delta(\{0,1\})$.
  It admits an almost faithful action of a torus of rank $r-1$ and the compass of a fixed point $y$ associated to a vertex $e_0-e_2$ is the intersection of the compass of $X$ at $y$ 
     and the hyperplane parallel to the facet.
  In particular, by Corollary~\ref{cor_compass_and_faces}\ref{item_edges_are_in_compass} there are $r-1$ elements of the compass coming from the edges, all of them of multiplicity $1$.
  
  We claim that the compass of $y$ in $X$ or in $Y$ does not contain the vectors of the form $a\left((e_1 - e_i) - (e_0 - e_2)\right)$ 
     for any $i \in \setfromto{3}{r}$ and $a\ge 1$.
  Indeed, the symmetric counterpart (see Corollary~\ref{cor_compass_and_faces}\ref{item_compass_in_cones}, or Lemma~\ref{lem_t-action-contact}) 
    of this vector is equal to  
    \[
      -(e_0 - e_2) - a\left((e_1 - e_i) - (e_0 - e_2)\right) = a(e_i - e_1) + (a - 1)(e_0 - e_2),
    \]
    which does not belong to the cone $\sigma=\RR_{\geq 0}(\Delta(A_r)-(e_0-e_2))$.
  This last claim can be seen using the inequalities of $\sigma$, one of them is $e_0^* +e_i^* \le 0$.
     
  Suppose $r\ge 4$. By Lemma~\ref{lem_Ar_lattice_points_on_delta}, the only lattice points in $\delta$ are the vertices.
  Thus by Corollary~\ref{cor_compass_and_fixpt}, 
     any element of the compass must be of the form $a(v - (e_0 - e_2))$ for some vertex $v$ and $a\ge 1$.
  Therefore there are only $r-1$ elements of the compass,
     $\dim Y = r-1$, and thus $Y$ is a (smooth, projective) toric variety with 
     $\Pic Y = \ZZ\cdot  L|_{Y}$ 
     and $h^0(L|_{Y}) = 2(r-1)$  by Lemma~\ref{lem_faces_and_h0L}.
  But the only smooth complete toric variety of dimension $r-1$  with $\Pic Y \simeq \ZZ$ is $\PP^{r-1}$, 
     and the ample generator of the Picard group has an $r$-dimensional space of sections.
  Therefore, $2(r-1) =r$, a contradiction with our assumption that $r\ge 4$.
  
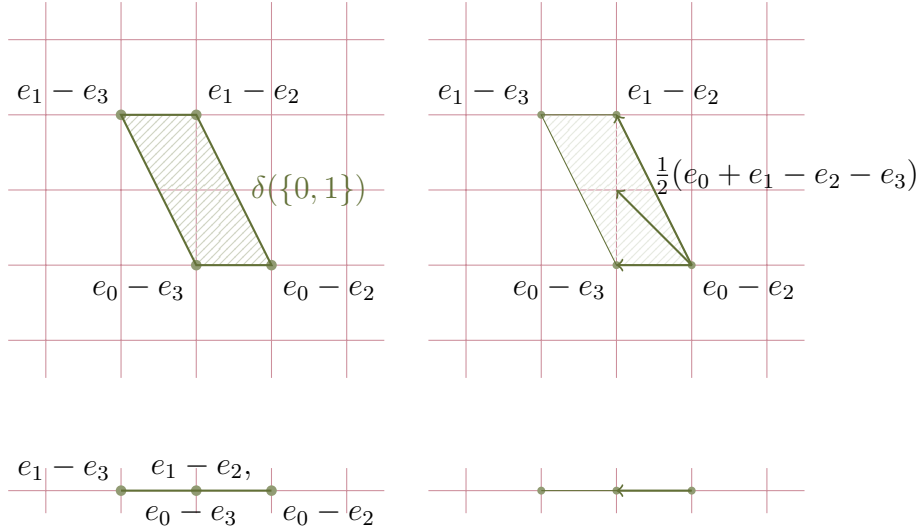
\begin{figure}[hbt]
\begin{center}
\begin{tikzpicture}[baseline=.4cm,scale=1]
\draw[step=1, fiolet, thin] (-2.5,-2.5) grid (2.5, 2.5);
\fill[pattern color=oliwkowy!30, pattern=north east lines] (0,-1) -- (1,-1) -- (0,1) -- (-1,1) -- cycle;
\fill[oliwkowy!70] (0,-1) circle (2pt);
\fill[oliwkowy!70] (1,-1) circle (2pt);
\fill[oliwkowy!70] (0,1) circle (2pt);
\fill[oliwkowy!70] (-1,1) circle (2pt);
\draw[thick, oliwkowy] (0,-1) node[anchor=north east, color=black]{$e_0- e_3$} -- (1,-1) node[anchor=north west, color=black]{$e_0- e_2$} -- (0,1) node[anchor=south west, color=black]{$e_1- e_2$}-- (-1,1) node[anchor=south east, color=black]{$e_1- e_3$} -- cycle;
\draw (1.5,0) node[color=oliwkowy]{$\delta(\set{0,1})$};
\draw[step=1, fiolet, thin] (-2.5, -4.3) grid (2.5, -3.7);
\fill[oliwkowy!70] (-1,-4) circle (2pt) node[anchor=south east, color=black]{$e_1- e_3$};
\fill[oliwkowy!70] (0,-4) circle (2pt) node[anchor=south, color=black]{$\ e_1- e_2, $} node[anchor=north,  color=black]{$e_0- e_3 \ $};
\fill[oliwkowy!70] (1,-4) circle (2pt) node[anchor=north west, color=black]{$e_0- e_2$};
\draw[thick, oliwkowy] (-1,-4) -- (1,-4);
\end{tikzpicture}
\quad 
\begin{tikzpicture}[baseline=.4cm,scale=1]
\draw[step=1, fiolet, thin] (-2.5,-2.5) grid (2.5, 2.5);
\fill[pattern color=oliwkowy!15, pattern=north east lines] (0,-1) -- (1,-1) -- (0,1) -- (-1,1) -- cycle;
\fill[oliwkowy!70] (0,-1) circle (1.5pt);
\fill[oliwkowy!70] (1,-1) circle (1.5pt);
\fill[oliwkowy!70] (0,1) circle (1.5pt);
\fill[oliwkowy!70] (-1,1) circle (1.5pt);
\draw[oliwkowy] (0,-1) node[anchor=north east, color=black]{$e_0- e_3$} -- (1,-1) node[anchor=north west, color=black]{$e_0- e_2$} -- (0,1) node[anchor=south west, color=black]{$e_1- e_2$}-- (-1,1) node[anchor=south east, color=black]{$e_1- e_3$} -- cycle;
\draw (0.35,-0.2) node[anchor=south west]{$\tfrac{1}{2}(e_0+e_1 - e_2- e_3)$};
\draw[->, thick, oliwkowy] (1,-1) -- (0,-1);
\draw[->, thick, oliwkowy] (1,-1) -- (0,0);
\draw[->, thick, oliwkowy] (1,-1) -- (0,1);
\draw[step=1, fiolet, thin] (-2.5, -4.3) grid (2.5, -3.7);
\fill[oliwkowy!70] (-1,-4) circle (1.5pt); 
\fill[oliwkowy!70] (0,-4) circle (1.5pt);
\fill[oliwkowy!70] (1,-4) circle (1.5pt);
\draw[oliwkowy] (-1,-4) -- (1,-4);
\draw[->, thick, oliwkowy] (1,-4) -- (0,-4);
\end{tikzpicture}
\end{center}
\caption{Face $\delta(\set{0,1})$ for $r=3$ with a compass and a downgrading to one dimensional torus.} \label{fig_Ar_delta_for_A3}
\end{figure}

  It remains to consider the case of $r=3$, which has an additional lattice point in the relative interior of $\delta$, as shown in Lemma~\ref{lem_Ar_lattice_points_on_delta} 
    and illustrated on the top left in Figure~\ref{fig_Ar_delta_for_A3}.
  By Corollary~\ref{cor_compass_and_fixpt} the compass of $y$ in $Y$ must be contained in the union of three half lines starting at the vertex $e_0-e_2$ shifted to $0$
    and passing through the lattice points of $\delta$.
  By the above considerations, the only possibilities are the three vectors indicated on the top right in Figure~\ref{fig_Ar_delta_for_A3}.

Downgrading the action of the two dimensional torus by a vertical squeeze of the lattice (bottom of Figure~\ref{fig_Ar_delta_for_A3}),
   we are in the situation of Proposition~\ref{prop_smallDelta2}.
Thus the manifold $Y$ is either a projective space or a quadric of dimension at least $3$. 
By Lemma~\ref{lem_faces_and_h0L} the space of sections $H^0(Y, L|_{Y})$ is $4$-dimensional, 
  thus $Y \simeq \PP^3$. 
(Note that here we exploit Assumption~\ref{assumptions-simple}\ref{item_assumption_H0L}.)
The compass of $y$ in $Y$ after the downgrading must be just the vector of length $1$ with some multiplicity by Lemma~\ref{lem_reduction_of_compass}.
This is in contradiction with Corollary~\ref{cor_smallDelta2-compass} for $\PP^3$ and concludes the proof.
\end{proof}

\subsection{Case of \texorpdfstring{$C_r$}{Cr} with \texorpdfstring{$r\geq 2$}{r at least 2}}\label{sectCr}
\begin{lemma}\label{lem_C_r-big} 
  Assume $X$, $L$, $G$, and $H$ satisfy Assumptions~\ref{assumptions-simple}.
  Then $G$ is not of type $C_r$ (for any $r\geq 2$).
\end{lemma}

\begin{proof}
  We use information from \cite[Table 3]{Bourbaki}. The polytope
  $\Delta(C_r)$ is in a Euclidean space $\RR^r$ with a basis $e_i$,
  $i=1,\dots, r$. The vertices of $\Delta(C_r)$ are long roots $\pm
  2e_i$, each edge of $\Delta(C_r)$ contains a short root $\pm e_i\pm
  e_j$, $i\ne j$ and the roots are the only weights contained in the
  edges.
  In other words, the condition \ref{item_star_one} is satisfied.
  By Lemma~\ref{lem_edges_of_root_polytopes} the manifolds corresponding to edges are isomorphic to $\PP^2$,
     and by  Corollary~\ref{cor_compass_and_faces}\ref{item_edges_are_in_compass} the difference $2 (e_2 - e_1)$
     is in the compass of the fixed point corresponding to $2e_1$.
  This is in a contradiction with Corollary~\ref{cor_compass_and_faces}\ref{item_compass_in_cones},
     thus $G$ cannot be of type $C_r$.
\end{proof}


\subsection{Cases of \texorpdfstring{$B_r$}{Br} and \texorpdfstring{$D_r$}{Dr}}\label{sectBrDr}
To avoid an overlap with the previous cases we assume $r\geq 3$ for the
case $B_r$ and $r\geq 4$ for the case $D_r$.  We use \cite[Table 2 and
4]{Bourbaki}.  We take the lattice $\ZZ^r$ generated by $e_1,\dots, e_r$
and define $M=\sum_i\ZZ e_i+\ZZ(\sum_i e_i)/2$.  Then
$\Delta=\Delta(B_r)=\Delta(D_r)$ has vertices $\pm e_i\pm e_j$ for
$i\ne j$. The vertex $\pm e_i\pm e_j$ is connected by an edge to $\pm
e_r\pm e_s$ if they have one common index in pairs $(i,j)$ and $(r,s)$
and the same sign for this index.  There are distinguished facets of
$\Delta$ defined as $\delta_i^{\pm}:=\{u\in M_\RR: e_i^*(u)=\pm 1\}$ with vertices
$e_i\pm e_j,\ j\ne i$ and $-e_i\pm e_j, j\ne i$, respectively.
More generally, the inequalities defining $\Delta$ include:
\[
   \pm e_i^* \le 1 \text{ and } \sum_i (\pm e_i^*) \le 2
\]
These give $2r + 2^r$ linear inequalities which can be briefly described as $|e_i^*|\le 1$ and $\sum_i |e_i^*| \le 2$.
It is straightforward to see that each inequality is the supporting inequality of a facet, that is these are minimal inequalities.
It can also be verified that these are all inequalities of $\Delta$, but we are not going to use this last statement.

All the non-zero weights contained in $\Delta$ are the following:
\begin{itemize}
 \item the vertices of $\Delta$, which are all roots in the case $D_r$, or long roots in the case $B_r$, 
 \item points $\pm e_i$  which lie on the facets $\{u\in M: e_i^*(u)=\pm 1\}$ (and which in the case $B_r$ are the short roots),  and
 \item for $r\leq 4$ points $(\sum_{i=1}^r \pm e_i)/2$ which for $r=3$ lie in the
         interior of $\Delta$ and for $r=4$ they lie on the facets of type $\{u\in
         M_\RR: \sum_i \pm e_i^*(u)= 2\}$.
\end{itemize}

\begin{lemma}\label{lem_points_that_can_be_in_compass}
   Suppose $r\ge 3$ and fix a vertex of $\Delta$, say $e_1+ e_2$. Let $\sigma = \RR_{\ge 0}\cdot (\Delta - (e_1+e_2))$, 
      as in Corollary~\ref{cor_compass_and_faces}\ref{item_compass_in_cones}.
   Then every lattice point $v$ in $\sigma \cap (-e_1 - e_2 -\sigma)$ satisfies at least one of the following:
   \begin{itemize}
      \item $v$ is contained in at least one of the hyperplanes $e_1^* =0$ or $e_2^* =0$, or
      \item $v = -e_1 - e_2$, or
      \item $r\le 4$ and $v = -\frac{1}{2} (e_1 + e_2 \pm e_3 \pm e_4)$  (if $r=4$) or $v = -\frac{1}{2} (e_1 + e_2 \pm e_3)$ (if $r=3$). 
   \end{itemize}
\end{lemma}

\begin{proof}
   Some inequalities of $\sigma$ are easily obtained from the inequalities of $\Delta - (e_1+e_2)$.
   In particular, for all $v\in \sigma$ we have
   \renewcommand{\theenumi}{\textnormal{(\roman{enumi})}}
   \begin{enumerate}
    \item  $e_1^*(v) \le 0$,
    \item  $e_2^*(v) \le 0$,  
    \item  $(e_1^* + e_2^*  + \sum_{i=3}^r \pm e_i^*)(v) \le 0$ for any choices of signs $\pm$.
   \end{enumerate}
   Similarly, the inequalities of $-e_1 - e_2 -\sigma$ include:
   \begin{enumerate}
       \addtocounter{enumi}{3}
    \item  $e_1^*(v) \ge -1$,
    \item  $e_2^*(v) \ge -1$,  
    \item  \label{item_inequality_of_shifted_minus_C_sum_ge_2}
           $(e_1^* + e_2^*  + \sum_{i=3}^r \pm e_i^*)(v) \ge -2$ for any choices of signs $\pm$.
   \end{enumerate}
\renewcommand{\theenumi}{(\arabic{enumi})}
   In particular, the only possible values of $e_1^*(v)$ for a lattice point $v\in M$
      in the intersection of the two cones are $0$, $-\frac{1}{2}$, or $-1$ 
      (and similarly for $e_2^*(v)$).
   The case of one of them equal to $0$ is the first item in the lemma.
   Thus assume otherwise, that both $e_1^*(v)$ and $e_2^*(v)$ are non-zero.
   By the construction of $M$, either all coordinates of $v$ are integral,
      or all are congruent to $\frac{1}{2}$ modulo $1$.
   In particular, $e_1^*(v) =e_2^*(v) \in \set{-\frac{1}{2}, -1}$.
   
   If $e_1^*(v) =e_2^*(v) =-1$, then by \ref{item_inequality_of_shifted_minus_C_sum_ge_2} 
      we must have $\sum_{i=3}^r \pm e_i^*(v) \ge 0$ for any choices of signs.
   Thus $e_i^*(v) =0 $ for all $i \in \setfromto{3,4}{r}$, that is $v=-e_1 - e_2$ as in the second item of the lemma.

   Finally, consider the case $e_1^*(v) =e_2^*(v) =-\frac{1}{2}$, and all the other coordinates are also non-integral.
   We must have $-1 \le \sum_{i=3}^r \pm e_i^*(v) \le 1$, thus $r\le 4$ and $e_i^*(v) = \pm \frac{1}{2}$ for $i=3$ and $i=4$ (when $r=4$).
\end{proof}

\begin{lemma}\label{lem_B_r-D_r-facets}
  Suppose $(X,L)$ and $(G,H)$ satisfy Assumptions~\ref{assumptions-simple} 
  with $G$ of type $B_r$, $r\geq 3$ or $D_r$,
  $r\geq 4$.  Then the following conditions hold:
  \begin{enumerate}
  \item \label{item_Br_Dr_facets_projection_ei}
    the source/sink of the action on $X$ of the 1-dimensional
    subtorus associated to the projection $e_i^*: M\rightarrow \ZZ\cdot \tfrac{1}{2}$ is the
    quadric $\cQ^{2r-3}$ in the case $B_r$ and quadric $\cQ^{2r-4}$ in
    the case $D_r$,
  \item \label{item_Br_Dr_middle_point_is_not_fixed}
    the middle point $\pm e_i$ of the facet $e_i^* = \pm 1$
      does not correspond to any fixed-point component of $X^H$,
    that is $\pm e_i \notin \widetilde{\Delta}(X, L, H)$,
  \item \label{item_Br_Dr_facets_projection_sum_of_ei}
    for $r=4$ the source/sink of the action on $X$ of the
    1-di\-men\-sional subtorus associated to the projection $\sum_{i=1}^4
    e_i^*: M\rightarrow \ZZ$ is the quadric $\cQ^4$.
  \end{enumerate}
\end{lemma}
\begin{proof}
  In \ref{item_Br_Dr_facets_projection_ei} fix $i=1$ 
     and consider the source, that is the fixed-point component corresponding to $1 \in \ZZ\cdot \tfrac{1}{2}$.
  The projection $e_1^*$ 
     maps the facet $\delta_1^{+}= \conv( e_1\pm e_j,\ j\ne 1)$ to $\set{+1}$.
  Thus the source is just $Y:=Y_{\delta_1^{+}}$.
  After shifting $e_1$ to $0$, the polytope $\delta_1^{+}$ becomes the convex hull of $\set{\pm e_j,\ j\ne 1}$, 
      and the lattice $M$ intersected with the linear span of $e_j$ (for $j \ne 1$) is equal to the group generated by the $e_j$'s.

   Therefore $\Delta(Y, L|_{Y}, (\CC^*)^{r-1}) = \conv(\pm e_j,\ j\ne 1)$ by Lemma~\ref{lem_faces_and_h0L}
      and by Corollary~\ref{cor_torus-on-quadrics} the manifold $Y$ is a quadric hypersurface. 
   The dimension of the quadric follows from Lemma~\ref{lem_faces_and_h0L}.
   
   To see \ref{item_Br_Dr_middle_point_is_not_fixed}, 
      note that by Lemma~\ref{lem_equality_of_Delta_for_faces} the middle point corresponds to a fixed point in $X^H$ (that is, the middle point is in $\widetilde{\Delta}(X,L, H)$)
      if and only if the point is in $\widetilde{\Delta}(Y_{\delta_i^{\pm}},L|_{Y_{\delta_i^{\pm}}}, (\CC^*)^{r-1})$.
   But by the above arguments, $Y_{\delta_i^{\pm}}$ is a quadric with the automorphisms group of rank $r-1$, 
      thus the action of $(\CC^*)^{r-1}$ must be (up to a finite cover) 
      the standard action of the maximal torus on a quadric, which has no non-extremal fixed points, see Subsection~\ref{subsec_examples_of_torus_action}.
   Thus $\pm e_i \notin \widetilde{\Delta}(X,L, H)$. 

   Now we restrict to the case $r=4$ in order to show \ref{item_Br_Dr_facets_projection_sum_of_ei}.
   Say, consider only the source $Y_{\delta}$ for the facet  $\delta:= \Delta \cap \{u\in M_\RR: (e_1^* + e_2^* + e_3^* + e_4^*)(u)= 2\}$.
   The lattice points on $\delta$ are the $6$ vertices $e_i + e_j$ 
      and the interior point $\frac{1}{2}(e_1 + e_2 + e_3 + e_4)$.
   Shifting $\frac{1}{2}(e_1 + e_2 + e_3 + e_4)$ to the origin,
      and choosing $\frac{1}{2}(e_1 + e_2 - e_3 - e_4)$, $\frac{1}{2}(e_1 - e_2 + e_3 - e_4)$, 
      $\frac{1}{2}(e_1 - e_2 - e_3 + e_4)$ as the basis of the $3$-dimensional sublattice containing the lattice span of $\delta$,
      this facet becomes a $3$-dimensional polytope such as in Corollary~\ref{cor_torus-on-quadrics}.
   Thus $Y_{\delta}$ is a quadric and by counting the number of sections in Lemma~\ref{lem_faces_and_h0L} we get $\dim Y_{\delta} = 4$,
      proving the claim.
\end{proof}
\begin{lemma}\label{lem_B_r-D_r-big}
  If $(X,L)$ and $(G,H)$ satisfy Assumption~\ref{assumptions-simple} 
  with $G$ of type $B_r$ and $r\geq 3$ or $D_r$ and $r\geq 4$. Then:
  \begin{itemize}
   \item $     \dim X =  \begin{cases}
                 \dim Gr(\PP^1, \cQ^{2r-3})= 4r-5 & \text{in the case $B_r$}\\
                 \dim Gr(\PP^1, \cQ^{2r-4})= 4r-7 & \text{in the case $D_r$,}
               \end{cases}$
   \item  the only points in $\widetilde{\Delta}$ are the vertices of $\Delta$,
   \item  for a fixed type of the group, the compass of any extremal component $y \in X$ corresponding to a fixed vertex of $\Delta$ is independent of $X$.
  \end{itemize}
\end{lemma}
\begin{proof}
  We count the elements of the compass at the fixed point $y\in X^H$ associated to the
  vertex $e_1+e_2$.
  By Corollary~\ref{cor_compass_and_faces}\ref{item_compass_in_cones} and Lemma~\ref{lem_points_that_can_be_in_compass}
    any element $\nu$ of the compass must be of one of the following types:
   \begin{itemize}
      \item $\nu$ is contained in at least one of the hyperplanes $e_1^* =0$ or $e_2^* =0$, or
      \item $\nu = -e_1 - e_2$, or
      \item $r=4$ and $\nu = -\frac{1}{2} (e_1 + e_2 \pm e_3 \pm e_4)$. 
      \item $r=3$ and $\nu = -\frac{1}{2} (e_1 + e_2 \pm e_3)$. 
   \end{itemize}

  If $\nu$ is of the first type, 
     from Lemmas~\ref{lem_reduction_of_compass} 
     and~\ref{lem_B_r-D_r-facets}\ref{item_Br_Dr_facets_projection_ei}
     we conclude that $\nu$ has multiplicity $1$ in the compass and is one of the 
     $\pm e_i -e_2$ and $\pm e_i -e_1$, for $i\geq 3$ and in the case $B_r$ also
       $-e_2$ and $-e_1$.
   This gives a total of $4(r - 2)$ elements in the case $D_r$ and $2$ more elements in the case $B_r$.

  If $\nu= -(e_1+e_2)$, then it is of the second type, 
     and thus $\nu$ also has multiplicity $1$ in the compass by Corollary~\ref{cor_compass_and_faces}\ref{item_minus_v_in_compass}.
  Thus to show the dimension part of the lemma, we have to prove, that the third or fourth type of points does not appear in the compass. 
  
  Note that, the remaining claims about $\widetilde{\Delta}$ and the compass at $e_1+e_2$ also follow, once we show that the third and fourth types do not appear.
  Indeed, by Corollary~\ref{cor_compass_and_fixpt}, the only other candidates for points in $\widetilde{\Delta}$ are the midpoints of the faces $\pm e_i$, which are excluded by 
  Lemma~\ref{lem_B_r-D_r-facets}\ref{item_Br_Dr_middle_point_is_not_fixed}, or $0\in \Delta$, which is excluded by Lemma~\ref{lem_zero_not_in_Delta_tilde}.
  The statement about compass is straightforward from the above calculations.
  
  Thus there is nothing left to prove for $r>4$. 
  First assume $r=4$.
  Say we want to show that $\nu = -\frac{1}{2} (e_1 + e_2 +e_3 + e_4)$ is not in the compass (the proof for the other cases is analogous).
  This point is the middle point of the facet $\delta$ with the supporting hyperplane $e_1^*+e_2^*+e_3^*+e_4^*=2$ as in 
    Lemma~\ref{lem_B_r-D_r-facets}\ref{item_Br_Dr_facets_projection_sum_of_ei}.
  In particular, $\dim Y_{\delta}=4$, and there are already $4$ distinct elements of the compass of the form 
     $e_3-e_1$, $e_3-e_2$, $e_4-e_1$, and $e_4-e_2$ that are parallel to this face.
  Thus by Lemma~\ref{lem_reduction_of_compass}, they are all in the compass of $y$ in $Y_{\delta}$,
     hence they are the only elements of this compass, and $\nu$ is not in either of the compasses $\cC(y, Y_{\delta}, (\CC^*)^3)$ or $\cC(y, X, H)$.

  Finally, we treat the most delicate case, $r=3$ and $G$ of type $B_3$.
  The
  following diagram presents the root polytope $\Delta=\Delta(B_3)$
  with $\bullet$ denoting the long roots (vertices) and $\circ$ denoting the short
  roots of $B_3$. The doubled line segments indicate the six elements of the
  compass at one of the extremal fixed points, we ignored the element
  pointing towards the center of the polytope.
$$
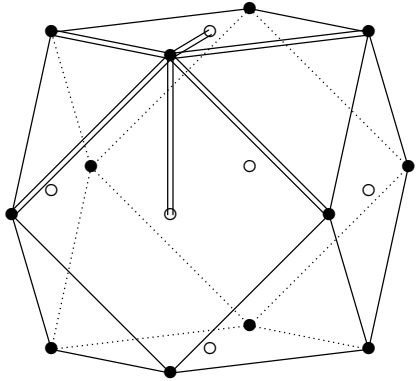
\begin{xy}<30pt,0pt>:
(2,2)*={}="a", (2,-2)*={}="b", (-2,-2)*={}="c", (-2,2)*={}="d",
(3,2.6)*={}="a1", (3,-1.4)*={}="b1", (-1,-1.4)*={}="c1", (-1,2.6)*={}="d1",
(0,0)*={\circ}="s1", (0.5,2.3)*={\circ}="s2", (1,0.6)*={\circ}="s3",
(0.5,-1.7)*={\circ}="s4", (2.5,0.3)*={\circ}="s5",
(-1.5,0.3)*={\circ}="s6",
(0,2)*={\bullet}="l1", (0,-2)*={\bullet}="l2", (2,0)*={\bullet}="l3",
(-2,0)*={\bullet}="l4", (1,2.6)*={\bullet}="l5",
(1,-1.4)*={\bullet}="l6", (3,0.6)*={\bullet}="l7",
(-1,0.6)*={\bullet}="l8", (-1.5,2.3)*={\bullet}="l9",
(-1.5,-1.7)*={\bullet}="l10", (2.5,2.3)*={\bullet}="l11",
(2.5,-1.7)*={\bullet}="l12",
"l1";"l3" **@{=}, "l3";"l2" **@{-}, "l2";"l4" **@{-}, "l4";"l1" **@{=},
"l5";"l7" **@{.}, "l7";"l6" **@{.}, "l6";"l8" **@{.}, "l8";"l5" **@{.},
"l9";"l1" **@{=}, "l9";"l4" **@{-}, "l9";"l5" **@{-}, "l9";"l8" **@{.},
"l11";"l1" **@{=}, "l11";"l5" **@{-}, "l11";"l3" **@{-}, "l11";"l7" **@{-},
"l12";"l3" **@{-}, "l12";"l7" **@{-}, "l12";"l2" **@{-}, "l12";"l6" **@{.},
"l10";"l2" **@{-}, "l10";"l6" **@{.}, "l10";"l4" **@{-}, "l10";"l8" **@{.},
"l1";"s1" **@{=}, "l1";"s2" **@{=}
\end{xy}
$$
Let $\nu =-\frac{1}{2}(e_1+e_2+e_3)$ and denote by $m$ the multiplicity of $\nu$ in the compass $\cC(y, X, H)$.
By the duality Lemma~\ref{lem_t-action-contact} the multiplicity of $\nu':=-\frac{1}{2}(e_1+e_2-e_3)$ is also equal to $m$ and we aim to show $m=0$.
We assume $m>0$ and argue to get a contradiction.

The next diagram is a cross section of the root polytope $\Delta$
(dotted line segments) by the plane $e_1^*-e_2^*=0$. Now $\otimes$
denotes the zero weight while $\star$'s denote the other weights in
the interior of $\Delta$.
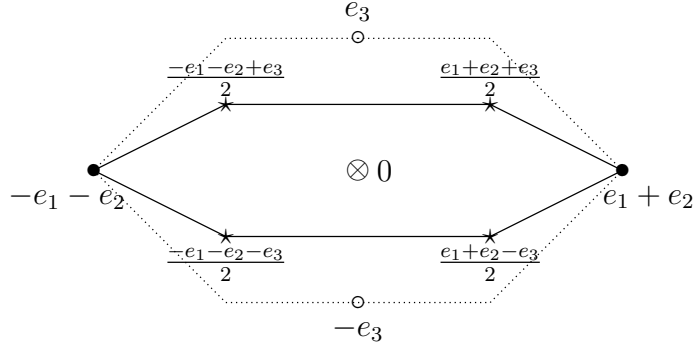
$$
\begin{xy}<50pt,0pt>:
(0,0)*={\otimes}="0", (0.2,0)*={0},
(0,1)*={\circ}="s1", (0,1.2)*={e_3}, (0,-1)*={\circ}="s2", (0,-1.2)*={-e_3},
(2,0)*={\bullet}="l1", (2.2,-0.2)*={e_1+e_2}, (-2,0)*={\bullet}="l2", (-2.2,-0.2)*={-e_1-e_2},
(1,1)*={}="a", (1,-1)*={}="b", (-1,-1)*={}="c", (-1,1)*={}="d",
"s1";"a" **@{.}, "a";"l1" **@{.}, "l1";"b" **@{.}, 
"b";"s2" **@{.}, "s2";"c" **@{.}, "c";"l2" **@{.},
"l2";"d" **@{.}, "d";"s1" **@{.},
(1,0.5)*={\star}="w1", (1,-0.5)*={\star}="w2", (-1,-0.5)*={\star}="w3", (-1,0.5)*={\star}="w4",
(1,0.7)*={\frac{e_1+e_2+e_3}{2}}, (1,-0.7)*={\frac{e_1+e_2-e_3}{2}}, 
(-1,-0.7)*={\frac{-e_1-e_2-e_3}{2}}, (-1,0.7)*={\frac{-e_1-e_2+e_3}{2}},
"w1";"l1" **@{-}, "l1";"w2" **@{-}, "w2";"w3" **@{-}, 
"w3";"l2" **@{-}, "l2";"w4" **@{-}, "w4";"w1" **@{-},
\end{xy}
$$
Consider the 1-dimensional subtorus $H' \subset H$ associated to the projection $e^*_1-e^*_2: M \rightarrow \ZZ\cdot\frac{1}{2}$.
Note that $0 \in \widetilde\Delta(X, L, H')$, as the fixed-point set $X^{H'}$ contains in particular $y$.
Let $Y_0 \subset X^{H'}$ be the component that contains $y$.
By Corollary~\ref{cor_contact-fixedpt-contact} the variety $Y_0$ is a contact
manifold with an action of the $2$-dimensional torus $H/H'$.
Moreover, by Lemma~\ref{lem_reduction_of_compass} the compass $\cC(y, Y_0, H/H')$ consists of $\nu$ and $\nu'$ both with multiplicity $m$, and $-e_1 -e_2$ with multiplicity $1$.
In particular, $\dim Y_0= 2m+1$.

We claim that the solid line segments on the above figure are the boundary of the fixed-point polytope $\Delta(Y_0, L|_{Y_0}, H/H')$. 
Indeed, $\Delta(Y_0, L|_{Y_0}, H/H')$ is a convex hull of some of the lattice points in $\widetilde{\Delta}(X, L, H) \cap \set{e^*_1-e^*_2 =0}$ 
   by Lemma~\ref{lem_reduction_of_action}\ref{item_reduction_of_action_restricting_to_fix_points_of_subgroup}.
Note however, that $\pm e_3$ does not belong to $\widetilde{\Delta}(X, L, H) $ by Lemma~\ref{lem_B_r-D_r-facets}\ref{item_Br_Dr_middle_point_is_not_fixed}.
Hence the $\Delta(Y_0, L|_{Y_0}, H/H')$ is a subset of the region bounded by the solid lines. 
Moreover, the compass at $y$  together with Corollary~\ref{cor_compass_and_fixpt} indicates that there must be fixed-point components $Z$ and $Z'$ corresponding to $\nu$ and $\nu'$, respectively.
Now move to the compass at $Z$ (or $Z'$).
By the duality in Lemma~\ref{lem_t-action-contact}, the next edge must also be included in $\Delta(Y_0, L|_{Y_0}, H/H')$, which shows the claim.

 From Lemma~\ref{lem_2-fixpts-comp} applied to
the action of 1-dimensional torus associated to the projection along
the respective edge we conclude that the extremal fixed-point
components in $Y_0$ associated to $\star$'s are of dimension $m-1$ (in
fact they are $\PP^{m-1}$'s). By Lemma~\ref{lem_t-action-contact}
applied to $Y_0$ the compass $\cC(Y_1,Y_0,H')$ of $H'$ on the extremal
fixed-point component $Y_1$ associated to, say $(e_1+e_2+e_3)/2$
contains $-(e_1+e_2+e_3)/2$ with multiplicity $m$ and $(e_1+e_2-e_3)/2$,
$-e_1-e_2$, both with multiplicity 1, coming from the decomposition
$$
   -\frac{1}{2}(e_1+e_2+e_3)=(-e_1-e_2)+\frac{1}{2}(e_1+e_2-e_3).
$$  

Hence by Lemma~\ref{lem_reduction_of_compass} the compass  $\cC(Y_1,X,H)$ contains $m+2$ elements 
listed above and, possibly, a few other elements from outside the plane $e_1^* - e_2^* = 0$.
Such an element $w$ must satisfy the following conditions:
\begin{itemize}
 \item the halfline $(e_1+e_2+e_3)/2  + w \cdot \RR_{>0}$ must intersect the set $\widetilde{\Delta}$ (Corollary~\ref{cor_compass_and_fixpt}),
 \item the dual element $-(e_1+e_2+e_3)/2-w$ must also be in the compass (Lemma~\ref{lem_t-action-contact}), hence satisfies the above.
\end{itemize}
There are $27$ lattice points in $\Delta$, out of which at most $21$ are in $\widetilde{\Delta}$ 
  (Lemma~\ref{lem_B_r-D_r-facets}\ref{item_Br_Dr_middle_point_is_not_fixed}).
Further removing the points on the plane $e_1^* - e_2^* = 0$ we are left with $14$ points.
Listing them and explicitly checking if they satisfy the second condition above, we are left only with $4$ candidates, 
  coming in two pairs:
$$
  \frac{1}{2}(e_1-e_2+e_3) \text{ and } (-e_1-e_3); \quad 
  \frac{1}{2}(-e_1+e_2+e_3) \text{ and } (-e_2-e_3).
$$
The first pair is contained in the plane $e_1^*-e_3^*=0$, and the second is contained in the plane $e_2^*-e_3^*=0$, 
  and they are both analogous to the pair of  $\frac{1}{2}(e_1+e_2-e_3)$ and $(-e_1-e_2)$ in the plane $e_1^*-e_2^*=0$.
In particular, the same proof as above shows that the multiplicity of these four vectors in the compass is equal to $1$.
Therefore $\dim X\leq \dim Y_1 + m+6 = 2m+5$, a contradiction.
\end{proof}

\begin{cor}\label{cor_contactB_r-D_r}
  Let $X$ be a contact Fano manifold which satisfies Assumptions~\ref{assumptions-simple} 
  with $G$ of type $B_r$ or $D_r$. Then $X$
  is isomorphic to the Grassmanian of lines on the quadric
  $\cQ^{2r-1}$ or $\cQ^{2r-2}$, respectively.
\end{cor}
  
\begin{proof}
  In Lemma~\ref{lem_B_r-D_r-big} we showed that the fixed-point locus $X^H$ consists of extremal fixed points only and the compass
  at each of these points is determined uniquely. 
  The respective Grassmannian of lines on the quadric of matching dimension 
    satisfies Assumptions~\ref{assumptions-simple}.
  Thus the result follows by Proposition~\ref{prop_fixpt+compass-determine-homogeneous}.
\end{proof}

Note that in the above corollary for the first time in Section~\ref{sec_contact_simple} we used the whole action of $G$, not only the action of $H$.

\subsection{Case of \texorpdfstring{$A_2$}{A2} and \texorpdfstring{$G_2$}{G2}}\label{sectA2}
In this subsection we deal with the action of groups of types $A_2$ (that
is $SL(3)$) or $G_2$ on a contact manifold $X$ of dimension $7$ or $9$. 
Below we illustrate the weights of $G=SL(3)$ in the root polytope
$\Delta(A_2)$. The same polytope and weigth lattice is obtained for the group $G$ of type
$G_2$. The vertices of $\Delta(A_2)$ are denoted by $\bullet$ and
labeled clockwise by $\alpha_i$, the center $\otimes$ is the zero
weight and the other weights inside $\Delta(A_2)$ are denoted by $\circ$
and labeled by $\beta_i$ also clockwise. 
The indices of $\alpha$'s and  $\beta$'s are considered modulo $6$. 
The fixed-point
components associated to $\alpha$'s are extremal. The fixed-point
components associated to $\beta$'s will be called inner and the ones
associated to the zero weight will be called central. The dotted line
segment illustrates the fixed-point locus of some $1$-parameter subgroup
of $H\subset G$ associated to the respective projection $M\rightarrow
\ZZ$ on which we have the action of the quotient $\CC^*$.

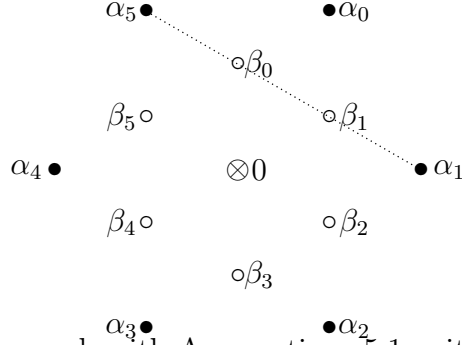
$$
\begin{xy}<40pt,0pt>:
(0,0)*={\otimes}="1", (0.2,0)*={0},
(0,1)*={\circ}="x^(-1)*y", (0.2,1)*={\beta_0},
(0.866,0.5)*={\circ}="y", (1.1,0.5)*={\beta_1}, 
(0.866,-0.5)*={\circ}="x", (1.1,-0.5)*={\beta_2},
(0,-1)*={\circ}="x*y^(-1)", (0.2,-1)*={\beta_3},
(-0.866,-0.5)*={\circ}="y^(-1)", (-1.1,-0.5)*={\beta_4},  
(-0.866,0.5)*={\circ}="x^(-1)", (-1.1,0.5)*={\beta_5},
(0.866,1.5)*={\bullet}="a0", (1.1,1.5)*={\alpha_0},
(1.732,0)*={\bullet}="a1", (2,0)*={\alpha_1},  
(0.866,-1.5)*={\bullet}="a2", (1.1,-1.5)*={\alpha_2},
(-0.866,-1.5)*={\bullet}="a3", (-1.1,-1.5)*={\alpha_3},
(-1.732,0)*={\bullet}="a4", (-2,0)*={\alpha_4},
(-0.866,1.5)*={\bullet}="a5", (-1.1,1.5)*={\alpha_5},
"a5";"a1" **@{.},
\end{xy}
$$

Throughout we work with Assumptions~\ref{assumptions-weaker}, with $G$ of type $A_2$ or $G_2$.
Denote by $y_{\alpha_i} \in X^H$ the unique extremal point with $\mu(y_{\alpha_i})=\alpha_i$.

In the proofs below we will exploit the downgrading mentioned above and its consequences.
Thus fix $i \in \setfromto{0}{5}$ and let $H'=H'_{i}$ be the subtorus corresponding 
  to the projection  $\pi = \pi_i \colon M\ra \ZZ$ 
  which maps all of $\alpha_{i-1},\beta_{i},\beta_{i+1},\alpha_{i+1}$ to $1\in \ZZ$.
For $i=0$, the dotted line in the figure passes through all these points.
By the last item of Lemma~\ref{lem_ABBdecomposition-extremal_cells}\ref{item_extremal_cells_case_dim_0} 
    there is a unique connected component $Y= Y_i \subset X^{H'}$  
    which contains the extremal fixed points $y_{\alpha_{i\pm 1}}$ and all inner components of $X^H$
    associated to $\beta_i$ and $\beta_{i+1}$. 
This manifold $Y$ admits the restricted action of a $1$-dimensional torus $H/H'$ with 
\[
  \Delta(Y, L|_{Y}, H/H') = \Gamma (Y, L|_Y, H/H') = [0,3]
\]
where $0$ represents $\alpha_{i-1}$ and $3$ represents $\alpha_{i+1}$.
Using Lemma~\ref{lem_reduction_of_compass} and the compass calculation at $y_{\alpha_{i-1}}$ in Lemma~\ref{lem_A2-components-compass} below
   we see that $\dim Y = 2$ if $\dim X = 7$ and $\dim Y = 3$ if $\dim X = 9$. 
    
\begin{lemma}\label{lem_A2-components-compass}
  Under Assumptions~\ref{assumptions-weaker} with the group $G$ of type $A_2$ or $G_2$, 
     suppose in addition that dimension of $X$ is either $7$ or $9$. 
  Then the following holds:
  \begin{enumerate}[leftmargin=*]
  \item\label{item_A2_compass_extremal}
    the extremal components of $X^H$ are points and the compass of the
    action of $H$ at a point associated to character $\alpha_i$
    consists of the following characters:
    \begin{itemize}[leftmargin=*]
    \item $\alpha_{i-1}-\alpha_i$, $\alpha_{i+1}-\alpha_i$, $-\alpha_i$, all three of 
      multiplicity $1$ and
    \item $\beta_i-\alpha_i$, $\beta_{i+1}-\alpha_i$, both with
      multiplicity $2$ if $\dim X =7$ or multiplicity $3$ if $\dim X =9$,
    \end{itemize}
  \item \label{item_A2_compass_inner}
    inner components of $X^H$ are points if $\dim X=7$ or points
    or curves if $\dim X=9$; the compass of the action of $H$ at a
    component associated to the weight $\beta_i$ consists of the
    following characters:
    \begin{itemize}[leftmargin=*]
    \item $\alpha_i-\beta_i$, $\alpha_{i-1}-\beta_i$, both with
      multiplicity $1$,
    \item $\beta_j-\beta_i$, where $j\ne i, i+3$, all with
      multiplicity $1$, except if  $\dim X=9$ and the fixed-point component
      is a point in which case $\beta_{i\pm 1}-\beta_i$ are of
      multiplicity $2$,
    \item $\beta_i$ with multiplicity $1$ if $\dim X=7$ or multiplicity
      $1$ or $2$ if $\dim X=9$, and dimension of the component is $0$ or $1$,
      respectively,
    \end{itemize}
  \item \label{item_A2_compass_central}
     there is no central component of $X^H$.
  \end{enumerate}
\end{lemma}
\begin{proof}
  Part \ref{item_A2_compass_extremal} is proven with the usual methods, as in previous subsections.
  Specifically, the multiplicities of $\alpha_{i-1}-\alpha_i$, $\alpha_{i+1}-\alpha_i$, $-\alpha_i$ follow from Corollary~\ref{cor_compass_and_faces}.
  The same corollary, part~\ref{item_compass_in_cones} show that $\beta_i-\alpha_i$ and  $\beta_{i+1}-\alpha_i$  are the only other candidates for elements of the compass.
  Their multiplicities must be equal by the duality of Lemma~\ref{lem_t-action-contact}.
  Thus the statement follows from the dimension calculation.
 
  For Part~\ref{item_A2_compass_inner}, we note that using the duality (Lemma~\ref{lem_t-action-contact}) and Corollary~\ref{cor_compass_and_fixpt}, 
    for every $\nu$ in the compass of any component $Z$ corresponding to $\beta_i$,
    we must have both halflines $\beta_i + \QQ_{> 0} \cdot \nu$ and $\beta_i + \QQ_{> 0} \cdot (- \nu - \beta_i)$ contain a lattice point of $\Delta$.
  This condition shows, that no other vector than those $7$ vectors listed in the statement can be in the compass. 
  It remains to determine the multiplicities. For
  \[
    \nu \in \set{\alpha_i-\beta_i, \alpha_{i-1}-\beta_i, \beta_{i+1}-\beta_i, \beta_{i-1}-\beta_i, \beta_{i+2}-\beta_i, \beta_{i-2}-\beta_i, -\beta_i}  
  \]
  denote by $b_{Z,\nu}$ the multiplicity of $\nu$ in the compass. We always have $b_{Z, - \beta_i} = \dim Z +1$.
  
  Consider the component $Y=Y_i$, the extremal fixed points  $y_{\alpha_j}$, the projection $\pi=\pi_i$, and the corresponding subtorus $H' = H'_i \subset H$ as introduced above.
  By Lemma~\ref{lem_reduction_of_compass} applied twice we have 
  \[
    \set{-2, -1^{n}, 1} = \pi(\cC(y_i, X, H)) = \cC(Y, X, H') = \pi(\cC(Z, X, H)),
  \]
  where $n=3$ for $\dim X =7$ and $n=4$ for $\dim X=9$.
  The only candidate for a member in the compass $\cC(Z, X, H)$, that is mapped to $1$ via $\pi$ is $\alpha_i-\beta_i$.
  Therefore, $b_{Z,\alpha_i-\beta_i}=1$, and the same proof with a different choice of $\pi = \pi_{i-1}$ shows that $b_{Z,\alpha_{i-1}-\beta_i}=1$.
  By the  duality (Lemma~\ref{lem_t-action-contact}), also  $\beta_{i+2}-\beta_i$ and $\beta_{i-2}-\beta_i$ have multiplicity $1$,
     while $\beta_{i+1}-\beta_i$ and $\beta_{i-1}-\beta_i$ have equal multiplicity.
  The only claim in \ref{item_A2_compass_inner} left to prove is that $b_{Z,\beta_{i+1}-\beta_i}\ge 1$, 
     which follows from Bia{\l}ynicki-Birula Decomposition (Theorem~\ref{thm_ABB-decomposition}) applied to the action of $H/\CC^*$ on $Y$.
  Indeed,  if $b_{Z,\beta_{i+1}-\beta_i} = 0$, then the respective BB-cells would be of dimension equal to $\dim Y$,   
     which is impossible, as $\pi(\beta_i)=1$ is not the smallest nor the largest in the $\Delta(X, L, H'_i) = [-2,2]$

  To prove \ref{item_A2_compass_central} we argue by contradiction and assume that we
  have a central component $Z$. 
  The compass of $Z$ in $X$ with respect to $H$ is symmetric: 
  if it contains a vector $\nu$, then it also contains $-\nu$
  with the same multiplicity (Lemma~\ref{lem_t-action-contact}).
  Consider the projection of $M$ along $\nu \in \cC(Z, X, H)$ and
    the corresponding subtorus $H' \subset H$.
  Pick the component $Y \subset X^{H'}$ which contains $Z$. 
  An extremal component $Y' \subset Y^{H/H'}$ of the fixed-point set of the quotient torus $H/H'$ action on $Y$ is  a component of $X^H$ by 
    Lemma~\ref{lem_reduction_of_action}\ref{item_reduction_of_action_restricting_to_fix_points_of_subgroup} and
    it is not a central component.
  Hence by Parts~\ref{item_A2_compass_extremal} and \ref{item_A2_compass_inner} we know that  $\dim Y'\leq 1$
  and the compass $\cC(Y' , Y , H/H')$ has at most $2$ elements (and $2$ elements are possible only if $\dim X=9$).
  We also know that $\dim Z$ and $\dim Y$ are odd by Corollary~\ref{cor_contact-fixedpt-contact}.
  
  We conclude that the only possibility is $\dim X =9$, $\dim Y=3$ and $\dim Z=1$ and $\nu=\beta_i$ for some $i$, with multiplicity $1$.
  Thus the compass $\cC(Z, X, H)$ consists of the $\beta_i$'s, each with multiplicity at most one which is
     impossible because it should consist of $8 = \dim X - \dim Z$ vectors, while there are only $6$ of the $\beta_i$'s.
\end{proof}

For a while now we restrict to the case $\dim X =7$. 

\begin{lemma}\label{lem_A2-on7fold}
  Under Assumptions~\ref{assumptions-weaker},
     suppose in addition $G$ is of type $A_2$ or $G_2$ and $\dim X =7$.
  Then there exists a unique inner fixed point $y_{\beta_i}$ for every $i\in \setfromto{0}{5}$.
\end{lemma}
\begin{proof}
  We consider $Y=Y_i$  with the action of $H/H'=H/H'_i \simeq \CC^*$ as introduced above.
  Lemma~\ref{lem_A2-components-compass} together with Lemma~\ref{lem_reduction_of_action}
    show that $Y$ is a surface and its polytope of fixed points $\Delta(Y, L|_{Y}, \CC^*)$,
    polytope of sections $\Gamma(Y,L|_{Y}, \CC^*)$ and the compasses $\cC(y, Y, \CC^*)$ are exactly
    as in Example~\ref{ex_Hirzebruch_surface}. 
  Thus there is exactly 1 fixed point in $Y^{H/H'}$ corresponding to $\beta_i$.
  But all fixed points corresponding to $\beta_i$ are contained in $Y$,
     so there is exactly 1 fixed point in $X^H$ corresponding to $\beta_i$, as claimed.
\end{proof}

The following lemma is a straightforward explicit verification. 

\begin{lemma}\label{lem_reductionB3-to-A2}
Consider the $7$-dimensional contact manifold $X=Gr(\PP^1, \cQ^5)$.
A three-dimensional torus acts on $X$, and the lattice and polytopes corresponding to type $B_3$ are described in Subsection~\ref{sectBrDr}.
Using the basis notation as in that subsection, consider a downgrading to a two dimensional torus $H$, corresponding to a projection $\ZZ^3 + (\frac{1}{2},\frac{1}{2},\frac{1}{2})\ZZ \to M$,
  with the kernel generated by $(\frac{1}{2},-\frac{1}{2},-\frac{1}{2})$.
Then the pair $(X, H)$ satisfies Assumptions~\ref{assumptions-weaker} with $G$ of type $G_2$.
In particular, its fixed points $X^H$ and compasses are described in Lemmas~\ref{lem_A2-components-compass} and~\ref{lem_A2-on7fold}.
\end{lemma}
We remark that the torus downgrading comes from the usual embedding of the $G_2$-group in $SO(7)$.

\begin{prop}\label{prop_A2_7folds_final}
  Under Assumptions~\ref{assumptions-weaker} with $\dim X =7$ and $G$ of type $A_2$ or $G_2$,
     we have $h^0(L) = \dim SO(7) = 21$.
  In particular, $X$ of dimension $7$ and $G$ or type $A_2$ or $G_2$ cannot satisfy Assumptions~\ref{assumptions-simple}.
\end{prop}

\begin{proof}
   Let $X' = Gr(\PP^1, \cQ^5)$. 
   By Lemma~\ref{lem_reductionB3-to-A2} both Lemmas~\ref{lem_A2-components-compass} and~\ref{lem_A2-on7fold} apply to $X'$ equally well as to $X$.
   There are only finitely many fixed points $X^H$ and $(X')^H$ 
     by Assumption~\ref{assumptions-weaker}\ref{item_assumption_extremal_points_are_isolated_pts} and Lemma~\ref{lem_A2-components-compass}.
   For each $X$ and $X'$ there are unique fixed points with $\mu(y)$ equal to 
   \begin{itemize}
    \item $\alpha_i$ (Lemma~\ref{lem_extremal=sink}), or
    \item $\beta_i$ (Lemma~\ref{lem_A2-on7fold}),
   \end{itemize}
   and there is no fixed point corresponding to $0 \in M$.
   Moreover, the compasses at these fixed points are uniquely determined by Lemma~\ref{lem_A2-components-compass}.
   Therefore by Proposition~\ref{prop_if_combinatorial_data_agrees_then_representations_agree} we must have $h^0(X,L) \simeq h^0(X',L') = 21$ as claimed.
\end{proof}

From now on we treat the case $\dim X=9$. 
\begin{lemma}\label{lem_9fold_no1dim}
  Under Assumptions~\ref{assumptions-weaker} with $\dim X =9$ and $G$ of type $A_2$ or $G_2$,
     the fixed-point locus $X^H$ has no component of dimension $1$.  
\end{lemma}
\begin{proof}
  Suppose by contradiction that $Z \subset X^H$ is a fixed curve.
  By Lemma~\ref{lem_A2-components-compass}  we must have $\mu(Z) = \beta_i$ for some $i$.
  We argue in several steps, using restriction and downgrading the action with respect to two different subtori $\CC^* \subset H$.

  \emph{Step 1, in which we show $Z \simeq \PP^1$ and $L|_{Z} \simeq \cO_{\PP^1}(2)$.}
  Here we consider the one parameter subgroup $H' \subset H$ which corresponds to the projection $M\to \ZZ$ with kernel generated by $\beta_i$. 
  Let $Y' \subset X^{H'}$ be the component containing $Z$.
  In particular, $Y'$ is a contact Fano manifold by Corollary~\ref{cor_contact-fixedpt-contact} with the contact line bundle $L|_{Y'}$.
  The fixed points and compass calculations in Lemma~\ref{lem_A2-components-compass} (and using Lemma~\ref{lem_reduction_of_compass}) 
     show $\dim Y'=3$ and there is no component of $(Y')^{H/H'}$ corresponding to $0 \in \ZZ\cdot\beta_i$.
  Moreover, there are exactly two components of $(Y')^{H/H'}$, corresponding to the extremal values $\beta_i$ and $-\beta_i$, 
     respectively (Lemma~\ref{lem_extremal=sink}).
  The component corresponding to $\beta_i$ is $Z$, and the other component is either a point or a curve 
     (Lemma~\ref{lem_A2-components-compass}\ref{item_A2_compass_inner}).
  Thus by the homology description in Theorem~\ref{thm_ABB-decomposition} the manifold $Y$ has topological Euler characteristic at most $4$.
  Any contact Fano $3$-fold must be isomorphic to a projective space $\PP^3$ 
     or to the projectivisation of a cotangent bundle $\PP (T^*\PP^2)$ 
     (this result was originally claimed by \cite{ye}, see \cite[\S1.3]{jabu_kapustka_kapustka_special_lines} for a detailed historical discussion).
  Given the restriction on the Euler characteristic, $Y' \simeq \PP^3$ and therefore $L|_{Y'} \simeq \cO_{\PP^3}(2)$.
  Since $Z$ is an eigenspace of a torus action, $Z$ is a line on $\PP^3$, in particular, $Z \simeq \PP^1$ and also $L|_{Z} \simeq \cO_{\PP^1}(2)$.
  
  \emph{Step 2, in which we review our notation for another $\CC^*$-action on $X$.}
  We swich our attention to $Y=Y_{i-1}$, the submanifold of dimension $3$, 
    which arises from the downgrading of the action as described in the introductory paragraphs of this subsection.
  The quotient $\CC^*$ acts on $Y$ and $\Delta(Y, L|_{Y}, \CC^*) = \Gamma(Y, L|_{Y}, \CC^*) = [0,3]$ 
    with $\beta_i$ corresponding to $2\in [0,3]$.
  The setting is similar to the proof of Example~\ref{ex_Hirzebruch_surface}:
  since $\Gamma(Y, L|_{Y}, \CC^*) = \Delta(Y,L|_{Y}, \CC^*)$ there are sections $\sigma_0$ and $\sigma_3$ of $L|_{Y}$
     which have weights $0$ and $3$ respectively.
  Let $D_0$ and $D_3$ be the corresponding divisors.
  By Lemma~\ref{lem_local_description_inv_divisor} the divisor $D_0$ contains all fixed points of $Y$ except $y_{\alpha_{i-1}}$ 
     (the fixed point corresponding to $0$),
     and analogously $D_3$ contains all of them except $y_{\alpha_{i+1}}$ corresponding to $3$.
   Consider the local defining equations $f_{0,z}$ and $f_{3,z}$ of $D_0$ and $D_3$ at a general point $z\in Z$,
      as in Lemma~\ref{lem_local_description_inv_divisor}.
   Denote the local coordinate ring by $\CC[s,t,u]$ with weights of $s$, $t$, and $u$ equal to $-1$, $0$, and $1$, respectively.
   
   \emph{Step 3, in which we analyze the divisor $D_0$.}
   The weight of $f_{0,z}$ is $2$, thus $f_{0,z}= g(t, su ) \cdot u^2$ for some power series $g$ in two variables.
   Thus $D_0$ has a component $U$ with multiplicity at least $2$ corresponding to $u=0$ near $z$.
   By Theorem~\ref{thm_ABB-decomposition}, $U$ must coincide with the closure of the cell corresponding to the fixed-point component $Z$,
      and its closure contains only one extra point, the extremal point $y_{\alpha_{i+1}}$.
   In particular, $U$ is smooth except perhaps at $y_{\alpha_{i+1}}$.
   Looking at the local equation of $D_0$ near $y_{\alpha_{i+1}}$ 
      we see that $U$ is given by a weight $-1$ equation in a coordinate ring with all negative weights 
      (in fact all the weights are equal to $-1$ by the usual compass calculation, but we will not use it). 
   Therefore $U$ is also smooth at $y_{\alpha_{i+1}}$, and by the previous arguments $U$ is smooth everywhere. 
   Moreover $y_{\alpha_{i+1}}$ has a neighbourhood in $U$ isomorphic with $\AAA^2$,
      and $U \setminus \AAA^2 = Z \simeq \PP^1$.
   Therefore $U\simeq\PP^2$ and $Z$ is a line on this $\PP^2$ and by Step 1 the line bundle $L|_{U}\simeq \cO_{\PP^2}(2)$.
   Denote by $K\subset U$ the line which is the closure of the orbit of $\CC^*$ containing $z$ and $y_{\alpha_{i+1}}$.
   
  \emph{Step 4, in which we analyze the divisor $D_3$ and conclude the proof.}
  We aim to calculate the intersection number $D_3.K$ (which should be equal to $2$ by Step 3, since $D_3$ is in the linear system of $L$, 
       and $L|_K \simeq \cO_{\PP^1}(2)$)
       and arrive at a contradiction.
   The local equation $f_{3,z}$ is not divisible by $u$ because $D_3$ does not contain $\alpha_{i+1}$.
   The weight of $f_{3,z}$ is $-1$ by Lemma~\ref{lem_local_description_inv_divisor}, 
       hence $f_{3,z} = s \cdot  h(t, su)$ for some power series $h$ in two variables, which is not divisible by the second variable.
   In particular, $D_3$ near $z$ is equal to the sum of $(s=0)$ and $(h(t, su)=0)$, where the second divisor does not vanish on all of $Z$.
   Since $z$ was chosen as a general point of $Z$, $h$ is invertible near $0$, and $D_3$ is equal to $(s=0)$ near $z$.
   In particular, $D_3.K=1$, a contradiction.
\end{proof}

\begin{lemma}\label{lem_A2_9fold_3_fixed_pts}
  Under Assumptions~\ref{assumptions-weaker} with $\dim X = 9$ and $G$ of type $A_2$ or $G_2$, 
     for every $i$ there are exactly $3$ fixed points of $X^H$ corresponding to $\beta_i$.
\end{lemma}

\begin{proof}
  Fix $i\in \setfromto{0}{5}$, and as before consider the smooth 3-fold $Y=Y_i$ with a $\CC^*$ action
  and containing all the fixed points associated to $\alpha_{i-1}$, $\beta_i$ and
  $\beta_{i+1}$, $\alpha_{i-1}$. 
  By Lemma~\ref{lem_9fold_no1dim} there are only $a$ fixed points associated to $\beta_i$ and 
    $b$ fixed points associated to $\beta_{i+1}$ for some (finite) integers $a, b \ge 0$.
  We consider the rational function in one variable $t$:
  \begin{align*}
     F(t)&=\sum_{y\in Y^H} \frac {t^{\mu(y)}}{\prod_{\nu\in \cC(y, X, H)} (1-t^{\nu})}\\
      &= \frac{1}{(1-t)^3} + a \cdot \frac{t}{(1-t)^2(1-t^{-1})}\\
      &\qquad + b \cdot \frac{t^2}{(1-t)(1-t^{-1})^2} + \frac{t^3}{(1-t^{-1})^3}\\
      &= \frac{t^6 - b t^4 + a t^2 - 1}{(t-1)^3}.
  \end{align*}
  The second equality follows from Lemmas~\ref{lem_A2-components-compass} and~\ref{lem_reduction_of_compass}.
  By Proposition~\ref{prop_character_is_a_Laurent_polynomial} the rational function $F$ must be a Laurent polynomial in $t$. 
  In particular, $t=1$ must be a root of the numerator, thus $a=b$.
  With $a=b$, dividing out both numerator and denominator by $(t-1)$ we get:
  \[
    F(t) = \frac{(t+1)(t^4 + (1-a) t^2 + 1)}{(t-1)^2}
  \]
  Again, $t=1$ must be a root of the numerator, thus $a=b=3$, and $F=(1+t)^3$.
  So there are exactly $3$ fixed points for each $\beta_i$.
\end{proof}

Now we check the model case also satisfies Assumptions~\ref{assumptions-weaker},
  analogously to Lemma~\ref{lem_reductionB3-to-A2}.

\begin{lemma}\label{lem_reductionD4-to-A2}
Consider the $9$-dimensional contact manifold $X=Gr(\PP^1, \cQ^6)$.
A four-dimensional torus acts on $X$, and the lattice and polytopes corresponding to type $D_4$ are described in Subsection~\ref{sectBrDr}.
Using the basis notation as in that subsection, consider a downgrading to a two dimensional torus $H$, corresponding to a projection 
\[ 
  \ZZ^4 + (\frac{1}{2},\frac{1}{2},\frac{1}{2},\frac{1}{2})\ZZ \to M,
\]
   with the kernel generated by $(\frac{1}{2},-\frac{1}{2},-\frac{1}{2},-\frac{1}{2})$ and $(0,0,0,1)$.
Then the pair $(X, H)$ satisfies Assumptions~\ref{assumptions-weaker} with $G$ of type $G_2$.
\end{lemma}
The proof is a straightforward and explicit calculation. 
The torus embedding is the restriction of standard embedding: $G_2 \hookrightarrow SO(7) \hookrightarrow SO(8)$.
As a conclusion, we show the analogue of Proposition~\ref{prop_A2_7folds_final} in dimension $9$.

\begin{prop}\label{prop_A2_9folds_final}
  Under Assumptions~\ref{assumptions-weaker} with $\dim X =9$ and $G$ of type $A_2$ or $G_2$,
     we have $\dim \Aut(X) =  h^0(L) = \dim SO(8) = 28$.
  In particular, $X$ of dimension $9$ and $G$ of type $A_2$ or $G_2$ cannot satisfy Assumptions~\ref{assumptions-simple}.
\end{prop}

\begin{proof}
   Let $X' = Gr(\PP^1, \cQ^6)$. 
   By Lemma~\ref{lem_reductionD4-to-A2} all Lemmas~\ref{lem_A2-components-compass}, \ref{lem_9fold_no1dim}, and~\ref{lem_A2_9fold_3_fixed_pts}
      apply both to $X$ and $X'$.
   There are only finitely many fixed points $X^H$ and $(X')^H$ 
     by Assumption~\ref{assumptions-weaker}\ref{item_assumption_extremal_points_are_isolated_pts} and 
     Lemmas~\ref{lem_A2-components-compass}, \ref{lem_9fold_no1dim}.
   For each $X$ and $X'$ there are unique fixed points with $\mu(y)$ equal to $\alpha_i$ (Lemma~\ref{lem_extremal=sink}), 
     and three fixed points with $\mu(y)$ equal to $\beta_i$ (Lemma~\ref{lem_A2_9fold_3_fixed_pts}),
     and there is no fixed point corresponding to $0 \in M$.
   Moreover, the compasses at these fixed points are uniquely determined by Lemma~\ref{lem_A2-components-compass}.
   Therefore by Proposition~\ref{prop_if_combinatorial_data_agrees_then_representations_agree}
     we must have $h^0(X,L) = h^0(X',L') = 28$ as claimed.
   This is in contradiction with Assumption~\ref{assumptions-simple}\ref{item_assumption_H0L}.
\end{proof}

This concludes the proof of Theorem~\ref{thm_contact_simple}, as we have already analyzed all the possible cases.

\section{Proofs and concluding remarks}\label{sec_proofs}

Having all the technical statements done, we conclude the article with gathering known results from the literature 
   and applying our lemmas to show main results of the article.
We provide a lower bound on the dimension of automorphism group of a low dimensional contact Fano manifold, following \cite{Salamon-Inventiones}.
We further discuss the main theorems listed in the introduction, providing the necessary cross-references and citations.
We also provide a brief overview of the consequences of twisor construction for positive quaternion-K\"ahler manifolds.

\subsection{Dimension of the automorphism group of a contact Fano manifold}\label{sec_dimension}

In this subsection  we provide a lower bound on the dimension of the group of automorphism of a contact Fano  manifold of dimension $7$ or $9$.
The argument is analogous to  \cite[Thm.~7.5]{Salamon-Inventiones}, whose statement is slighly weaker,
  as it only concerns quaternion-K\"ahler manifolds, and these correspond to contact Fano manifolds with K\"ahler-Einstein metric. 
In addition we rely on \cite[Cor.~1.2]{kebekus_lines2} and Bogomolov--Gieseker inequality. 
Although it is possible to perform the calculations by hand (as shown in \cite{Salamon-Inventiones}), 
  for brevity we refer to computer calculations carried out in {\tt magma} \cite{magma}.
  
\begin{thm}\label{thm_dim_of_Aut_X}
   Suppose $X$ is a contact Fano manifold.
   \begin{itemize}
    \item If $\dim X = 7$, then $\dim (\Aut X) \ge 5$.
    \item If $\dim X = 9$, then $\dim (\Aut X) \ge 8$.
   \end{itemize}
\end{thm}

Note that (unlike in Theorem~\ref{main_thm}) we do not need to assume that $\Aut(X)$ is reductive.

\begin{proof}
   The statement holds true if $X$ is one of $\PP^7$, $\PP(T^* \PP^4)$, $\PP^9$, or $\PP(T^*\PP^5)$, 
     so without loss of generality we may assume $\Pic X \simeq \ZZ\cdot L$.
   Therefore, $\dim (\Aut X) = h^0(X,L)$ by Lemma~\ref{lem_contact-canonical-linearization} and  the latter is equal to $\chi(X, L)$ by Kodaira vanishing since $L$ is ample and $X$ is Fano.
   Thus we have to show that $\chi(X, L)$ is at least $5$ or $8$ respectively.
   
   We calculate the Hilbert polynomial $p(m) = \chi(X, L^{m})$ using Hirze\-bruch-Riemann-Roch Theorem:
   \[
      p(m) = \int_{X} td(T X)ch(L^m).
   \]
   The Todd class is multiplicative $td(T X) = td(F)td(L)$ (by the short exact sequence $0\to F \to TX \to L\to 0$),
      and the Chern classes of $F$ satisfy the symmetry property arising from the isomorphism $F \simeq F^* \otimes L$.
   This determines the odd Chern classes of $F$ in terms of the even Chern classes and $c_1(L)$.
   Explicitly, if $\dim X = 7$, then
   \begin{align*}
      c_1(F) &= 3 c_1(L),\\
      c_3(F) &= 2 c_2(F) c_1(L) - 5c_1(L)^3,\\
      c_5(F) &=   c_4(F) c_1(L) - c_2(F) c_1(L)^3 + 3c_1(L)^5, 
   \intertext{and if $\dim X =9$, then}
      c_1(F) &= 4 c_1(L),\\
      c_3(F) &= 3 c_2(F) c_1(L) - 14c_1(L)^3,\\
      c_5(F) &= 2 c_4(F) c_1(L) - 5c_2(F) c_1(L)^3 + 28c_1(L)^5,\\
      c_7(F) &=   c_6(F) c_1(L) - c_4(F) c_1(L)^3 + 3c_2(F) c_1(L)^5 - 17c_1(L)^7. 
   \end{align*}
   Thus, the formula for the polynomial $p(m)$ is an explicit polynomial in $m$ and the Chern classes $c_1(L)$, $c_2(F)$, $c_4(F)$, $c_6(F)$ and (in the case $\dim X=9$) $c_8(F)$.
   There are extra conditions to impose on this polynomial, implied by Kodaira vanishing:
   $p(-2)= p(-1)=0$ and $p(0)=1$ (in the case $\dim X=7$, the condition $p(-2)=0$ is vacuous, as it follows from the symmetry, or Serre duality).
   Combining all of those identities and denoting $\mathfrak{d} = \deg (X,L) = c_1(L)^{\dim X}$
        (the self intersection of $L$) we get that:
   \begin{align*}
       \text{if $\dim X = 7$, then}\\
       c_1(TX)^2 c_1(L)^5 &= 16 \mathfrak{d}, & c_2(TX)c_1(L)^5& = 4 \mathfrak{d} + 12 p(1) - 48,\\
       \text{if $\dim X = 9$, then}\\
       c_1(TX)^2 c_1(L)^7 &= 25 \mathfrak{d}, & c_2(TX)c_1(L)^7& = 9 \mathfrak{d} + 24 p(1) - 168.
   \end{align*}
   The vector bundle $TX$ is stable by \cite[Cor.~1.2]{kebekus_lines2}, hence by Bogomolov--Gieseker inequality \cite[Thm~0.1]{Langer_Semistable_sheaves_in_pos_char} we have:
   \[
      (2 \dim X \cdot c_2(TX) - (\dim X -1)\cdot  c_1(TX)^2) \cdot c_1(L)^{\dim X -2}  \ge 0. 
   \]
   Substituting the explicit expressions and rearranging terms we get
   \[
      p(1)\ge 4 + \frac{5}{21} \mathfrak{d} \quad \text{ or }\quad  p(1)\ge 7 + \frac{19}{216} \mathfrak{d},
   \]
   respectively, which proves the claim of the theorem.
   
   The calculations were performed using the Chern classes package for {\tt magma} available at \cite{chern_classes_package}.
\end{proof}
\begin{rem}
   The Hilbert polynomials of $p(m)$ from the proof of Theorem~\ref{thm_dim_of_Aut_X} are equal to:
   \begin{align*}
      p(m) &= \mathfrak{d} \tbinom{m+7}{7} -  2\mathfrak{d} \tbinom{m+6}{6} + (\mathfrak{d}+ p(1) -4) \tbinom{m+5}{5}\\
           &- (p(1) - 4) \tbinom{m+4}{4} + \tbinom{m+3}{3}, \text{ or}\\
      p(m) &= \mathfrak{d} \tbinom{m+9}{9} -  \tfrac{5}{2} \mathfrak{d} \tbinom{m+8}{8} + (2\mathfrak{d}+ 2p(1) -14) \tbinom{m+7}{7} \\
           &- \left(\tfrac{1}{2}\mathfrak{d} + 3p(1) - 21\right) \tbinom{m+6}{6} + (p(1) - 5) \tbinom{m+5}{5} - \tbinom{m+4}{4},
   \end{align*}
   for $\dim X = 7$ or $9$, respectively.
   In particular, if $\dim X = 9$, then the degree $\mathfrak{d}$ is even.
   For $\dim X  \ge 11$ the information from Kodaira vanishing and Bogomolov--Gieseker 
     inequality is not enough to determine the bounds on $h^0(L)$ in terms of the degree of $(X,L)$.
   If $\dim X = 11$, then we have:
     \begin{align*}
      p(m) &= \mathfrak{d} \tbinom{m+11}{11} -  3\mathfrak{d} \tbinom{m+10}{10} + (3\mathfrak{d} - 8 p(1) + p(2) +27) \tbinom{m+9}{9}\\
           &-(\mathfrak{d} - 16p(1) + 2 p(2) + 54) \tbinom{m+8}{8}  -(7 p(1) - p(2) -21)\tbinom{m+7}{7}\\
           &-(p(1) - 6) \tbinom{m+6}{6} + \tbinom{m+5}{5}\text{, and}\\
      11 p(2) &+ 297 \ge 88 p(1) + 4 d.      
    \end{align*}
\end{rem}

We may strengthen the statement of the final item of Theorem~\ref{thm_contact_simple}.
\begin{cor}\label{cor_A1}
    Suppose Assumptions~\ref{assumptions-simple} are satisfied with $G$ of type $A_1$.
    Then $\dim X\ge 11$.
\end{cor}

\begin{proof}
    If $\dim X=5$ or $\dim X=3$ by theorems of \cite{druel} and \cite{ye, peternell_jabu_contact_3_folds},
       either Assumptions~\ref{assumptions-weaker}\ref{item_assumption_contact} and \ref{item_assumption_PicX_ZL} fail to simultaneously hold,
       or $G$ is of type $G_2$, a contradiction.
    Since $\dim G = 3$, if $\dim X= 7$ or $\dim X =9$, then we get a contradiction from Theorem~\ref{thm_dim_of_Aut_X}.
\end{proof}

\subsection{Classification results for contact manifolds}\label{subsec_classification_contact}

We gather the results from  earlier sections and the literature in order to prove 
  Theorems~\ref{thm_contact_dim_7_9} and \ref{main_thm} concerning 
  the classification of contact Fano manifolds with reductive automorphism group.
In addition, we assume that either the dimension of the manifold is small, 
  or rank of the automorphism group is sufficiently large.

\begin{proof}[Proof of Theorem~\ref{main_thm}]
Here we assume that $X$ is a contact Fano manifold of dimension $2n+1$ and $G=\Aut (X)$ is reductive of the rank at least $n-2$.
We claim that $X$ is a homogeneous space.
    If $\dim X \le 5$ then the theorems of \cite{druel} and \cite{ye, peternell_jabu_contact_3_folds} imply the statement.
If $\Pic X \ne \ZZ$, then the theorems of \cite{KPSW, Demailly} imply that $X\simeq \PP(T^*M)$ for some projective manifold $M$.
Further, since $X$ is Fano, $TM$ is an ample vector bundle and $M\simeq \PP^{n+1}$ by~\cite{mori_proj_mflds_with_ample_tangent},
   that is $X\simeq \PP (T^*\PP^{n+1})$, which is homogeneous, as claimed.
Now suppose $\Pic X =\ZZ$, but $L$ does not generate $\Pic(X)$, say $L \simeq (L')^a$ for a line bundle $L'$ and an integer $a >1$.
Then $-K_X \simeq L^{n+1} \simeq (L')^{an+a}$, thus $X$ is Fano of index at least $2n+2$.
By \cite{kobayashi_ochiai}, $X\simeq \PP^{2n+1}$, in particular $X$ is homogeneous.

Therefore, it remains to treat the case when $\Pic X = \ZZ \cdot L$.
By Proposition~\ref{prop_contact->Gsimple} the group $G$ is simple
   and we claim that Assumptions~\ref {assumptions-weaker} and~\ref{assumptions-simple} are satisfied.
Indeed, \ref{item_assumption_contact}, \ref{item_assumption_PicX_ZL}, \ref{item_assumption_torus_action}, 
  \ref{item_assumption_torus_simple_automorphisms} are immediate.
Items~\ref{item_assumption_extremal_points_are_isolated_pts} and \ref{item_assumption_H0L} hold by Lemma~\ref{lem_contact-Gamma=Delta}
  while \ref{item_assumption_Delta_weak} follows from  Lemma~\ref{lem_contact-canonical-linearization}
  and $\Delta(X,L,H)=\Gamma(X,L,H)$ from \ref{item_assumption_extremal_points_are_isolated_pts}.

Thus we can apply Theorem~\ref{thm_contact_simple}. 
Most of the cases of that theorem have too large dimension compared to the rank of $G$. 
The only case left which is not a homogeneous space is $G$ of type $A_1$, which is treated in Corollary~\ref{cor_A1}.
This shows that $X$ must be a homogeneous space, and the explicit list in Table~\ref{tab_list_of_contact_with_high_rank_of_torus}
  arises from comparing the rank of a simple group and the dimension of the corresponding adjoint variety.
\end{proof}
  
\begin{proof}[Proof of Theorem~\ref{thm_contact_dim_7_9}]
   Let $X$ be a contact Fano manifold of dimension at most $9$.
   If $\dim X = 3$ or $\dim X =5$ then  
     \cite{ye, peternell_jabu_contact_3_folds} and \cite{druel}
     solved the problem.
   If $\dim X = 7$ or $9$, then $\dim \Aut(X) \ge 5$ by Theorem~\ref{thm_dim_of_Aut_X}.
   Therefore, the rank of $\Aut(X)$ must be at least $2$, and we are in a position to apply Theorem~\ref{main_thm}, which shows the claim.   
\end{proof}

\subsection{Classification results for quaternion-K\"ahler manifolds}

In  this section we provide references 
  that connect our results about contact Fano manifolds to the claims about quaternion-K\"ahler manifolds.
For the details of the construction of the twistor space we follow \cite{Salamon-Inventiones}. 
See also the more precise references in the proofs below.

\begin{thm}\label{thm_twistor}
   The twistor construction determines a bijection between the following sets:
   \begin{itemize}
    \item The set of contact Fano manifolds admitting a K\"ahler-Einstein metric (up to an algebraic isomorphism),
    \item The set of positive quaternion-K\"ahler manifolds (up to a conformal diffeomorphism).
   \end{itemize}
\end{thm}
\begin{proof}
   Let $X$ be the twistor space \cite[\S2]{Salamon-Inventiones} of a quaternion-K\"ahler manifold $\cM$.
   Then $X$ is a complex, contact manifold \cite[Thms~4.1, 4.3]{Salamon-Inventiones}.
   If $\cM$ is in addition a positive quaternion-K\"ahler manifold, 
      then the twistor space is Fano, in particular, projective \cite[Cor.~6.2]{Salamon-Inventiones}, 
      and admits a K\"ahler-Eistein metric  \cite[Thm~6.1]{Salamon-Inventiones}.
   By  \cite[Thm~3.2]{LeBrunSalamon} the map from 
      the set of positive quaternion-K\"ahler manifolds up to a conformal diffeomorphisms 
      into the set of contact Fano manifolds admitting a K\"ahler-Einstein metric up to an algebraic isomorphism 
      is well defined and injective. 
   It is also surjective by \cite[Thm~A]{lebrun}, completing the proof.
\end{proof}

\begin{thm}\label{thm_dim_Aut_and_Isom}
   Except in the case $\cM \simeq \HHH\PP^n$ and $X\simeq \CC\PP^{2n+1}$,
     the real dimension of the isometry group of a positive quaternion K\"ahler manifold $\cM$ is equal 
     to complex dimension of the group of automorphism of its twistor space $X$. 
   In fact, the connected component of the latter group is the complexification of the connected component of the first one.
\end{thm}
\begin{proof}
   By \cite[Lem.~6.4, 6.5]{Salamon-Inventiones} the dimension of the isometry group is equal to $h^0(X,L)$,
      where $L$ is the contact line bundle as in Sections~\ref{sec_contact_general} and~\ref{sec_contact_simple}.
   (Note the traditional discrepancy between the notation in algebraic geometry papers such as this one, and quaternion-K\"ahler papers such as \cite{Salamon-Inventiones},
     where $L$ usually denotes the ``half'' of the contact line bundle, which is defined only locally, but whose square is equal to the contact line bundle.)
   Further, $h^0(X,L)$ is equal to the dimension of the group of contact automorphisms of $X$ 
     (see \cite[Prop.~1.1]{Beauville} or \cite[Cor.E.14(i)]{jabu_dr}).
   Finally, the claim in most of the cases follows from the uniqueness of the contact structure 
     \cite[Cor.~4.5]{kebekus_lines1} or \cite[Cor.E.14(ii)]{jabu_dr}).
   It remains to verify the claim explicitly in the only contact Fano case with $\Pic X\ne \ZZ$, i.e. $\PP(T^* \PP^{n+1})$, 
     which is straightforward.
   
   To see that $\rm{Isom}(\cM)$ complexifies to $\Aut(X)$,
     note that the naturality of the twistor construction gives an embedding 
      $\rm{Isom}(\cM) \hookrightarrow \Aut(X)$.
   Since $\Aut(X)$ is a complex Lie group and $\rm{Isom}(\cM)$ is compact, we get an immersive homomorphism $\rm{Isom}(\cM)_{\CC} \to \Aut(X)$.
   By the dimension count above, it is also surjective.
\end{proof}

\begin{thm}\label{thm_reductive_Aut_for_KE}
  Let $X$ be the twistor space of a positive quaternion-K\"ahler manifold $\cM$.
  Then $\Aut X$ is reductive.
\end{thm}
\begin{proof}
  By \cite[Thm~6.1]{Salamon-Inventiones} the manifold $X$ is Fano and admits a K\"ahler-Einstein metric.
  Thus by \cite[Thm. 1]{matsushima_Kahler_Einstein_have_reductive_group_of_automorphisms} the automorphism group is reductive.
\end{proof}

We are now ready to put together the building blocks and prove Theorems~\ref{thm_qK_dim_12_16} and~\ref{thm_qK_rank_of_isometries_group}.

\begin{proof}[Proof of Thereom~\ref{thm_qK_rank_of_isometries_group}]
  Let $\cM$ be a positive quaternion-K\"ahler manifold of dimension~$4n$, 
    and suppose the isometry group ${\rm Isom}(\cM)$ has rank at least $n-2$.
  Let $X$ be the twistor space of $\cM$.
  The manifold $X$ is contact Fano of dimension $2n+1$ and it admits a K\"ahler-Einstein metric by Theorem~\ref{thm_twistor}. 
  By Theorem~\ref{thm_reductive_Aut_for_KE} the group of automorphisms of $X$ is reductive.
  By Theorem~\ref{thm_dim_Aut_and_Isom}, the rank of $\Aut X$ is at least $n-2$, too.
  Therefore $X$ is an adjoint variety by Theorem~\ref{main_thm}.
  Since the adjoint varieties are exactly the twistor spaces of Wolf spaces, by Theorem~\ref{thm_twistor} 
    the manifold $\cM$ must be one of the Wolf spaces, as claimed.
\end{proof}

\begin{proof}[Proof of Thereom~\ref{thm_qK_dim_12_16}]
  Let $\cM$ be a positive quaternion-K\"ahler manifold of dimension~$12$ or $16$ and let $X$ be the twistor space of $\cM$.
  Since $X$ is a contact Fano manifold of dimension $7$ or $9$ (respectively) admitting a K\"ahler-Einstein metric by Theorem~\ref{thm_twistor}, 
    and by Theorem~\ref{thm_reductive_Aut_for_KE} the group of automorphisms of $X$ is reductive.
  Therefore $X$ is an adjoint variety by Theorem~\ref{thm_contact_dim_7_9}.
  Thus $\cM$ must be one of the Wolf spaces by Theorem~\ref{thm_twistor}. 
\end{proof}

\appendix

\section{Riemann-Roch and localization in \texorpdfstring{$K$}{K}-theory (by Andrzej Weber)}
\label{sect_appendix}

Suppose an algebraic torus $H\simeq (\CC^*)^r$ acts on an algebraic variety $X$.
An algebraic vector $H$-bundle over $X$ defines an element in  topological equivariant $K$-theory $K^H(X)$ defined by Segal \cite{Segal}, as well in the algebraic equivariant K-theory, see \cite{Nielsen, Thomason, EdidinGraham}. 
Which theory we use is irrelevant for us. For an accesible overview of algebraic equivariant $K$-theory we refer the reader to \cite[\S5]{ChrissGinzburg}. 
The Localization Theorem provides a formula
for the equivariant Euler characteristic $\chi^H$, which is  an element of the equivariant $K$-theory of a point $K^H(pt)$. This ring is just the
representation ring
$$R(H)\simeq\ZZ[t_1,t_1^{-1},t_2,t_2^{-1},\dots,t_r,t_r^{-1}]\,.$$
In the appendix we use the following notation:
fix an integral basis $$x_1,x_2,\dots,x_r\in M\subset \mathfrak{h}^*$$  of the dual of Lie algebra of $H$. Elements of $M$ are called weights.
Denote by
$t_1,t_2,\dots,t_r$  the corresponding characters.
That is, the image of $x_i$ under the identification of $M$ (with additive notation) 
   with $\Hom(H,\CC^*)$ (with multiplicative notation) is denoted by $t_i$.
For a weight $\nu=\sum_{i=1}^r a_i x_i$ the corresponding character is denoted by $t^\nu=\prod_{i=1}^r t_i^{a_i}$.

\begin{thm}\label{thm_RRekwgeneral} Assume that $X$ is a smooth compact complex $H$-manifold and $E$ is an $H$-vector bundle. Then the equivariant Euler characteristic is equal to
\begin{equation}\chi^H(X;E)=\sum_{i\in I}p_{i!}\left(\frac{E_{|Y_i}}{\lambda_{-1}(N^*_{Y_i})}\right)\in R(H)\,.\end{equation}
Here $X^H=\bigsqcup_{i\in I} Y_i$ is the decomposition into connected components, $\lambda_{-1}(N^*_{Y_i})$ is the equivariant Euler class of the normal 
bundle to $Y_i$. 
The map $p_i:Y_i\to pt$ is the constant map and $p_{i!}$ is the push forward in the equivariant $K$-theory.
\end{thm}

 We will explain what are the objects appearing in the theorem below, but first let us derive a corollary:

\begin{cor}\label{cor_RRekw} Suppose
$E=L$ is a line bundle and $\mu $ is a linearization of the action of $H$ on $L$.
Assume the fixed points $X^H$ consist of isolated points $y_1, y_2,\dotsc, y_k$ and curves $C_{k+1},C_{k+2}, \dots, C_\ell$.
The genus of $C_i$ is denoted by $g_i$.
 Suppose \begin{enumerate}[leftmargin=*]
\item for $i=1,2,\dots k$: $\mu_i =\mu(y_i)$ is the weight of the action of $H$ on $L|_{y_i}$
and $\nu_{i,j}$ are the weights of $H$ on $T^*_{y_i}X$;
\item for $i=k+1,k+2,\dots \ell$:
$\mu_i = \mu(C_i)$ is the weight of the action of $H$ on $L_{|C_i}$, $d_i$ --- the degree of $L_{|C_i}$;   the conormal bundle to $C_i$ decomposes into a direct sum
$N^*(C_i)=\bigoplus N^*(i,j)$, and suppose that the action of $H$ on the summand $ N^*(i,j)$ is of the weight $\nu_{i,j}$, let $\mathop{rk}(N^*(i,j))=r_{i,j}$ and let $\int_{C_i}c_1(N^*(i,j))=n_{i,j}$.\end{enumerate}
Then the equivariant Euler characteristic $\chi^H(X, L)$ is equal to
\[
\sum_{i=1}^k\frac {t^{\mu_i}}{\prod_{j} (1-t^{\nu_{i,j}})}+\sum_{i=k+1}^\ell \frac{t^{\mu_i}}{\prod_j(1 - t^{\nu_{i,j}})^{r_{i,j}}}\left(
1-g_i+d_i+\sum_j \tfrac{n_{i,j} }{ t^{-\nu_{i,j}}-1}\right).
\]
\end{cor}

The notation $\mu(y_i)$, $\mu(C_i)$ above is consistent with Sections~\ref{sec_torus_action}--\ref{sec_contact_simple}.
The characters $\nu_{i,j}$ form the compass of $y_i$ or $C_i$ in $X$ with respect to the action of $H$.

Corollary~\ref{cor_RRekw} is applied in Propositions~\ref{prop_if_combinatorial_data_agrees_then_representations_agree}, \ref{prop_character_is_a_Laurent_polynomial}
   and Example~\ref{ex_Hirzebruch_surface} in the situation when the fixed-point set is finite. 
We consider here the case when $X^H$ is of dimension $\leq 1$ having in mind further applications, 
   and also to point out which invariants of the fixed-point components are relevant to compute the space of global sections.

It is hard to trace the first appearance of Theorem~\ref{thm_RRekwgeneral}. Let us review what is present in literature. 
In a paper of Atiyah-Bott from the Wood-Hole conference, \cite{WoodsHole} the result in the case of isolated fixed points is given.
The formula was then repeated by Grothendieck, \cite[Cor.~6.12]{Gro} and Nielsen, \cite[\S4.7]{Nielsen}. 
For the case of nonisolated fixed points we quote Atiyah-Singer paper on equivariant Index Theorem. The relevant theorem  is called there
,,Holomorphic Lefschetz Theorem'' \cite[(4.6)]{AtiyahSinger}.
In \cite{BaumFultonQuart}, where  finite group actions are studied, such a kind of formula is called Lefschetz-Riemann-Roch. 
One can also apply widely known Atiyah-Bott-Berline-Vergne localization in equivariant cohomology, \cite{AtiyahBott, BerlineVergne}. Equivariant Riemann-Roch theorem allows to deduce localization in equivariant K-theory from localization in equivariant cohomology.  For localization theorems in algebraic equivariant K-theory see \cite[Theorem 2.1]{Thomason}, \cite[Theorem 5.11.7]{ChrissGinzburg}.

When  we identify the direct image of sheaves in homological algebra with the push-forward in homotopy theory then
the statement of Theorem~\ref{thm_RRekwgeneral} follows from \cite{tomDieck}, which is valid for any complex-oriented generalized cohomology theory.
Equally well we can deduce Theorem~\ref{thm_RRekwgeneral} from~\cite[Theorem 4.3(b)]{EdidinGraham}.

To show Theorem~\ref{thm_RRekwgeneral} in the form presented here, consider \emph{the equivariant Chern character} $ch^H$.
It maps equivariant $K$-theory to equivariant cohomology, in particular 
\begin{equation}\label{repChern}
   ch^H\colon K^H(pt)\longrightarrow \hat \HH{_H^*}(pt;\QQ)=\prod_{i\geq 0}\HH_H^i(pt;\QQ).
\end{equation}
For a weight $\nu=\sum_{i=1}^n \nu_i x_i$  the image $ch^H(t^\nu)=\exp(\nu)\in \hat \HH^*_H(pt;\QQ)$ will be denoted by $t^\nu$ again, in order to maintain a brevity of formulas.  
We use  \cite[\S4.10]{Nielsen}, which expresses $\chi^H(X,E)$ as:
\begin{equation}\label{LRR2}
   \chi^H(X,E)= \sum_{i\in I}\int_{Y_i} \frac{td(Y_i)\cdot ch^H(E_{|Y_i})}{ch^H\lambda_{-1}(N^*_{Y_i})},
\end{equation}
where $E\mapsto \lambda_{-1}(E)$ is the multiplicative transformation of $K$-theory such that for a line bundle $L$
$$\lambda_{-1}(L)=1-L\,.$$
 A part of the proof of localization theorem is  to show  that $\lambda_{-1}(N^*_{Y_i})$ is invertible after a suitable localization.
Note that the action of the torus $H$ on the base of the bundle $N^*_{Y_i}$ is trivial. 
If $L$ is an invariant line subbundle of $N^*_{Y_i}$ such that $H$ acts with weight $\nu$ on the fibers of $L$ (the weight is necessarily nonzero) then
\begin{multline}  ch^H\lambda_{-1}(L)^{-1}=\frac1{1-ch^H(L)}=\frac1{1-t^{\nu}\left(1+c_1(L)+\dots\right)}=\\=\frac1{1-t^{\nu}}\cdot\frac1{1-\frac{t^{\nu}}{1-t^{\nu}}c_1(L)+\dots}=\frac1{1-t^{\nu}}\left(1+\frac{c_1(L)}{t^{-\nu}-1}+\dots\right)\,.\label{wzor-rozwiniecie}\end{multline}
The integrals in (\ref{LRR2}) can be rewritten as
\begin{equation}\label{twistedintegral}
\int_{Y_i} td(Y_i)\cdot ch^H\left(\frac{E_{|Y_i}}{\lambda_{-1}(N^*_{Y_i})}\right)\,.\end{equation}
The expression $\frac{E_{|Y_i}}{\lambda_{-1}(N^*_{Y_i})}$ belongs to $ S^{-1}K^H(Y_i)\simeq S^{-1}R(H)\otimes K(Y_i)$, where $S$ is the multiplicative system generated by $1-t^\nu$, for all $\nu$ which are the weights of $H$-action on $N^*_{Y_i}$. 
Formally the nonequivariant Chern character is  a map
$$ch\colon K(Y_i)\to H^*(Y_i)\,.$$
The Chern character defined for the representation ring
$$ch^H\colon R[H]\to {\QQ}[[x_1,x_2,\dots x_r]]$$
$$t_i\mapsto exp(x_i)\,,$$
coincides with \eqref{repChern}.
We obtain an extension 
$$ch^H\colon
S^{-1}R(H)\otimes K(Y_i)\to (ch^H(S))^{-1}{\QQ}[[x_1,x_2,\dots x_r]]\otimes H^*(Y_i)\,.$$
Now we apply the  Hirzebruch-Riemann-Roch theorem for the extended Chern character appearing in \eqref{twistedintegral} and conclude that 
\begin{equation}\chi^H(X;E)=
\sum_{i\in I}p_{i\,!} \left(\frac{E_{|Y_i}}{\lambda_{-1}(N^*_{Y_i})}\right)\,.\end{equation}
This way we have identified the Nielsen formula \eqref{LRR2} with the statement of Theorem~\ref{thm_RRekwgeneral}.

Corollary \ref{cor_RRekw} follows directly from \eqref{LRR2}.

\begin{proof}[Proof of Corollary \ref{cor_RRekw}.]
 Assume that each fixed-point component of $X^H$ is either a point or a curve.
The contributions to $\chi^H(X;L)$ in (\ref{LRR2}) coming from the components $Y_i$  are the following:

If $Y_i=\{y_i\}$ is a point and $\nu_{i,j}$ are the cotangent weights (compass) then
 $$\lambda_{-1}(N^*_{y_i})^{-1}=\prod_{j=1}^{\dim X}\frac1{1-t^{\nu_{i,j}}}$$
because  $N^{*}_{y_i}=T^{*}_{y_i}X$, which is an equivariant bundle over a point.
Then
$$\int_{Y_i} \frac{td(Y_i)\cdot ch^H(L_{|Y_i})}{ch^H\lambda_{-1}(N^*_{Y_i})}=\frac {t^{\mu_i}}{\prod_{j=1}^{\dim X}(1-t^{\nu_{i,j}})}\,.$$

Suppose $Y_i=C_i$ is a curve. Let $[p]$ be the generator of $\HH^2(C_i)$. 
Let us decompose the conormal bundle
$N^*(C_i)=\bigoplus N^*(i,j)$, and suppose that the action of $H$ on the summand $ N^*(i,j)$ is of the weight $\nu_{i,j}$ and $\mathop{rk}(N^*(i,j))=r_{i,j}$.
Then by (\ref{wzor-rozwiniecie})
$$\lambda_{-1}(N^*(Y_i))^{-1}=\prod_j\frac{1 +
\tfrac{n_{i,j} }{t^{-\nu_{i,j}}-1}[p]}{(1 - t^{\nu_{i,j}})^{r_{i,j}}}\,,$$
where $c_1(N^*(i,j))=n_{i,j}[p]$ is the first Chern class.
The Todd class of the curve $C_i$ is equal to  $td(Y_i)=1+(1-g_i)[p]$. Assume that the action of $H$ on $L_{|C_i}$ is of weight $\mu_i$. Let $d_i$ be the degree of $L$ on $C_i$. The integral over $Y_i$
is equal to the homogenous component of degree one (with respect to $[p]$) of the product
$$\left(1+(1-g_i)[p]\right)\cdot t^{\mu_i}\left(1+d_i[p]\right)\cdot\prod_j\frac{1 +
\tfrac{n_{i,j}}{t^{-\nu_{i,j}}-1}[p]}{(1 - t^{\nu_{i,j}})^{r_{i,j}}}.$$
We obtain the contribution
$$\frac{t^{\mu_i}}{\prod_j(1 - t^{\nu_{i,j}})^{r_{i,j}}}\left(
1-g_i+d_i+\sum_j \tfrac{n_{i,j} }{ t^{-\nu_{i,j}}-1}\right)\,.$$
\end{proof}

\bibliography{torus_on_contact_all} 

\begin{thebibliography}{10}

\bibitem{Amann_Phd}
Manuel Amann.
\newblock {\em Positive Quaternion K{\"a}hler Manifolds}.
\newblock PhD thesis, Universit{\"a}t M{\"u}nster, 2009.
\newblock \url{https://d-nb.info/996176438/34}.

\bibitem{AtiyahSinger}
M.~F. Atiyah and I.~M. Singer.
\newblock The index of elliptic operators. {III}.
\newblock {\em Ann. of Math. (2)}, 87:546--604, 1968.

\bibitem{WoodsHole}
Michael~F. Atiyah and Raoul Bott.
\newblock On the {Woods} {Hole} fixed point theorem.
\newblock In {\em Lecture notes prepared in connection with the seminars held
  at the Summer Institute on Algebraic Geometry, Whitney Estate, Woods Hole,
  Massachusetts}, pages 57--60. AMS, 1964.
\newblock http://www.jmilne.org/math/Documents/woodshole.pdf.

\bibitem{AtiyahBott}
Michael~F. Atiyah and Raoul Bott.
\newblock {The moment map and equivariant cohomology.}
\newblock {\em {Topology}}, 23:1--28, 1984.

\bibitem{BaumFultonQuart}
Paul Baum, William Fulton, and George Quart.
\newblock Lefschetz-{R}iemann-{R}och for singular varieties.
\newblock {\em Acta Math.}, 143(3-4):193--211, 1979.

\bibitem{Beauville}
Arnaud Beauville.
\newblock Fano contact manifolds and nilpotent orbits.
\newblock {\em Comment. Math. Helv.}, 73(4):566--583, 1998.

\bibitem{berger}
Marcel Berger.
\newblock Sur les groupes d'holonomie homog\`ene des vari\'et\'es \`a connexion
  affine et des vari\'et\'es riemanniennes.
\newblock {\em Bull. Soc. Math. France}, 83:279--330, 1955.

\bibitem{BerlineVergne}
Nicole Berline and Mich\`ele Vergne.
\newblock Classes caract\'eristiques \'equivariantes. {F}ormule de localisation
  en cohomologie \'equivariante.
\newblock {\em C. R. Acad. Sci. Paris S\'er. I Math.}, 295(9):539--541, 1982.

\bibitem{Bialynicki}
A.~Bia{\l}ynicki-Birula.
\newblock Some theorems on actions of algebraic groups.
\newblock {\em Ann. of Math. (2)}, 98:480--497, 1973.

\bibitem{Bielawski}
Roger Bielawski.
\newblock Complete hyper-{K}\"ahler {$4n$}-manifolds with a local
  tri-{H}amiltonian {$\mathbf{R}^n$}-action.
\newblock {\em Math. Ann.}, 314(3):505--528, 1999.

\bibitem{magma}
Wieb Bosma, John Cannon, and Catherine Playoust.
\newblock The {M}agma algebra system. {I}. {T}he user language.
\newblock {\em J. Symbolic Comput.}, 24(3-4):235--265, 1997.
\newblock Computational algebra and number theory (London, 1993). Available for
  use on-line at \url{http://magma.maths.usyd.edu.au/calc/}.

\bibitem{Bourbaki}
N.~Bourbaki.
\newblock {\em \'El\'ements de math\'ematique. {F}asc. {XXXIV}. {C}hapitre
  {VI}: syst\`emes de racines}.
\newblock Actualit\'es Scientifiques et Industrielles, No. 1337. Hermann,
  Paris, 1968.

\bibitem{braden_macpherson_moment_graphs}
Tom Braden and Robert MacPherson.
\newblock From moment graphs to intersection cohomology.
\newblock {\em Math. Ann.}, 321(3):533--551, 2001.

\bibitem{brion_faces_of_moment_polytope}
Michel Brion.
\newblock On the general faces of the moment polytope.
\newblock {\em Internat. Math. Res. Notices}, (4):185--201, 1999.

\bibitem{Brion}
Michel Brion.
\newblock On linearization of line bundles.
\newblock {\em J. Math. Sci. Univ. Tokyo}, 22(1):113--147, 2015.

\bibitem{chern_classes_package}
Weronika Buczy{\'n}ska and Jaros{\l}aw Buczy{\'n}ski.
\newblock {C}hern classes package for {M}agma.
\newblock
  \url{https://www.mimuw.edu.pl/~jabu/CV/publications/chern_classes_package.m}.

\bibitem{jabu_dr}
Jaros{\l}aw Buczy{\'n}ski.
\newblock Algebraic {L}egendrian varieties.
\newblock {\em Dissertationes Math. (Rozprawy Mat.)}, 467:86, 2009.
\newblock PhD thesis, Institute of Mathematics, Warsaw University, 2008.

\bibitem{jabu_kapustka_kapustka_special_lines}
Jaros{\l}aw Bucz{y\'n}ski, Grzegorz Kapustka, and Micha{\l} Kapustka.
\newblock Special lines on contact manifolds.
\newblock arXiv: 1405.7792, 2014.

\bibitem{jabu_moreno_contact_survey}
Jaros{\l}aw Buczy{\'n}ski and Giovanni Moreno.
\newblock Complex contact manifolds, varieties of minimal rational tangents,
  and exterior differential systems.
\newblock In {\em Geometry of {L}agrangian {G}rassmannians and nonlinear
  {PDE}s}, volume 117 of {\em Banach Center Publ.}, pages 145--176. Polish
  Acad. Sci. Inst. Math., Warsaw, 2019.

\bibitem{peternell_jabu_contact_3_folds}
Jaros{\l}aw Buczy{\'n}ski and Thomas Peternell.
\newblock Contact {M}oishezon threefolds with second {B}etti number one.
\newblock {\em Arch. Math. (Basel)}, 98(5):427--431, 2012.

\bibitem{cannas_da_silva_intro_sympl_geom}
Ana Cannas~da Silva.
\newblock {\em Introduction to symplectic and {H}amiltonian geometry}.
\newblock Publica\c{c}\~oes Matem\'aticas do IMPA. [IMPA Mathematical
  Publications]. Instituto de Matem\'atica Pura e Aplicada (IMPA), Rio de
  Janeiro, 2003.

\bibitem{Carrell}
James~B. Carrell.
\newblock Torus actions and cohomology.
\newblock In {\em Algebraic quotients. {T}orus actions and cohomology. {T}he
  adjoint representation and the adjoint action}, volume 131 of {\em
  Encyclopaedia Math. Sci.}, pages 83--158. Springer, Berlin, 2002.

\bibitem{ChrissGinzburg}
Neil Chriss and Victor Ginzburg.
\newblock {\em Representation theory and complex geometry}.
\newblock Birkh\"auser Boston, Inc., Boston, MA, 1997.

\bibitem{cox_book}
David~A. Cox, John~B. Little, and Henry~K. Schenck.
\newblock {\em Toric varieties}, volume 124 of {\em Graduate Studies in
  Mathematics}.
\newblock American Mathematical Society, Providence, RI, 2011.

\bibitem{Demailly}
Jean-Pierre Demailly.
\newblock On the {F}robenius integrability of certain holomorphic {$p$}-forms.
\newblock In {\em Complex geometry ({G}\"ottingen, 2000)}, pages 93--98.
  Springer, Berlin, 2002.

\bibitem{druel}
St{\'e}phane Druel.
\newblock Structures de contact sur les vari\'et\'es alg\'ebriques de dimension
  5.
\newblock {\em C. R. Acad. Sci. Paris S\'er. I Math.}, 327(4):365--368, 1998.

\bibitem{EdidinGraham}
Dan Edidin and William Graham.
\newblock Algebraic cycles and completions of equivariant {$K$}-theory.
\newblock {\em Duke Math. J.}, 144(3):489--524, 2008.

\bibitem{Fang1}
Fuquan Fang.
\newblock Positive quaternionic {K}\"ahler manifolds and symmetry rank.
\newblock {\em J. Reine Angew. Math.}, 576:149--165, 2004.

\bibitem{Fang2}
Fuquan Fang.
\newblock Positive quaternionic {K}\"ahler manifolds and symmetry rank. {II}.
\newblock {\em Math. Res. Lett.}, 15(4):641--651, 2008.

\bibitem{fiebig_moment_graphs_in_repr_theory_and_geom}
Peter Fiebig.
\newblock Moment graphs in representation theory and geometry.
\newblock In {\em Schubert calculus---{O}saka 2012}, volume~71 of {\em Adv.
  Stud. Pure Math.}, pages 75--96. Math. Soc. Japan, [Tokyo], 2016.

\bibitem{fultonharris}
William Fulton and Joe Harris.
\newblock {\em Representation theory}, volume 129 of {\em Graduate Texts in
  Mathematics}.
\newblock Springer-Verlag, New York, 1991.
\newblock A first course, Readings in Mathematics.

\bibitem{Gro}
A.~{Grothendieck}.
\newblock Formule de {L}efschetz. (redige par {L}. {I}llusie).
\newblock In {\em Semin. Geom. algebr. Bois-Marie 1965-66, SGA 5, Lect. Notes
  Math. 589, Expose No.III.}, pages 73--137. Springer, 1977.

\bibitem{guillemin_holm_zara_GKM_description_equivariant_cohomology}
V.~Guillemin, T.~Holm, and C.~Zara.
\newblock A {GKM} description of the equivariant cohomology ring of a
  homogeneous space.
\newblock {\em J. Algebraic Combin.}, 23(1):21--41, 2006.

\bibitem{2herreras}
Hayde{\'e} Herrera and Rafael Herrera.
\newblock {$\hat A$}-genus on non-spin manifolds with {$S\sp 1$} actions and
  the classification of positive quaternion-{K}\"ahler 12-manifolds.
\newblock {\em J. Differential Geom.}, 61(3):341--364, 2002.

\bibitem{2herreras_erratum}
Hayde{\'e} Herrera and Rafael Herrera.
\newblock Erratum to ``{$\hat A$}-genus on non-spin manifolds with {$S^1$}
  actions and the classification of positive quaternion-{K}\"ahler
  12-manifolds'' [mr1979364].
\newblock {\em J. Differential Geom.}, 90(3):521, 2012.

\bibitem{hitchin}
N.~J. Hitchin.
\newblock K\"ahlerian twistor spaces.
\newblock {\em Proc. London Math. Soc. (3)}, 43(1):133--150, 1981.

\bibitem{Iskovskih_Fano_3fold_1}
V.~A. Iskovskih.
\newblock Fano threefolds. {I}.
\newblock {\em Izv. Akad. Nauk SSSR Ser. Mat.}, 41(3):516--562, 717, 1977.

\bibitem{kebekus_lines1}
Stefan Kebekus.
\newblock Lines on contact manifolds.
\newblock {\em J. Reine Angew. Math.}, 539:167--177, 2001.

\bibitem{kebekus_lines2}
Stefan Kebekus.
\newblock Lines on complex contact manifolds. {II}.
\newblock {\em Compos. Math.}, 141(1):227--252, 2005.

\bibitem{KPSW}
Stefan Kebekus, Thomas Peternell, Andrew~J. Sommese, and Jaros\l aw~A.
  Wi\'sniewski.
\newblock Projective contact manifolds.
\newblock {\em Invent. Math.}, 142(1):1--15, 2000.

\bibitem{Kim}
Jin~Hong Kim and Hee~Kwon Lee.
\newblock On positive quaternionic {K}\"ahler manifolds with {$b_4=1$}.
\newblock {\em Osaka J. Math.}, 49(3):551--561, 2012.

\bibitem{Kirwan}
Frances~Clare Kirwan.
\newblock {\em Cohomology of quotients in symplectic and algebraic geometry},
  volume~31 of {\em Mathematical Notes}.
\newblock Princeton University Press, Princeton, NJ, 1984.

\bibitem{Knop-etal}
Friedrich Knop, Hanspeter Kraft, Domingo Luna, and Thierry Vust.
\newblock Local properties of algebraic group actions.
\newblock In {\em Algebraische {T}ransformationsgruppen und
  {I}nvariantentheorie}, volume~13 of {\em DMV Sem.}, pages 63--75.
  Birkh\"auser, Basel, 1989.

\bibitem{kobayashi_ochiai}
Shoshichi Kobayashi and Takushiro Ochiai.
\newblock Characterizations of complex projective spaces and hyperquadrics.
\newblock {\em J. Math. Kyoto Univ.}, 13:31--47, 1973.

\bibitem{landsbergmanivel02}
Joseph~M. Landsberg and Laurent Manivel.
\newblock Construction and classification of complex simple {L}ie algebras via
  projective geometry.
\newblock {\em Selecta Math. (N.S.)}, 8(1):137--159, 2002.

\bibitem{Langer_Semistable_sheaves_in_pos_char}
Adrian Langer.
\newblock Semistable sheaves in positive characteristic.
\newblock {\em Ann. of Math. (2)}, 159(1):251--276, 2004.

\bibitem{lebrun}
Claude LeBrun.
\newblock Fano manifolds, contact structures, and quaternionic geometry.
\newblock {\em Internat. J. Math.}, 6(3):419--437, 1995.

\bibitem{LeBrunSalamon}
Claude LeBrun and Simon Salamon.
\newblock Strong rigidity of positive quaternion-{K}\"ahler manifolds.
\newblock {\em Invent. Math.}, 118(1):109--132, 1994.

\bibitem{Luna}
Domingo Luna.
\newblock Slices \'etales.
\newblock pages 81--105. Bull. Soc. Math. France, Paris, M\'emoire 33, 1973.

\bibitem{matsushima_Kahler_Einstein_have_reductive_group_of_automorphisms}
Yoz\^{o} Matsushima.
\newblock Sur la structure du groupe d'hom\'{e}omorphismes analytiques d'une
  certaine vari\'{e}t\'{e} k\"{a}hl\'{e}rienne.
\newblock {\em Nagoya Math. J.}, 11:145--150, 1957.

\bibitem{mori_proj_mflds_with_ample_tangent}
Shigefumi Mori.
\newblock Projective manifolds with ample tangent bundles.
\newblock {\em Ann. of Math. (2)}, 110(3):593--606, 1979.

\bibitem{Nielsen}
H.~Andreas Nielsen.
\newblock Diagonalizably linearized coherent sheaves.
\newblock {\em Bull. Soc. Math. France}, 102:85--97, 1974.

\bibitem{poon_salamon}
Y.~S. Poon and Simon~M. Salamon.
\newblock Quaternionic {K}\"ahler {$8$}-manifolds with positive scalar
  curvature.
\newblock {\em J. Differential Geom.}, 33(2):363--378, 1991.

\bibitem{Salamon-survey}
S.~M. Salamon.
\newblock Quaternion-{K}\"ahler geometry.
\newblock In {\em Surveys in differential geometry: essays on {E}instein
  manifolds}, volume~6 of {\em Surv. Differ. Geom.}, pages 83--121. Int. Press,
  Boston, MA, 1999.

\bibitem{Salamon-Inventiones}
Simon Salamon.
\newblock Quaternionic {K}\"ahler manifolds.
\newblock {\em Invent. Math.}, 67(1):143--171, 1982.

\bibitem{Segal}
Graeme Segal.
\newblock Equivariant {$K$}-theory.
\newblock {\em Inst. Hautes \'{E}tudes Sci. Publ. Math.}, (34):129--151, 1968.

\bibitem{Thomason}
R.~W. Thomason.
\newblock Une formule de {L}efschetz en {$K$}-th\'{e}orie \'{e}quivariante
  alg\'{e}brique.
\newblock {\em Duke Math. J.}, 68(3):447--462, 1992.

\bibitem{tomDieck}
Tammo tom Dieck.
\newblock Lokalisierung \"aquivarianter {K}ohomologie-{T}heorien.
\newblock {\em Math. Z.}, 121:253--262, 1971.

\bibitem{Wahl}
J.~M. Wahl.
\newblock A cohomological characterization of {${\bf P}^{n}$}.
\newblock {\em Invent. Math.}, 72(2):315--322, 1983.

\bibitem{Fano-largeindex}
Jaros\l aw~A. Wi\'sniewski.
\newblock On {F}ano manifolds of large index.
\newblock {\em Manuscripta Math.}, 70(2):145--152, 1991.

\bibitem{ye}
Yun-Gang Ye.
\newblock A note on complex projective threefolds admitting holomorphic contact
  structures.
\newblock {\em Invent. Math.}, 115(2):311--314, 1994.

\end{thebibliography}
\bibliographystyle{plain}
\end{document}